\documentclass[pdflatex,sn-mathphys-num]{sn-jnl}

\usepackage{graphicx}
\usepackage{multirow}
\usepackage{amsmath,amssymb,amsfonts}
\usepackage{amsthm}
\usepackage[title]{appendix}
\usepackage{xcolor}
\usepackage{textcomp}
\usepackage{manyfoot}
\usepackage{booktabs}
\usepackage{algorithm}
\usepackage{algorithmicx}
\usepackage{algpseudocode}
\usepackage{listings}
\usepackage{float}
\usepackage{microtype}
\usepackage{tikz}
\usepackage{subcaption}
\usepackage{marvosym}

\usetikzlibrary{shadows.blur, calc}
\usetikzlibrary{shadows,fadings}

\colorlet{colA}{red!75!orange}
\colorlet{colB}{teal!65!cyan}
\colorlet{colC}{blue!60!violet}

\numberwithin{equation}{section}

\theoremstyle{thmstyleone}
\newtheorem{theorem}{Theorem}[section]
\newtheorem{proposition}[theorem]{Proposition}
\newtheorem{lemma}[theorem]{Lemma}
\newtheorem{corollary}[theorem]{Corollary}

\theoremstyle{thmstyletwo}
\newtheorem{example}{Example}[section]
\newtheorem{remark}{Remark}[section]

\theoremstyle{thmstylethree}
\newtheorem{definition}{Definition}[section]

\begin{document}

\title[Structured Spectral Step-Sizes and Rebound-Aware Ordering]
{Structured Spectral Step-Sizes and Rebound-Aware Ordering for Gradient Methods}

% ================= Authors =================
\author[1]{\fnm{Yijia} \sur{Zhou}}
\email{yijiazhou2@outlook.com}

\author*[1]{\fnm{Ran} \sur{Gu}}
\email{rgu@nankai.edu.cn}

\author[2]{\fnm{Ya-xiang} \sur{Yuan}}
\email{yyx@lsec.cc.ac.cn}

\affil[1]{%
  \orgdiv{NITFID, School of Statistics and Data Science, LPMC, KLMDASR, LEBPS and AAIS},
  \orgname{Nankai University},
  \orgaddress{\city{Tianjin}, \country{China}}
}

\affil[2]{%
  \orgdiv{State Key Laboratory of Scientific and Engineering Computing},
  \orgname{Academy of Mathematics and Systems Science, Chinese Academy of Sciences},
  \orgaddress{\city{Beijing}, \country{China}}
}

\abstract{
The performance of gradient methods depends critically on step-size
selection. Spectral step-size design involves constructing candidates
from finite history and ordering them during iteration. We address
both through a finite-moment spectral framework. For candidate
construction, we relate Huang--Dai--Liu determinant pencils,
limited-memory steepest-descent Krylov--Ritz extraction,
pseudo-memory moment realization, and Gu--Du recovery of spectral
nodes and component weights. The framework distinguishes
generalized-eigenvalue recoverability, finite-window
implementability, and component-weight interpretability. Under
explicit positivity, rank-one, regularity, and exact-recovery
assumptions, these properties become compatible in a common
finite-moment regime, with the long, short, and mean LMSD
constructions as representatives. For ordering, we identify spectral
rebound: a step targeting a lower spectral scale may reactivate
previously reduced high-frequency components. Idealized one-target,
block, and recomputed memory-\(m\) models yield phase-count
recurrences motivating a high-frequency-first strategy with spectral
refreshes. We propose a rebound-aware spectral gradient method
(RASG) that ranks Ritz candidates by frequency and estimated
strength, and sets phase duration via a Ritz-uncertainty estimator.
Two variants, RASG-KT and RASG-LMSD, use a Kato--Temple-inspired
estimator and a long/short LMSD discrepancy. For strictly convex
quadratics, global \(R\)-linear convergence holds under an
admissibility safeguard. Experiments on ill-conditioned quadratic
problems and two large-scale convex problems illustrate the behavior.
}



\maketitle

\section{Introduction}
\label{sec:introduction}

\subsection{Motivation and problem setting}
\label{subsec:intro-motivation}

Gradient methods remain among the most widely used first-order methods for unconstrained optimization. They also serve as the algorithmic backbone of modern artificial intelligence and machine learning, where gradient and stochastic gradient methods are routinely used to train large-scale models. Their computational simplicity makes them particularly attractive when evaluating the gradient is substantially less expensive than forming or factorizing a Hessian. Their practical efficiency, however, depends strongly on the choice of step-size. An excessively conservative step-size produces slow progress, whereas an aggressive step-size may amplify selected error components and cause persistent oscillation. Step-size selection is therefore central to both the design and the analysis of gradient methods.

A natural model for studying spectral step-size behavior is the strictly convex quadratic problem
\begin{equation}
\label{eq:intro-quadratic}
    \min_{x\in\mathbb{R}^n} \quad f(x) = \frac{1}{2}x^\top Ax-b^\top x, \qquad A=A^\top\succ0.
\end{equation}
For the gradient iteration
$x^{k+1}=x^k-\alpha_k g^k$ with $g^k=Ax^k-b$,
the gradients satisfy
\begin{equation}
\label{eq:intro-gradient-dynamics}
    g^{k+1} = (I-\alpha_kA)g^k.
\end{equation}
Thus, for a quadratic objective, the effect of a step-size sequence is described exactly by scalar polynomials in \(A\) acting on the initial gradient.

The quadratic setting provides a transparent environment in which candidate construction, component suppression, and step-size ordering can be analyzed separately. The proposed step-size rules use gradient history and scalar inner products rather than an explicit factorization of \(A\). This makes analogous implementations possible for general smooth objectives through local secant information. Nevertheless, the fixed-matrix identities used in the theoretical development are exact only for quadratic objectives. The nonquadratic implementation studied in Section~\ref{sec:convex-case} is therefore treated as an empirical extension and is coupled with a globalization strategy.

\subsection{Spectral step-sizes and related work}
\label{subsec:intro-related-work}

Spectral step-size design involves two complementary tasks. The first is to construct curvature-sensitive candidates from a finite iteration history. The second is to determine the order and duration with which these candidates should be applied.

\medskip
\noindent\textit{Spectral candidate construction.}
Let $A=Q\Lambda Q^\top$, where $\Lambda=\operatorname{diag}(\lambda_1,\ldots,\lambda_n)$ and $0<\lambda_1\leq\cdots\leq\lambda_n$, and define $d^k=Q^\top g^k$. Equation~\eqref{eq:intro-gradient-dynamics} then separates into the component recurrences
\begin{equation}
\label{eq:intro-component-dynamics}
    d_i^{k+1} = (1-\alpha_k\lambda_i)d_i^k, \qquad i=1,\ldots,n.
\end{equation}
Consequently, the reciprocal step-size $\theta_k=\alpha_k^{-1}$ targets the spectral scale \(\lambda_i\) when \(\theta_k\) is close to \(\lambda_i\).

The classical steepest-descent step-size \cite{cauchy1847methode} and the two Barzilai--Borwein step-sizes \cite{barzilai1988two}, denoted by BB1 and BB2, admit weighted spectral representations. For a quadratic objective,
\begin{equation}
\label{eq:intro-sd-bb-weighted-average}
\begin{aligned}
    \bigl(\alpha_{\mathrm{SD}}^k\bigr)^{-1}
    &= \frac{\sum_{i=1}^n\lambda_i(d_i^k)^2}{\sum_{i=1}^n(d_i^k)^2},
    &
    \bigl(\alpha_{\mathrm{BB1}}^k\bigr)^{-1}
    &= \frac{\sum_{i=1}^n\lambda_i(d_i^{k-1})^2}{\sum_{i=1}^n(d_i^{k-1})^2},
    &
    \bigl(\alpha_{\mathrm{BB2}}^k\bigr)^{-1}
    &= \frac{\sum_{i=1}^n\lambda_i^2(d_i^{k-1})^2}{\sum_{i=1}^n\lambda_i(d_i^{k-1})^2}.
\end{aligned}
\end{equation}
These formulas show that finite gradient histories probe the spectral measure induced by the current gradient. They also explain why dominant components exert greater influence on the selected reciprocal step-size.

This interpretation extends to several finite-dimensional spectral constructions. Yuan's step-size~\cite{yuan2006new} gives finite termination for two-dimensional strictly convex quadratics, and related constructions have been developed for three-dimensional quadratic models~\cite{xie2026new}. Huang et al.~\cite{huang2021equipping} derived a two-dimensional termination step-size through determinant identities involving filtered gradients. More generally, a number of classical and recent step-size rules admit spectral interpretations; see, for example, \cite{di2018steplength,dai2019family,sun2020new}.

Fletcher's limited-memory steepest-descent method \cite{fletcher2012limited} extracts Ritz values from a short gradient history and reuses their reciprocals as step-sizes. Its convergence properties were studied in~\cite{curtis2018}. The modified LMSD method of Gu and Du~\cite{gu2021modified} supplements Ritz extraction with component estimates and an admissibility criterion, allowing more accurate candidates to be selected from the reduced model. These methods motivate the construction question addressed in this paper: which structural conditions make a finite-history spectral candidate recoverable, implementable with limited storage, and interpretable in terms of reduced spectral components?

\medskip
\noindent\textit{Step-size sequencing.}
The order in which spectral step-sizes are applied is equally important. Delayed steepest-descent step-sizes were studied in \cite{friedlander1998gradient}. Dai and Yuan \cite{dai2003alternate} proposed an alternate minimization gradient method and established \(Q\)-superlinear convergence in two dimensions. Long-term spectral behavior of monotone gradient methods was analyzed in~\cite{dai2005analysis}. Adaptive BB strategies, asymptotically optimal sequences, and periodic short-step corrections were developed in \cite{frassoldati2008new,zhigljavsky2013asymptotically,Gonzaga_2015}.

The steepest-descent-with-correction method \cite{de2014efficient} combines steepest-descent iterations with a constant step-size derived from the Yuan formula, thereby promoting selective reduction of gradient components. The asymptotic analysis of Dai and Fletcher~\cite{dai2005asymptotic}, later refined by Gu and Du~\cite{gu2021modified}, showed that consecutive steps near a common spectral scale can enlarge the separation between dominant and subdominant gradient components. Under their asymptotic assumptions, this mechanism yields superlinear behavior in a dimension range depending on the memory size and the selected history index. Related spectral effects have been used to break asymptotic zigzagging and to accelerate BB-type methods \cite{huang2022asymptotic,huang2022acceleration}. Cyclic and adaptive cyclic strategies also appear in \cite{zhang2024cyclic,xie2025adaptive}.

These studies demonstrate that candidate quality alone does not determine performance. Even an accurate candidate may become ineffective if it is applied in an unfavorable order or reused after the reduced spectral information has become stale.

\medskip
\noindent\textit{Rebound and the ordering problem.}
Equation~\eqref{eq:intro-component-dynamics} reveals an additional difficulty. A step targeting a relatively small eigenvalue may amplify components associated with substantially larger eigenvalues. Hence a high-frequency component that has already been reduced can become dominant again after a lower-frequency phase. We refer to this reactivation as spectral rebound.

Rebound turns candidate ordering into a scheduling problem among interacting spectral scales. High-frequency components should generally be controlled before lower-frequency components are accurately targeted, but they may also need to be revisited after lower-frequency steps. A practical finite-memory method must therefore decide both which reduced spectral candidate to select and how long the corresponding step-size should be reused before refreshing the model.

The preceding discussion leads to two questions:
\begin{enumerate}
    \item Which spectral candidates can be reliably constructed from a bounded gradient history?
    \item How should these candidates be ranked, reused, and refreshed in the presence of spectral rebound?
\end{enumerate}

\subsection{Contributions}
\label{subsec:intro-contributions}

The contributions of this paper follow the distinction between candidate construction and candidate ordering.

\begin{enumerate}
    \item \textit{A finite-moment framework for spectral candidate construction.} We introduce four related structural families: \(\mathcal{HDL}_{\mathrm{gen}}\), \(\mathcal{HDL}_{\mathrm{gen}}^{(1)}\), \(\mathcal{HDL}_{\mathrm{nar}}\), and \(\mathcal{GD}_m\). The general HDL family describes recovery from reduced determinant pencils. Its rank-one subclass identifies separable spectral symbols that produce Gram-type reduced representations. The narrow HDL family records realizability from a bounded scalar-moment window, while the generalized Gu--Du family records recovery of spectral nodes together with component weights. We distinguish the roles of rank-one structure, finite-window realization, positivity, and exact node recovery. Under the assumptions stated in Section~\ref{subsec:gudu-compatibility-rigidity}, these properties become compatible in a positive rank-one finite-moment regime. The long, short, and mean LMSD constructions provide representative examples.

    \item \textit{Rebound-aware spectral ordering.} We identify spectral rebound as a mechanism by which a lower-frequency phase may reactivate previously reduced high-frequency components. Under explicit idealized transition rules, we derive phase-count recurrences for one-target, block, and recomputed memory-\(m\) models. These recurrences characterize the scheduling structure of the ideal models and motivate repeated high-frequency control with regular spectral refreshes. They are not asserted as iteration-complexity bounds for the practical finite-precision algorithm.

    \item \textit{A pseudo-memory rebound-aware spectral gradient method.} We propose the rebound-aware spectral gradient method, abbreviated as RASG. This method ranks Ritz candidates using spectral frequency and estimated component strength, filters the candidates through an admissibility safeguard, and adaptively determines the duration of each phase using an estimate of Ritz uncertainty. The RASG-KT and RASG-LMSD variants use a Kato--Temple-inspired estimator and a long/short LMSD discrepancy, respectively. For strictly convex quadratic objectives, we establish global \(R\)-linear convergence despite the bounded information lag introduced by multi-step reuse. Numerical experiments compare the proposed variants with representative spectral gradient methods on ill-conditioned quadratic problems. Two large-scale convex experiments are also reported as an empirical nonquadratic extension.
\end{enumerate}

\subsection{Organization}
\label{subsec:intro-organization}

The remainder of the paper is organized as follows. Section~\ref{sec:general-framework} develops the finite-moment construction framework and examines the relations among determinant pencils, rank-one symbols, pseudo-memory realizations, and Gu--Du moment recovery. Section~\ref{sec:hanoi-principle} introduces the rebound phenomenon and analyzes idealized one-target, block, and memory-\(m\) phase models. Section~\ref{sec:algorithm} presents the RASG method and its two uncertainty estimators. Section~\ref{sec:convergence_analysis} establishes global \(R\)-linear convergence for strictly convex quadratic objectives. Section~\ref{sec:numerical-experiments} reports the numerical experiments, and Section~\ref{sec:conclusion} concludes the paper.

\section{Finite-Moment Construction of Spectral Step-Sizes}
\label{sec:general-framework}

We first address the candidate-construction question raised in Section~\ref{sec:introduction}: which spectral scales can be recovered from a finite gradient history, and which additional structures make the resulting candidates useful in a low-memory algorithm?

Three properties are relevant. First, the candidates should arise from a reduced generalized eigenvalue problem, so that they possess a well-defined projection-based spectral interpretation. Second, the reduced matrices should be recoverable from a bounded scalar-moment window, allowing the full gradient history to be replaced by a pseudo-memory representation. Third, the same moment information should provide nonnegative component weights that quantify how strongly the corresponding reduced spectral scales are represented.

The structural classes introduced below separate these requirements. The general HDL class records recovery through a determinant pencil. The rank-one HDL subclass imposes separability of the effective spectral symbol and leads to Gram-type reduced matrices. The narrow HDL class records recovery from a bounded shifted moment window. Finally, the generalized Gu--Du class describes the recovery of discrete spectral nodes and weights from selected moment equations.

These classes are not equivalent in general. Rank-one separability, finite-window realization, and moment-node recovery impose distinct requirements, as demonstrated by the examples below. Compatibility is obtained only under the additional hypotheses stated explicitly in Section~\ref{subsec:gudu-compatibility-rigidity}. In particular, the long, short, and mean LMSD constructions belong to the positive rank-one finite-moment regime and provide the candidate pools used by the adaptive method in Section~\ref{sec:algorithm}.

\subsection{General and rank-one HDL classes}
\label{subsec:huang-pencils}

We begin with the two-dimensional step-size construction of Huang et al.~\cite{huang2021equipping}. Its scalar cross-product relation can be expressed as a generalized eigenvalue problem. Separating the determinant-zero property from the stronger rank-one factorization leads to two structural families.

Let \(\Omega=[\lambda_1,\lambda_n]\subset(0,\infty)\) denote a compact interval containing the spectrum of \(A\). A \emph{rational spectral filter} is a real rational function with no pole on \(\Omega\). Such a filter is well defined when evaluated at \(A\) through the spectral functional calculus.

Let \(\nu_1(k)\) and \(\nu_2(k)\) be gradient indices satisfying \(\nu_p(k)\leq k-1\), \(p=1,2\), and let \(\psi_1,\psi_2,\psi_3,\psi_4\) be rational spectral filters on \(\Omega\) satisfying \(\psi_1(x)\psi_2(x)=\psi_3(x)\psi_4(x)\), \(x\in\Omega\). The scalar identity considered by Huang et al.~\cite{huang2021equipping} has the form
\begin{equation}
\label{eq:huang-cross-product}
\begin{aligned}
    & \bigl(g^{\nu_1(k)}\bigr)^\top \psi_1(A)(I-\alpha A)g^k \;
    \bigl(g^{\nu_2(k)}\bigr)^\top \psi_2(A)(I-\alpha A)g^k \\
    &\qquad = \bigl(g^{\nu_1(k)}\bigr)^\top \psi_3(A)(I-\alpha A)g^k \;
    \bigl(g^{\nu_2(k)}\bigr)^\top \psi_4(A)(I-\alpha A)g^k.
\end{aligned}
\end{equation}

Define
\begin{equation}
\label{eq:huang-cross-matrix}
    \mathcal{B}_k(\alpha) =
    \begin{pmatrix}
    \bigl(g^{\nu_1(k)}\bigr)^\top \psi_1(A)(I-\alpha A)g^k &
    \bigl(g^{\nu_1(k)}\bigr)^\top \psi_3(A)(I-\alpha A)g^k \\[4pt]
    \bigl(g^{\nu_2(k)}\bigr)^\top \psi_4(A)(I-\alpha A)g^k &
    \bigl(g^{\nu_2(k)}\bigr)^\top \psi_2(A)(I-\alpha A)g^k
    \end{pmatrix}.
\end{equation}
Then \eqref{eq:huang-cross-product} is equivalent to $\det\mathcal{B}_k(\alpha)=0$.

Introduce the matrix-valued symbol
\begin{equation}
\label{eq:huang-filter-symbol}
    \Psi(x) = \begin{pmatrix} \psi_1(x) & \psi_3(x)\\ \psi_4(x) & \psi_2(x) \end{pmatrix}.
\end{equation}
The identity \(\psi_1(x)\psi_2(x)=\psi_3(x)\psi_4(x)\) is equivalent to \(\det\Psi(x)=0\), \(x\in\Omega\).

Writing \(\mathcal{B}_k(\alpha)=M_0^k-\alpha M_1^k\), the entries of the reduced matrices are
\begin{equation}
\label{eq:huang-matrix-entries}
    \bigl(M_0^k\bigr)_{pq} = \bigl(g^{\nu_p(k)}\bigr)^\top \psi_{pq}(A)g^k, \qquad
    \bigl(M_1^k\bigr)_{pq} = \bigl(g^{\nu_p(k)}\bigr)^\top \psi_{pq}(A)Ag^k.
\end{equation}
With the reciprocal variable \(\theta=\alpha^{-1}\),
the determinant condition becomes
$\det\bigl(M_1^k-\theta M_0^k\bigr)=0$.

When the symbol admits a rank-one factorization \(\Psi(x)=u(x)v(x)^\top\), the entrywise construction in \eqref{eq:huang-matrix-entries} reduces to the filtered-block factorization introduced below.

\begin{definition}[General HDL class \(\mathcal{HDL}_{\mathrm{gen}}\)]
\label{def:Hgen}
Let \(m\geq2\), and let \(\Psi(x)=[\psi_{pq}(x)]_{p,q=1}^m\) be an \(m\times m\) matrix-valued rational spectral filter whose entries have no poles on \(\Omega\). A step-size construction belongs to \(\mathcal{HDL}_{\mathrm{gen}}\) if it is specified by gradient indices \(\nu_p(k)\leq k-1\), \(p=1,\ldots,m\), and satisfies \(\det\Psi(x)=0\), \(x\in\Omega\). The associated reduced matrices are defined entrywise by
\[
    (M_0^k)_{pq}=\bigl(g^{\nu_p(k)}\bigr)^\top \psi_{pq}(A)g^k
\]
and
\[
    (M_1^k)_{pq}=\bigl(g^{\nu_p(k)}\bigr)^\top \psi_{pq}(A)Ag^k.
\]
Whenever the pencil is regular, the reciprocal step-size candidates are its finite generalized eigenvalues:
\begin{equation}
\label{eq:Hgen-det-pencil}
    \det\!\left(M_1^k-\theta M_0^k\right)=0.
\end{equation}
\end{definition}

The pointwise singularity of \(\Psi\) generalizes the structural identity behind \eqref{eq:huang-cross-product}. It does not, by itself, provide a Gram-type factorization of the reduced matrices. The following subclass imposes the stronger separability property.

\begin{definition}[Rank-one HDL class \(\mathcal{HDL}_{\mathrm{gen}}^{(1)}\)]
\label{def:Hgen-rank-one}
A construction in \(\mathcal{HDL}_{\mathrm{gen}}\) belongs to \(\mathcal{HDL}_{\mathrm{gen}}^{(1)}\) if its spectral-filter symbol admits the factorization \(\Psi(x)=u(x)v(x)^\top\), \(x\in\Omega\), where
\[
    u(x)=(u_1(x),\ldots,u_m(x))^\top
\]
and
\[
    v(x)=(v_1(x),\ldots,v_m(x))^\top
\]
are vector-valued rational spectral filters with no poles on \(\Omega\).

Define
\[
    L^k=\begin{bmatrix}
        u_1(A)g^{\nu_1(k)} & \cdots & u_m(A)g^{\nu_m(k)}
    \end{bmatrix}
\]
and
\[
    R^k=\begin{bmatrix}
        v_1(A)g^k & \cdots & v_m(A)g^k
    \end{bmatrix}.
\]
Then
\begin{equation}
\label{eq:rank-one-gram}
    M_0^k=(L^k)^\top R^k,\qquad M_1^k=(L^k)^\top AR^k.
\end{equation}
The reciprocal candidates are the finite generalized eigenvalues of \eqref{eq:Hgen-det-pencil}, provided that the pencil is regular.
\end{definition}

Every rank-one symbol is pointwise singular when \(m\geq2\). Therefore, \(\mathcal{HDL}_{\mathrm{gen}}^{(1)}\subseteq\mathcal{HDL}_{\mathrm{gen}}\). The inclusion is generally strict because a singular matrix-valued symbol may have rank larger than one.

\begin{remark}[Effective symbols]
\label{rem:Hgen1-gradient-structure}
The rank-one property is most naturally interpreted after all gradient-history factors have been absorbed into an effective spectral symbol. For quadratic objectives, every gradient in a fixed history window is a polynomial in \(A\) applied to a common anchor gradient. Consequently, the reduced matrices can be written as integrals of an effective matrix-valued symbol against the spectral measure induced by that anchor. Rank-one structure of this effective symbol, rather than pointwise singularity alone, yields the Gram-type spectral recovery used later in the finite-moment analysis.
\end{remark}

\begin{example}[The original two-dimensional HDL construction]
\label{ex:special-Huang}
The two-dimensional construction of Huang et al.~\cite{huang2021equipping} is a representative member of \(\mathcal{HDL}_{\mathrm{gen}}^{(1)}\). Let \(\psi_4(x)=1\), \(\psi_1(x)=(1-\alpha_{k-2}x)^{-1}\),
\[
    \psi_2(x)=(1-\alpha_{k-1}x)^{-1},\qquad \psi_3(x)=\psi_1(x)\psi_2(x),
\]
and assume that the denominators are nonzero on \(\Omega\). Set \(\nu_1(k)=k-2\), \(\nu_2(k)=k-1\). A rank-one factorization is
\[
    u(x)=\begin{pmatrix}(1-\alpha_{k-2}x)^{-1}\\1\end{pmatrix},\qquad
    v(x)=\begin{pmatrix}1\\ (1-\alpha_{k-1}x)^{-1}\end{pmatrix}.
\]
The corresponding filtered blocks are
\[
    L^k=\begin{bmatrix}
        (I-\alpha_{k-2}A)^{-1}g^{k-2} & g^{k-1}
    \end{bmatrix}
\]
and
\[
    R^k=\begin{bmatrix}
        g^k & (I-\alpha_{k-1}A)^{-1}g^k
    \end{bmatrix}.
\]
The reciprocal candidates are therefore the generalized eigenvalues of the pencil
\begin{equation}
\label{eq:special-Huang-pencil}
    \det\!\left((L^k)^\top AR^k-\theta(L^k)^\top R^k\right)=0.
\end{equation}
\end{example}

\subsection{Pseudo-memory realization and the narrow HDL class}
\label{subsec:moment-hankel-narrow}

We next consider the finite-window realization of LMSD-type Ritz values. The central observation is that the reduced LMSD pencil can be reconstructed from a short scalar-moment sequence, so the full gradient block need not be retained.

Let \(G^k=[g^{k-m},\ldots,g^{k-1}]\) and consider the LMSD pencil
\begin{equation}
\label{eq:lmsd-ritz-pencil}
    (G^k)^\top AG^kz=\theta(G^k)^\top G^kz.
\end{equation}
Throughout, we use a fixed memory length \(m\) and adopt the standard full-rank condition
\begin{equation}
\label{eq:no-breakdown}
    \dim\operatorname{span}\{g^{k-m},\ldots,g^{k-1}\}=m,
    \qquad\text{equivalently}\qquad H_0^k\succ0,
\end{equation}
at every spectral refresh. This is the usual no-breakdown setting in limited-memory spectral methods; see, e.g.,~\cite{fletcher2012limited,gu2021modified}. 

Let \(Q_k\) be an orthonormal basis of \(\mathcal V^k=\operatorname{span}\{g^{k-m},\ldots,g^{k-1}\}\), and define
\begin{equation}
\label{eq:projected-matrix}
    T_k=Q_k^\top AQ_k.
\end{equation}
The eigenvalues \(\theta_1^k\leq\cdots\leq\theta_m^k\) of \(T_k\) are the Ritz values of \(A\) on the current gradient-history subspace.

For a quadratic objective, choose the anchor \(w=g^{k-m}\). The gradient recurrence gives \(g^{k-m+i}=p_i^k(A)w\), \(i=0,\ldots,m\), where \(p_0^k(x)=1\) and \(p_{i+1}^k(x)=(1-\alpha_{k-m+i}x)p_i^k(x)\). Under \eqref{eq:no-breakdown}, $\operatorname{span}\{g^{k-m},\ldots,g^{k-1}\}=\operatorname{span}\{w,Aw,\ldots,A^{m-1}w\}$.

This polynomial-history representation provides the connection with the rank-one HDL framework. After the factors \(p_i^k\) are absorbed into the effective symbol, the reduced Gram matrices are generated by products of polynomial filters applied to a common anchor gradient. The LMSD construction therefore has an effective rank-one finite-moment representation, even though the identity \(L^k=R^k=G^k\) is not, by itself, an instance of the filtered-block parameterization in Definition~\ref{def:Hgen-rank-one}.

Define the scalar moments \(h_j^k=w^\top A^jw\), \(j\geq0\). In the monomial Krylov basis, the reduced Gram matrices are the Hankel matrices
\begin{equation}
\label{eq:moment-Hankel-pair}
    H_0^k=[h_{r+s}^k]_{r,s=0}^{m-1},\qquad H_1^k=[h_{r+s+1}^k]_{r,s=0}^{m-1}.
\end{equation}
Under \eqref{eq:no-breakdown}, the generalized eigenvalues of \(H_1^ky=\theta H_0^ky\) coincide with the Ritz values of \eqref{eq:lmsd-ritz-pencil}.

\begin{remark}[Yuan and NY step-sizes as special cases]
\label{rem:yuan-special-case}
For \(m=2\), the two-dimensional Krylov--Ritz construction recovers the spectral quantities underlying the Yuan step-size \cite{yuan2006new}. For \(m=3\), the corresponding projected problem gives the three-dimensional spectral quantities used by the NY construction~\cite{xie2026new}. Explicit formulas are collected in Appendix~\ref{subsec:app-yuan}. For a general memory size, the generalized eigenvalue problem can be solved in \(O(m^3)\) arithmetic.
\end{remark}

Forming \(G^k\) explicitly requires storing \(m\) vectors in \(\mathbb R^n\). The moment representation avoids this storage. Every inner product of two gradients in the current window satisfies \((g^{k-m+i})^\top g^{k-m+j}=w^\top p_i^k(A)p_j^k(A)w\) and is therefore a linear combination of the moments \(h_0^k,\ldots,h_{i+j}^k\). The scalar data
\begin{equation}
\label{eq:pseudo-memory-qk}
    q^k=\bigl(
        \|g^{k-m}\|_2^2,\,
        (g^{k-m+1})^\top g^{k-m},\,
        \|g^{k-m+1}\|_2^2,\,
        \ldots,\,
        (g^k)^\top g^{k-1},\,
        \|g^k\|_2^2
    \bigr)^\top
\end{equation}
satisfy a lower-triangular system \(q^k=S^kh^k\), where \(h^k=(h_0^k,\ldots,h_{2m}^k)^\top\) and \(S^k\) depends only on the step-sizes in the current history window. Provided that these step-sizes are nonzero, the diagonal entries of \(S^k\) are nonzero and the moments can be recovered by forward substitution. The explicit recurrence is given in Appendix~\ref{app:pseudo-memory}.

Once the moment window is available, forming the reduced Hankel matrices costs \(O(m^2)\), while solving the reduced generalized eigenvalue problem costs \(O(m^3)\). The vector storage consists of the current and previous gradients, together with \(O(m)\) scalar moments and step-sizes. Thus, for fixed small \(m\), the dominant storage is \(2n+O(m)\), rather than \(mn+O(m)\).

This motivates the following finite-window class.

\begin{definition}[Narrow HDL class \(\mathcal{HDL}_{\mathrm{nar}}\)]
\label{def:Hnar}
A construction in \(\mathcal{HDL}_{\mathrm{gen}}\) belongs to \(\mathcal{HDL}_{\mathrm{nar}}\) if there exist an anchor vector \(w\), moments \(h_j=w^\top A^jw\), \(j=0,\ldots,2m\), and a fixed linear map \(\mathcal T:\mathbb R^{2m}\longrightarrow\mathbb R^{m\times m}\) such that
$M_0=\mathcal T(h_0,\ldots,h_{2m-1})$ and $M_1=\mathcal T(h_1,\ldots,h_{2m})$.
The map \(\mathcal T\) may depend on the construction parameters but not on the numerical values of the moments.
\end{definition}

By definition, \(\mathcal{HDL}_{\mathrm{nar}}\subseteq\mathcal{HDL}_{\mathrm{gen}}\). The narrow condition rules out representations that require negative moments, poles on the spectral interval, or an unbounded moment tail.

\subsection{Generalized Gu--Du moment recovery}
\label{subsec:gudu-framework}

The Hankel realization recovers spectral candidates through a generalized eigenvalue problem. The Gu--Du viewpoint uses the same moment information to fit a discrete spectral measure whose nodes represent reduced spectral scales and whose weights quantify the strength of the corresponding components.

Consider an \(m\)-atomic model \(h_j=\sum_{i=1}^m t_i a_i^j\), \(j=0,\ldots,2m-1\), with distinct positive nodes \(a_i>0\) and positive weights \(t_i>0\). For the standard Hankel construction, the recovered nodes coincide with the Ritz values of the corresponding Krylov projection. Let \(Q_k\in\mathbb{R}^{n\times m}\) be an orthonormal basis of the current Krylov subspace, and let \(z_i\) be a unit eigenvector of \(T_k=Q_k^\top A Q_k\) associated with \(\theta_i^k\). The corresponding Ritz vector is \(v_i^k=Q_kz_i\). For the anchor \(w=g^{k-m}\), the atomic weight is
\begin{equation}
\label{eq:component-energy}
    t_i^k=\bigl(w^\top v_i^k\bigr)^2 = \left((g^{k-m})^\top Q_kz_i\right)^2,\qquad i=1,\ldots,m.
\end{equation}
Thus, \(t_i^k\) is the squared coefficient of the anchor gradient along the \(i\)-th Ritz direction. Under \eqref{eq:no-breakdown}, the moment model contains exactly \(m\) positive nodes and weights.

The pseudo-memory realization also provides the additional moment \(h_{2m}\). Define the residual vector \(E(a,t;h)=(E_0(a,t;h),\ldots,E_{2m}(a,t;h))^\top\) with
$E_j(a,t;h)=h_j-\sum_{i=1}^m t_ia_i^j$ for \(j=0,\ldots,2m\).
A selector chooses \(2m\) independent linear combinations of these \(2m+1\) residual equations.

\begin{definition}[Generalized Gu--Du class \(\mathcal{GD}_m\)]
\label{def:generalized-gudu-construction}
Let \(W\in\mathbb R^{(2m+1)\times2m}\) satisfy \(\operatorname{rank}W=2m\). A construction belongs to \(\mathcal{GD}_m\) on a specified admissible moment set if it recovers a positive distinct-node solution \(a_1,\ldots,a_m\), \(t_1,\ldots,t_m\) of
\begin{equation}
\label{eq:gudu-selected-equations}
    W^\top E(a,t;h)=0.
\end{equation}
The solution is understood up to a common permutation of the node--weight pairs. Whenever the recovered nodes are positive, the associated spectral step-sizes are \(\alpha_i=a_i^{-1}\), \(i=1,\ldots,m\).
\end{definition}

The rank condition on \(W\) guarantees independence of the selected residual equations, but it does not by itself guarantee existence, uniqueness, positivity, or distinctness of the recovered nodes. Whenever these properties are needed, they are stated explicitly.

\begin{example}[Long, short, and mean LMSD selectors]
\label{ex:gudu-long-short-mean}
The standard LMSD variants correspond to the following selections.
\begin{itemize}
    \item \textbf{Long LMSD.} Use \(E_0=E_1=\cdots=E_{2m-1}=0\). The resulting nodes are the generalized eigenvalues of the Hankel pair generated by \(h_0,\ldots,h_{2m-1}\).
    \item \textbf{Short LMSD.} Use \(E_1=E_2=\cdots=E_{2m}=0\). The shifted moment sequence has effective weights \(\widetilde t_i=t_ia_i\).
    \item \textbf{Mean LMSD.} Use \(E_j+E_{j+1}=0\), \(j=0,\ldots,2m-1\). The averaged moments \(\mu_j=h_j+h_{j+1}\) have effective weights \(\widetilde t_i=t_i(1+a_i)\).
\end{itemize}
\end{example}

\begin{remark}[Relation with BB1 and BB2]
\label{rmk:LMSDlongshort}
Long and short LMSD extend the spectral structures of BB1 and BB2, respectively; see~\cite{fletcher2012limited}. The long construction uses the moment ratio associated with \(\frac{g^\top Ag}{g^\top g}\), whereas the short construction shifts the weights by one power of the spectral node and is associated with \(\frac{g^\top A^2g}{g^\top Ag}\).
\end{remark}

\subsection{Rank-one structure and finite-window realization}
\label{subsec:huang-side-basic-relations}

Rank-one separability and finite-window realization encode distinct properties of a spectral construction. The former concerns the algebraic structure of the effective symbol and supports exact node recovery through a reduced pencil. The latter concerns the dependence of the reduced matrices on a bounded shifted moment window and supports pseudo-memory implementation.

For the examples in this subsection, all gradient-history polynomials are absorbed into an effective matrix-valued symbol \(P(x)\). Given a moment sequence \(h=(h_j)_{j\geq0}\) and a matrix-valued polynomial \(P(x)=\sum_{j=0}^{d}P^{(j)}x^j\), define $\Lambda_h(P):=\sum_{j=0}^{d}P^{(j)}h_j$. The corresponding reduced matrices are \(M_0=\Lambda_h(P)\), \(M_1=\Lambda_h(xP)\). A rational effective symbol is interpreted through the corresponding spectral integral whenever it has no pole on \(\Omega\).

The inclusions \(\mathcal{HDL}_{\mathrm{gen}}^{(1)}\subseteq\mathcal{HDL}_{\mathrm{gen}}\), \(\mathcal{HDL}_{\mathrm{nar}}\subseteq\mathcal{HDL}_{\mathrm{gen}}\) follow from the definitions. Neither rank-one structure nor finite-window realization implies the other.

\begin{example}[Independence of rank-one structure and finite-window realization]
\label{ex:rankone-narrow-independence}
First, consider the effective polynomial symbol
\[
    P(x)=\begin{pmatrix}
        1 & x & x\\
        x & x^2 & x^2\\
        x & x^2 & x^3
    \end{pmatrix}.
\]
Its second row satisfies \(R_2(x)=xR_1(x)\), and therefore \(\det P(x)\equiv0\). For generic \(x>0\), however, the matrix has rank two. It is therefore pointwise singular but not rank one.

The induced moment matrix is
\[
    M_0=\begin{pmatrix}
        h_0 & h_1 & h_1\\
        h_1 & h_2 & h_2\\
        h_1 & h_2 & h_3
    \end{pmatrix},
\]
while \(M_1\) is obtained by shifting every moment index by one. Hence both matrices are generated by a fixed shifted linear map from a finite moment window. This construction is narrow but not rank one, showing that \(\mathcal{HDL}_{\mathrm{nar}}\not\subseteq\mathcal{HDL}_{\mathrm{gen}}^{(1)}\).

Conversely, let \(m=2\), let \(\beta\neq0\), and assume that \(1-\beta x\neq0\) for all \(x\in\Omega\). Consider the rank-one rational symbol
\[
    P(x)=\begin{pmatrix}(1-\beta x)^{-1}\\1\end{pmatrix}
    \begin{pmatrix}1 & 1+x\end{pmatrix}.
\]
This symbol is pointwise rank one on \(\Omega\). If it admitted a universal finite-moment representation, then the scalar functional
\[
    \mu\longmapsto\int_\Omega\frac{1}{1-\beta x}\,d\mu(x)
\]
would be a fixed linear combination of finitely many polynomial moments for every positive atomic measure supported in \(\Omega\). Applying such an identity to point masses would force \((1-\beta x)^{-1}\) to coincide with a polynomial on \(\Omega\), which is impossible for \(\beta\neq0\). Thus the representation is rank one but not narrow, and \(\mathcal{HDL}_{\mathrm{gen}}^{(1)}\not\subseteq\mathcal{HDL}_{\mathrm{nar}}\).
\end{example}

The preceding example shows that rank-one recovery and finite-window realization are logically independent. Their intersection is nonempty, as illustrated by the constructions below.

\begin{example}[The special HDL construction lies in the overlap]
\label{ex:special-huang-overlap}
Reconsider Example~\ref{ex:special-Huang}. Set \(w=g^{k-2}\), \(a=\alpha_{k-2}\), \(b=\alpha_{k-1}\). The gradient recurrence gives \(g^{k-1}=(I-aA)w\), \(g^k=(I-bA)(I-aA)w\). After the history polynomials and inverse filters are simplified, the effective symbol is
\[
    P_{\mathrm{sp}}(x)=\begin{pmatrix}1\\ (1-ax)^2\end{pmatrix}
    \begin{pmatrix}1-bx & 1\end{pmatrix}.
\]
The symbol is rank one and has degree at most three. Consequently, \(M_0\) depends only on \(h_0,\ldots,h_3\), whereas \(M_1\) is obtained from the shifted window \(h_1,\ldots,h_4\) through the same linear map. Thus the special HDL construction admits both a rank-one effective representation and a narrow realization.
\end{example}

\begin{example}[Long, short, and mean LMSD]
\label{ex:lmsd-family-overlap}
Let \(z(x)=(1,x,\ldots,x^{m-1})^\top\). The long, short, and mean LMSD pencils are generated by
\[
    P_q(x)=q(x)z(x)z(x)^\top,\qquad q(x)\in\{1,x,1+x\}.
\]
Each symbol is rank one. Moreover, \(\deg P_q\leq2m-1\), so \(M_0=\Lambda_h(P_q)\) and \(M_1=\Lambda_h(xP_q)\) are determined by the shifted moment windows \(h_0,\ldots,h_{2m-1}\) and \(h_1,\ldots,h_{2m}\), respectively. Hence all three LMSD constructions admit narrow rank-one effective representations. These are the same long, short, and mean variants described from the selected-moment viewpoint in Example~\ref{ex:gudu-long-short-mean}.
\end{example}

The structural relations established above can be summarized as
\[
    \mathcal{HDL}_{\mathrm{gen}}^{(1)}\subsetneq\mathcal{HDL}_{\mathrm{gen}},\qquad
    \mathcal{HDL}_{\mathrm{nar}}\subsetneq\mathcal{HDL}_{\mathrm{gen}},
\]
while
\[
    \mathcal{HDL}_{\mathrm{gen}}^{(1)}\not\subseteq\mathcal{HDL}_{\mathrm{nar}},\qquad
    \mathcal{HDL}_{\mathrm{nar}}\not\subseteq\mathcal{HDL}_{\mathrm{gen}}^{(1)}.
\]
The overlap \(\mathcal{HDL}_{\mathrm{gen}}^{(1)}\cap\mathcal{HDL}_{\mathrm{nar}}\) is therefore a nontrivial structural regime in which rank-one spectral recovery and bounded-window implementation are simultaneously available.

\subsection{Compatibility of HDL pencils and Gu--Du moment recovery}
\label{subsec:gudu-compatibility-rigidity}

The HDL and Gu--Du constructions recover spectral information through different algebraic mechanisms. An HDL construction uses the generalized eigenvalues of a reduced pencil, whereas a Gu--Du construction uses selected moment equations to recover atomic nodes and weights. Equality of the recovered nodes does not imply that an arbitrary pencil is determined by the same finite moment data. Compatibility must therefore be formulated at the level of a specific effective representation.

Throughout this subsection, \(\Omega\subset(0,\infty)\) is a compact interval. All atomic nodes lie in \(\Omega\), and every rational filter is assumed to have no pole on \(\Omega\). The proofs are collected in Appendix~\ref{app:gudu-huang-proofs}.

\begin{definition}[Moment-compatible HDL representation]
\label{def:gudu-huang-realization}
Let \(G\in\mathcal{GD}_m\) recover positive distinct nodes \(a_1(h),\ldots,a_m(h)\) from an admissible moment sequence \(h\). An effective HDL pencil \(M_1(h)-\theta M_0(h)\) is called a \emph{moment-compatible HDL representation} of \(G\) if there exists a matrix-valued polynomial symbol \(P\) satisfying:
\begin{enumerate}
    \item the reduced matrices are generated by
    $M_0(h)=\Lambda_h(P)$ and $M_1(h)=\Lambda_h(xP)$;
    \item every entry of \(M_0(h)\) and \(M_1(h)\) is a fixed linear functional of \(h_0,\ldots,h_{2m}\);
    \item the pencil is regular, and its finite generalized eigenvalues are exactly \(a_1(h),\ldots,a_m(h)\), counted with algebraic multiplicity.
\end{enumerate}
The representation is called \emph{narrow} if there exists a fixed linear map \(\mathcal T:\mathbb R^{2m}\longrightarrow\mathbb R^{m\times m}\) such that
\begin{equation}
\label{eq:narrow-realization-window}
    M_0(h)=\mathcal T(h_0,\ldots,h_{2m-1}),\qquad M_1(h)=\mathcal T(h_1,\ldots,h_{2m}).
\end{equation}
The representation is called \emph{rank one} if \(\operatorname{rank}P(x)\leq1\), \(x\in\Omega\).
\end{definition}

For quadratic objectives, the effective symbol is obtained by absorbing the gradient-history polynomials into the original spectral filter. Suppose that \(g^{\nu_p(k)}=\rho_p(A)w\), \(g^k=\rho_0(A)w\). Then
$P(x)=D_\rho(x)\Psi(x)\rho_0(x)$ with
$D_\rho(x)=\operatorname{diag}(\rho_1(x),\ldots,\rho_m(x))$.
The reduced matrices satisfy \(M_0=\Lambda_h(P)\), \(M_1=\Lambda_h(xP)\). If the original symbol is rank one, then the effective symbol is rank at most one. The converse requires the history factors to be nonzero on \(\Omega\) and is not needed below.

\begin{proposition}[Node recovery condition for effective symbols]
\label{prop:gudu-node-compatibility-symbol}
Assume that \(h_j=\sum_{i=1}^m t_ia_i^j\), \(j=0,\ldots,2m\), where \(a_1,\ldots,a_m\) are distinct nodes and \(t_1,\ldots,t_m\) are nonzero weights. Let \(M_0=\Lambda_h(P)\), \(M_1=\Lambda_h(xP)\). Then \(\theta=a_s\) is a generalized eigenvalue of \(M_1-\theta M_0\) if and only if
\begin{equation}
\label{eq:node-compatibility-condition}
    \det\!\left(
        \sum_{\substack{i=1\\i\neq s}}^m t_i(a_i-a_s)P(a_i)
    \right)=0.
\end{equation}
\end{proposition}

\begin{remark}[Rank-one symbols and exact node recovery]
\label{rem:rankone-symbol-node-compatibility}
Suppose that \(\operatorname{rank}P(a_i)\leq1\), \(i=1,\ldots,m\). For each fixed \(s\), the matrix in \eqref{eq:node-compatibility-condition} is a sum of at most \(m-1\) rank-one matrices and therefore has rank at most \(m-1\). Hence every atomic node \(a_s\) is a generalized eigenvalue of the pencil. Rank-one structure guarantees node inclusion but does not, by itself, guarantee regularity of the pencil, uniqueness of the recovered weights, or equivalence with a specified Gu--Du selector. These properties require separate assumptions.
\end{remark}

\begin{proposition}[Finite-degree compatible representations are narrow]
\label{prop:gudu-gen-equals-nar}
Let \(G\in\mathcal{GD}_m\) admit a moment-compatible HDL representation with effective polynomial symbol
\[
    P(x)=\sum_{j=0}^{2m-1}P^{(j)}x^j.
\]
Then the same representation is narrow.
\end{proposition}

The preceding proposition is a representation-level result. It does not assert that every pencil sharing the same generalized eigenvalues has a bounded shifted-moment realization.

A selected-moment rule can also generate a spectral relation that is not linear in the recovered node. The following one-node example illustrates this distinction.

\begin{example}[A one-node moment rule with no scalar linear-pencil representation]
\label{ex:gudu-skip-e1}
Consider \(h_j=ta^j\), \(j=0,1,2\), and select the equations \(E_0=0\), \(E_2=0\). Then \(t=h_0\), \(a=\sqrt{h_2/h_0}\). Suppose that this node admitted a universal scalar linear-pencil representation. On the two-point measure \(\mu_\omega=\omega\delta_x+(1-\omega)\delta_y\), \(0<\omega<1\), the resulting node would have the form
\begin{equation}
\label{eq:one-dimensional-huang-rational-form}
    R(\omega)=\frac{\omega x\Phi(x)+(1-\omega)y\Phi(y)}{\omega\Phi(x)+(1-\omega)\Phi(y)},
\end{equation}
whereas the selected-moment rule gives \(S(\omega)=\sqrt{\omega x^2+(1-\omega)y^2}\). Universality would require \(R(\omega)=S(\omega)\) for all \(\omega\in(0,1)\). Differentiation at the endpoints gives
$\Phi(x)/\Phi(y)=(x+y)/(2y)$ and
$\Phi(y)/\Phi(x)=(x+y)/(2x)$.
Multiplication yields \(1=(x+y)^2/(4xy)\), which contradicts \((x+y)^2>4xy\) for \(x\neq y\). Thus the selected-moment construction cannot be represented by a universal scalar linear spectral quotient of the form \eqref{eq:one-dimensional-huang-rational-form}. Its natural relation \(h_0\theta^2-h_2=0\) is nonlinear in \(\theta\). Because the HDL classes in this paper are defined for \(m\geq2\), this example is illustrative and is not used to claim a set-theoretic separation between \(\mathcal{GD}_m\) and \(\mathcal{HDL}_{\mathrm{gen}}\).
\end{example}

\begin{remark}[Nonlinear spectral relations]
\label{rem:gudu-polynomial-spectral-relations}
Selected moment equations may generate nonlinear spectral relations after the atomic weights have been eliminated. Their systematic construction generally requires polynomial elimination or resultant methods and lies outside the linear-pencil framework considered here.
\end{remark}

Positive semidefiniteness supplies a rigidity property for universally recovering effective symbols.

\begin{proposition}[Positive semidefinite symbols with universal exact recovery are rank one]
\label{prop:psd-universal-recovery-rank-one}
Let \(P(x)\in\mathbb R[x]^{m\times m}\) be symmetric and positive semidefinite for every \(x\in\Omega\). Assume that, for every choice of pairwise distinct nodes \(x_1,\ldots,x_m\in\Omega\) and positive weights \(t_1,\ldots,t_m\), the matrices
\[
    M_0=\sum_{i=1}^m t_iP(x_i),\qquad
    M_1=\sum_{i=1}^m t_ix_iP(x_i)
\]
satisfy:
\begin{enumerate}
    \item \(M_0\) is nonsingular;
    \item \(\det(M_1-x_sM_0)=0\) for every \(s=1,\ldots,m\).
\end{enumerate}
Then \(\operatorname{rank}P(x)=1\) for all \(x\in\Omega\).
\end{proposition}

\begin{corollary}[Compatibility in the positive rank-one finite-moment regime]
\label{cor:positive-triple-equivalence}
Let \(G\in\mathcal{GD}_m\) admit a moment-compatible polynomial HDL representation with effective symbol \(P\). Assume that:
\begin{enumerate}
    \item \(P(x)\) is symmetric and positive semidefinite for every \(x\in\Omega\);
    \item the pencil is regular for every positive \(m\)-atomic moment sequence with pairwise distinct nodes in \(\Omega\);
    \item every atomic node is recovered as a generalized eigenvalue;
    \item \(\deg P\leq2m-1\).
\end{enumerate}
Then \(\operatorname{rank}P(x)=1\), \(x\in\Omega\), and the same representation is narrow. Consequently, the construction simultaneously possesses generalized-eigenvalue recovery, rank-one effective-symbol structure, bounded shifted-moment realization, and a positive Gu--Du node--weight interpretation.
\end{corollary}

\begin{remark}[LMSD constructions in the compatible regime]
\label{rem:lmsd-positive-triple-equivalence}
For the long, short, and mean LMSD constructions,
\[
    P_q(x)=q(x)z(x)z(x)^\top,\qquad
    z(x)=(1,x,\ldots,x^{m-1})^\top,
\]
with \(q(x)\in\{1,x,1+x\}\). On every compact interval \(\Omega\subset(0,\infty)\), \(q(x)>0\), \(P_q(x)\succeq0\), \(\operatorname{rank}P_q(x)=1\). For pairwise distinct nodes, the evaluation matrix \([z(a_1)\ \cdots\ z(a_m)]\) is a nonsingular Vandermonde matrix. Under the usual no-breakdown condition, the associated pencils are regular and the recovered weights are positive. Hence the long, short, and mean LMSD constructions satisfy the compatibility conditions of Corollary~\ref{cor:positive-triple-equivalence}.
\end{remark}

The converse direction relevant to the present framework can be stated without asserting an unconditional inclusion between the entire HDL and Gu--Du classes.

\begin{proposition}[Compatibility of positive narrow rank-one constructions with standard moment recovery]
\label{prop:psd-nar-genone-in-gudu}
Let \(P(x)=q(x)u(x)u(x)^\top\) be an \(m\times m\) polynomial symbol satisfying \(q(x)>0\), \(x\in\Omega\). Let
\[
    h_j=\sum_{i=1}^m t_ia_i^j,\qquad t_i>0,\qquad a_i\in\Omega,
\]
be an exact positive \(m\)-atomic moment sequence with pairwise distinct nodes. Assume that the pencil \(M_1-\theta M_0\), \(M_0=\Lambda_h(P)\), \(M_1=\Lambda_h(xP)\), is regular. Then its finite generalized eigenvalues are precisely \(a_1,\ldots,a_m\). The same nodes and positive weights satisfy the standard long Gu--Du moment equations \(h_j=\sum_{i=1}^m t_ia_i^j\), \(j=0,\ldots,2m-1\). Thus the pencil and the standard Gu--Du construction are compatible on the specified positive atomic moment class.
\end{proposition}

\subsection{Design criterion for finite-history spectral candidates}
\label{subsec:summary-equivalence-separation}

The preceding analysis distinguishes three properties that are often implicitly combined in finite-history spectral constructions. A regular reduced pencil provides generalized-eigenvalue recovery; a bounded shifted-moment representation provides pseudo-memory implementability; and a positive atomic moment model provides weights with a component-strength interpretation. None of these properties implies the other two without additional assumptions.

A compatible regime is obtained when the effective symbol has the form
$P(x)=q(x)u(x)u(x)^\top$ with $q(x)>0$ for $x\in\Omega$,
the corresponding pencil is regular, and the degree of \(P\) is compatible with the available moment window. Rank-one structure then supports exact atomic-node recovery, the degree bound yields a narrow realization, and the positive moment model supplies component weights. The long, short, and mean LMSD constructions constitute representative instances of this regime.

This characterization provides a design criterion for finite-history spectral candidates. The reduced pencil should possess a positive rank-one effective symbol, admit a bounded shifted-moment representation, and preserve the associated positive atomic weights. These conditions underpin, respectively, the spectral interpretation, pseudo-memory implementation, and component-based ranking mechanism used in Section~\ref{sec:algorithm}.

\begin{remark}[Full-dimensional recovery and finite termination]
\label{rem:full-dimensional-equivalence}
Suppose that \(m=n\) and that \(L^k,R^k\in\mathbb R^{n\times n}\) are nonsingular. Then
\[
\det\!\left((L^k)^\top AR^k-\theta(L^k)^\top R^k\right) = \det(L^k)\det(R^k)\det(A-\theta I).
\]
Hence the generalized eigenvalues are exactly the eigenvalues of \(A\). If the reciprocals of all distinct eigenvalues are subsequently applied with the multiplicities required by the minimal polynomial relative to the initial gradient, then the gradient is annihilated in exact arithmetic. Full-dimensional spectral recovery therefore provides the candidates required for finite termination; termination also depends on the schedule with which those candidates are applied.
\end{remark}

The structural relations used in the subsequent development are summarized in Tables~\ref{tab:huang-side-relations} and~\ref{tab:gudu-huang-full}.

\begin{table}[t]
\caption{Relations between rank-one and finite-window HDL representations.}
\label{tab:huang-side-relations}
\centering
\begin{tabular}{ll}
\toprule
\textbf{Property} & \textbf{Reference} \\
\midrule
Rank-one HDL representations are general HDL representations & Definition~\ref{def:Hgen-rank-one} \\
General HDL representations need not be rank one & Example~\ref{ex:rankone-narrow-independence} \\
Rank-one representations need not be narrow & Example~\ref{ex:rankone-narrow-independence} \\
Narrow representations need not be rank one & Example~\ref{ex:rankone-narrow-independence} \\
The rank-one--narrow overlap is nonempty & Examples~\ref{ex:special-huang-overlap} and~\ref{ex:lmsd-family-overlap} \\
\bottomrule
\end{tabular}
\end{table}

\begin{table}[t]
\caption{Compatibility conditions for determinant-pencil and Gu--Du moment recovery.}
\label{tab:gudu-huang-full}
\centering
\begin{tabular}{ll}
\toprule
\textbf{Property} & \textbf{Reference} \\
\midrule
Atomic-node recovery condition & Proposition~\ref{prop:gudu-node-compatibility-symbol} \\
Rank-one symbols recover all atomic nodes & Remark~\ref{rem:rankone-symbol-node-compatibility} \\
Finite-degree compatible representations are narrow & Proposition~\ref{prop:gudu-gen-equals-nar} \\
Universal PSD exact recovery forces rank one & Proposition~\ref{prop:psd-universal-recovery-rank-one} \\
Positive rank-one finite-moment compatibility & Corollary~\ref{cor:positive-triple-equivalence} \\
LMSD constructions satisfy the compatibility conditions & Remark~\ref{rem:lmsd-positive-triple-equivalence} \\
Positive rank-one pencils recover the standard atomic nodes & Proposition~\ref{prop:psd-nar-genone-in-gudu} \\
\bottomrule
\end{tabular}
\end{table}

\section{Rebound Dynamics and Adaptive Spectral Ordering}
\label{sec:hanoi-principle}

The previous section establishes how memory-\(m\) constructions extract multiple spectral candidates from a finite gradient history. We now address the complementary ordering problem: once these candidates are available, which candidate should be used next, and how should repeated step-sizes be scheduled across phases?

The ordering is governed by the gradient component dynamics in \eqref{eq:intro-component-dynamics}. A step-size targeting one spectral scale suppresses the corresponding component but may amplify components associated with larger eigenvalues. We refer to the repeated reactivation of previously reduced high-frequency components as the \emph{rebound phenomenon}. This mechanism implies that high-frequency components should be controlled before lower-frequency components are accurately targeted and may need to be revisited after lower-frequency steps. We first characterize the resulting phase count in an idealized one-target model and then consider block and memory-\(m\) extensions.

\subsection{The rebound phenomenon and one-target phase ordering}
\label{subsec:rebound-two-rules}

The goal of a gradient method on a quadratic objective is to reduce each gradient component \(d_i^k\). By \eqref{eq:intro-component-dynamics}, a step-size \(\alpha_k\approx\lambda_j^{-1}\) nearly eliminates \(d_j^k\) over a phase targeting \(\lambda_j\). At the same time, any other component \(d_i^k\) is multiplied by \(1-\lambda_i/\lambda_j\). If \(\lambda_i<2\lambda_j\), this factor has magnitude less than one; if \(\lambda_i\gg\lambda_j\), its magnitude is greater than one. Thus, after \(\ell\) consecutive steps targeting \(\lambda_j\),
\begin{equation}
\label{eq:consecutive-amplification}
    |d_i^{k+\ell}| \approx \left|1-\frac{\lambda_i}{\lambda_j}\right|^\ell |d_i^k|, \qquad \lambda_i>\lambda_j .
\end{equation}
We call this repeated reactivation of previously suppressed high-frequency components the \emph{rebound phenomenon}.

This observation yields two ordering rules. \textit{First}, large eigenvalues should be targeted before small ones. Spectral step-sizes such as BB1 are gradient-magnitude-weighted averages of the eigenvalues; see \eqref{eq:intro-sd-bb-weighted-average}. Unsuppressed large-eigenvalue components dominate these averages, thereby obscuring the lower spectral scales. \textit{Second}, large-eigenvalue components may need to be revisited after small-eigenvalue steps. Such steps can amplify high-frequency components through \eqref{eq:consecutive-amplification}; in finite precision, residual components remain and may become dominant again. High-frequency suppression is therefore a recurring requirement rather than a one-time initialization.

Let \(\lambda_1\leq\cdots\leq\lambda_n\), with larger eigenvalues interpreted as higher-frequency components. To isolate the ordering effect of rebound, consider a reduced one-target phase model in which each phase targets one spectral scale. Targeting \(\lambda_i\) settles component \(i\) but may reactivate all higher-frequency components \(j>i\). A component is regarded as accessible only after the higher-frequency components in the current subproblem have been settled. Consequently, a lower-frequency phase requires the associated higher-frequency components to be cleared beforehand and revisited afterward.

The following proposition gives the exact phase count under these idealized transition rules.

\begin{proposition}[One-target phase count under complete rebound]
\label{prop:hanoi-principle}
Let \(0<\lambda_1<\cdots<\lambda_n\) be the eigenvalues of \(A\), and let \(\delta_i>0\) be target suppression thresholds. Suppose:
\begin{enumerate}
    \item \(|d_i^0|>\delta_i\) for all \(i\).
    \item Each phase targets one unresolved component \(i\) using \(\alpha\approx\lambda_i^{-1}\).
    \item A phase targeting \(\lambda_i\) brings \(|d_i|\leq\delta_i\), raises \(|d_j|>\delta_j\) for all \(j>i\), and leaves every unresolved lower-frequency component \(j<i\) unresolved.
    \item Within each unresolved subproblem, component \(i\) is accessible only after all higher-frequency components \(j>i\) in that subproblem have been settled.
\end{enumerate}
Let \(N_k\) be the minimum number of phases needed to settle \(d_{n-k+1},\ldots,d_n\) under these transition rules. Then
\begin{equation}
\label{eq:hanoi-recurrence}
    N_k=2N_{k-1}+1,\qquad N_1=1,
    \qquad\text{and hence}\qquad N_k=2^k-1.
\end{equation}
\end{proposition}

The proof is given in Appendix~\ref{app:StandardHanoi}. The recurrence expresses the phase structure induced by complete rebound: each lower-frequency phase must be preceded and followed by suppression of the associated higher-frequency components. It is an exact count for the stated phase model rather than an iteration-complexity bound for the practical algorithm. Figure~\ref{fig:hanoi-combined} illustrates the corresponding high-frequency revisit pattern in BB1: during iterations \(4\)--\(6\) the step-size targets an intermediate scale and only the higher-frequency component grows, whereas during iterations \(10\)--\(12\) the step-size targets \(\lambda_1\) and the largest component \(\lambda_3=10^2\) rebounds most strongly, while the \(\lambda_2\) component rebounds moderately.

\begin{figure}[!htbp]
\centering
\begin{minipage}[t]{0.48\textwidth}
    \centering
    \includegraphics[width=\linewidth]{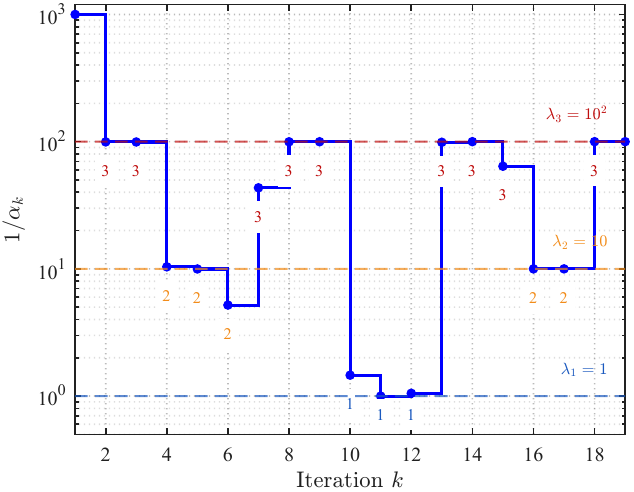}\\[2pt]
    {\small (a) Trajectory of $1/\alpha_k$}
\end{minipage}
\hfill
\begin{minipage}[t]{0.48\textwidth}
    \centering
    \includegraphics[width=\linewidth]{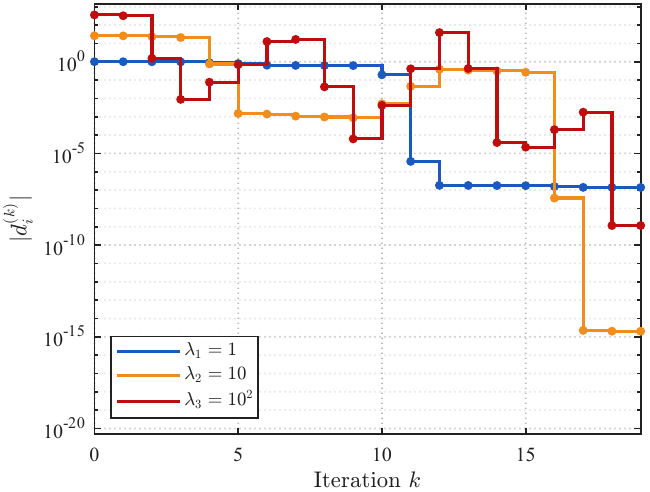}\\[2pt]
    {\small (b) Component magnitude evolution}
\end{minipage}
\caption{Evolution of \(1/\alpha_k\) during the first 18 iterations of BB1 on \(A=\operatorname{diag}(1,10,100)\) with \(x^0=(1,\,2.6,\,3.6)^\top\). The trajectory exhibits the high-frequency revisit pattern induced by spectral rebound.}
\label{fig:hanoi-combined}
\end{figure}

\begin{remark}[Comparison with steepest descent]
\label{rem:phase-count-vs-sd}
Under the same well-separated assumption, our phase count is no worse than steepest descent in order of magnitude. Suppose every phase is executed in at most \(L_{\max}\) iterations. Proposition~\ref{prop:hanoi-principle} gives the exact phase count \(N_n=2^{\,n}-1\), hence at most \(L_{\max}(2^{\,n}-1)\) iterations. In the same regime \(\lambda_{i+1}\geq c\lambda_i\) with \(c\geq2\), the condition number satisfies \(\kappa=\lambda_n/\lambda_1\geq c^{\,n-1}\), and steepest descent has the classical contraction factor \(q_{\mathrm{SD}}=(\kappa-1)/(\kappa+1)\) \cite{yuan2010qlinear,akaike1959}. On a well-separated spectrum this factor is attained, so reaching accuracy \(\varepsilon\) takes on the order of \(\frac{\kappa}{2}\log(1/\varepsilon)\geq \frac{c^{\,n-1}}{2}\log(1/\varepsilon)\) iterations. Both counts are therefore exponential in \(n\), and at the threshold \(c=2\) they share the same exponential base; for \(c>2\) the steepest-descent count grows faster while \(N_n\) is unchanged, so increasing \(c\) only favors the comparison.

The phase count \(2^{\,n}-1\) is exact under the stated transition rules. The steepest-descent count is attained on the same well-separated spectra but depends on the initial spectral distribution; an initial error supported on a single eigenspace is eliminated in one exact spectral step regardless of \(\kappa\).
\end{remark}

\subsection{Block scheduling and practical memory-\(m\) ordering}
\label{subsec:block-hanoi}

The one-target model gives an exact recurrence for a one-dimensional spectral phase chain. With memory \(m>1\), several Ritz candidates are available at each refresh, and the ordering problem becomes a controlled scheduling problem: the method must determine which accessible component to target and when to refresh the spectral information.

We first consider an ideal block model in which a memory-\(m\) method settles an entire block of \(m\) spectral components before refreshing the spectral information. Partition the \(n\) components into \(K=\lceil n/m\rceil\) ordered groups, where the lowest-frequency group has size \(r=n-(K-1)m\) and the remaining groups have size \(m\). Applying the complete-rebound recurrence at the block level gives the following count.

\begin{corollary}[Block phase count under complete rebound]
\label{cor:block-hanoi}
Under the assumptions of Proposition~\ref{prop:hanoi-principle} applied to groups of \(m\) eigenvalues, let \(K=\lceil n/m\rceil\) and \(r=n-(K-1)m\). If each block is fully settled before proceeding to the next block within this idealized block model, then the number of blocks \(C_k\) needed to settle the \(k\) highest-frequency groups satisfies \(C_k=2C_{k-1}+1\), \(C_1=1\), and hence \(C_k=2^k-1\). The total number of individual phases is
\begin{equation}
\label{eq:block-hanoi-step-count}
    T_n=m(2^K-2)+r.
\end{equation}
\end{corollary}

The proof is given in Appendix~\ref{app:BlockHanoi}. Figure~\ref{fig:block-hanoi-all} illustrates the block scheduling process for four eigenvalues grouped as \(\{\lambda_1,\lambda_2\}\) and \(\{\lambda_3,\lambda_4\}\). The Krylov--Ritz method with \(m=2\) requires three blocks and six phases, matching \(C_2=3\) and \(T_4=6\). BB1 completes 13 phases and 45 gradient steps, below the one-target complete-rebound count \(N_4=15\), because some phases do not reactivate all higher-frequency components.

\begin{figure}[!htbp]
\centering
\begin{subfigure}[t]{0.32\textwidth}
    \centering
    \includegraphics[width=\linewidth]{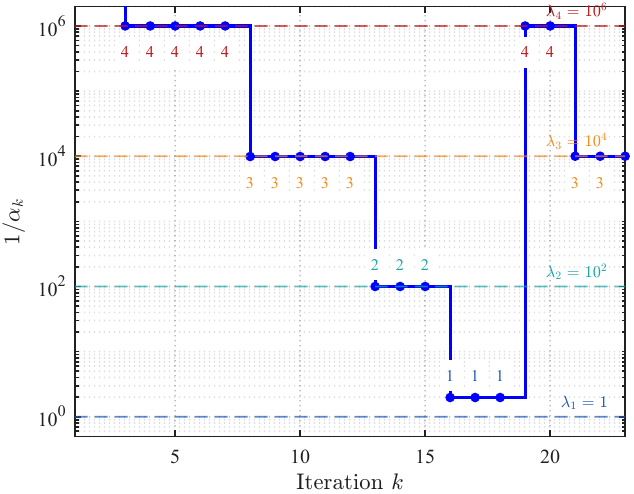}
    \caption{Inverse step-sizes $1/\alpha_k$.}
    \label{fig:block-hanoi-a}
\end{subfigure}%
\hfill
\begin{subfigure}[t]{0.32\textwidth}
    \centering
    \includegraphics[width=\linewidth]{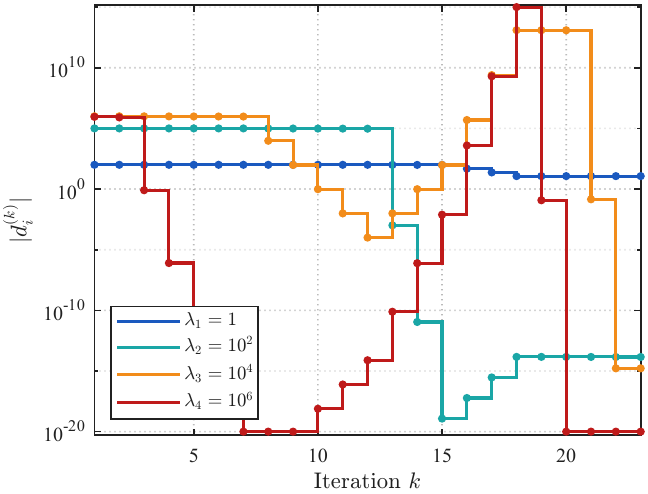}
    \caption{Error components.}
    \label{fig:block-hanoi-b}
\end{subfigure}%
\hfill
\begin{subfigure}[t]{0.32\textwidth}
    \centering
    \includegraphics[width=\linewidth]{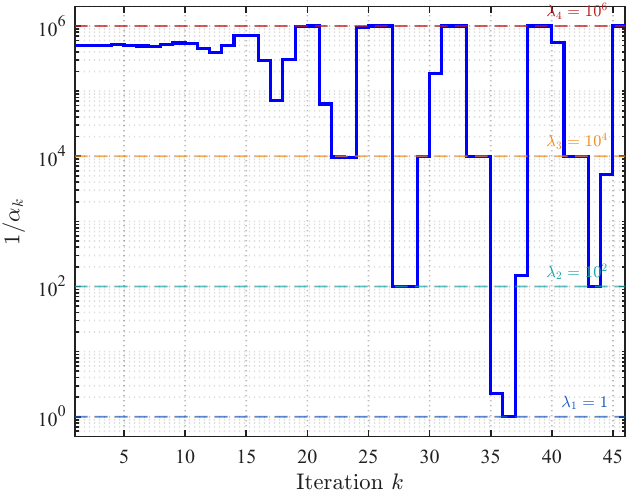}
    \caption{Inverse step-sizes for BB1.}
    \label{fig:block-hanoi-c}
\end{subfigure}
\caption{A toy example for the block phase recurrence. The parameters are \(A=\operatorname{diag}(1,10^2,10^4,10^6)\), \(x^0=(100,1000,100,1)^\top\), and the stopping criterion is \(\|g^k\|\leq10^{-5}\|g^0\|\).}
\label{fig:block-hanoi-all}
\end{figure}

The complete-block model is idealized. In practice, exhausting a block with stale Ritz values can be unstable because lower-frequency steps may strongly amplify higher-frequency components. Figure~\ref{fig:block-hanoi-explode-norm} demonstrates this effect in a six-dimensional example with \(m=2\): the gradient norm exceeds \(10^{30}\) once the lowest-frequency group \(\{\lambda_1,\lambda_2\}\) is targeted under fixed cycling.

\begin{figure}[!htbp]
\centering
\begin{subfigure}[t]{0.32\textwidth}
    \centering
    \includegraphics[width=\linewidth]{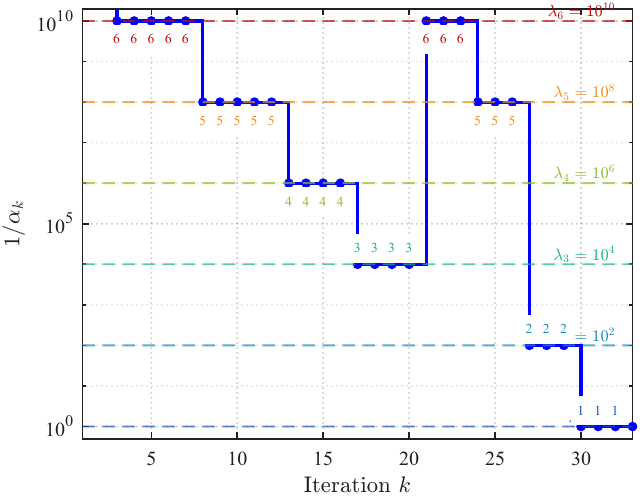}
    \caption{Inverse step-sizes \(1/\alpha_k\) under fixed-cycle block exhaustion. Ritz values drift from high-frequency toward low-frequency scales as dominant components are eliminated.}
    \label{fig:block-hanoi-explode-a}
\end{subfigure}\hfill
\begin{subfigure}[t]{0.32\textwidth}
    \centering
    \includegraphics[width=\linewidth]{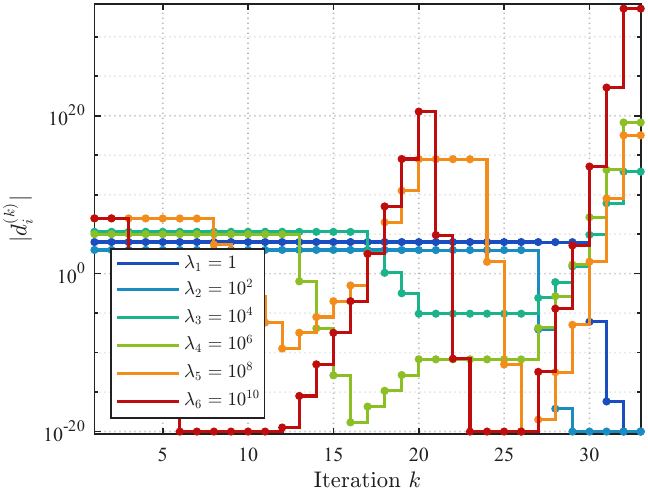}
    \caption{Component weights of the Ritz values during block exhaustion. High-frequency weight rebounds when lowest-frequency group are targeted.}
    \label{fig:block-hanoi-explode-b}
\end{subfigure}\hfill
\begin{subfigure}[t]{0.32\textwidth}
    \centering
    \includegraphics[width=\linewidth]{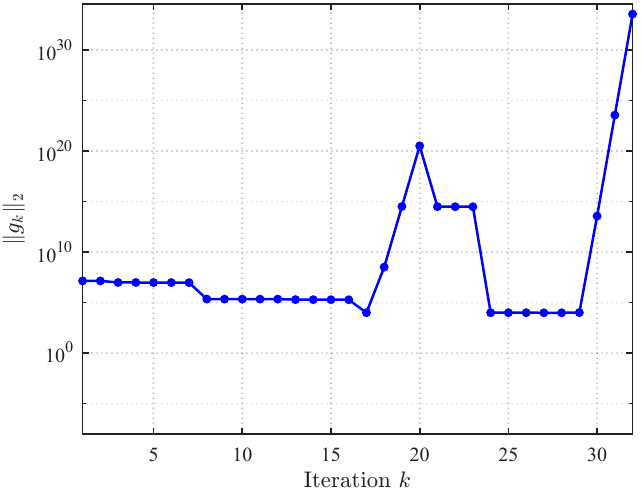}
    \caption{Gradient norm under fixed-cycle block exhaustion. It remains controlled during high-frequency stages but exceeds \(10^{30}\) once the low-frequency group is targeted.}
    \label{fig:block-hanoi-explode-c}
\end{subfigure}
\caption{Failure of fixed block exhaustion under spectral rebound. The parameters are \(A=\operatorname{diag}(1,10^2,10^4,10^6,10^8,10^{10})\), \(x^0=(10^4,10,20,0.1,0.1,10^{-3})^\top\), and \(m=2\).}
\label{fig:block-hanoi-explode-norm}
\end{figure}

The complete-block model is therefore too rigid for practical use. We instead execute one phase, refresh the spectral information, and select again. The following proposition gives the phase count for a special recomputed memory-\(m\) model.

\begin{proposition}[Phase count for recomputed memory-\(m\) ordering]
\label{prop:practical-hanoi}
Retain the notation and assumptions of Proposition~\ref{prop:hanoi-principle}. Consider the memory-\(m\) phase model in which at most \(m\) consecutive unresolved spectral components are simultaneously available as phase targets. For an \(n\)-component subproblem with \(n>m\), the initial available block contains the \(m\) highest-frequency unresolved components. After each phase, the block is refreshed, and a single phase can advance the lowest-frequency accessible boundary by at most one spectral level.

Assume further that the phase transition has the following recursive structure:
\begin{enumerate}
    \item exposing the lowest-frequency component of an \(n\)-component subproblem requires an optimal preparatory schedule for an \((n-m)\)-component subproblem;
    \item one additional phase settles the newly accessible component;
    \item after this phase, the remaining unresolved components form an \((n-1)\)-component subproblem governed by the same transition rules.
\end{enumerate}
A phase ends as soon as its selected component satisfies \(|d_i^k|\leq\delta_i\). Let \(f_m(n)\) denote the minimum number of phases required to bring \(|d_i^k|\leq\delta_i\) for all \(i=1,\ldots,n\) under this model. Then, for \(n>m\),
\begin{equation}
\label{eq:memory-m-recursion}
    f_m(n)=f_m(n-1)+f_m(n-m)+1,
\end{equation}
with $f_m(k)=k$ for \(0\leq k\leq m\). Moreover, \(f_m(n)=\Theta(\rho_m^{\,n})\), where \(\rho_m\in(1,2]\) is the unique positive real root of
\begin{equation}
    x^m=x^{m-1}+1.
\end{equation}
In particular, \(\rho_1=2\), \(f_1(n)=N_n=2^n-1\), and \(\rho_m\) is strictly decreasing in \(m\), with \(\rho_m\to1\) as \(m\to\infty\).
\end{proposition}

The proof is given in Appendix~\ref{app:practicalHanoi}. The recurrence describes the exact phase count of the stated recursive transition model; it is not asserted as an iteration bound for the finite-precision adaptive method.

\begin{remark}
\label{rem:explicit-rates}
For \(m=2\), the characteristic equation \(x^2=x+1\) gives \(\rho_2=(1+\sqrt{5})/2\approx1.618\), whereas the ideal block model has growth rate \(2^{1/2}\approx1.414\). For \(m=3\), the equation \(x^3=x^2+1\) gives \(\rho_3\approx1.3247\), while the ideal block rate is \(2^{1/3}\approx1.260\). Thus, increasing the memory reduces the phase-count growth rate relative to the one-target model, while complete block exhaustion gives a smaller idealized rate at the cost of using stale spectral information.
\end{remark}

The scheduling paradigm suggested by Proposition~\ref{prop:practical-hanoi} is to maintain a memory window of \(m\) Ritz values, select an accessible target, execute one phase, and refresh the spectral approximation before selecting again. In practice, candidate accuracy, component strength, phase length, and rebound severity jointly determine the next target. The following section incorporates these quantities into an adaptive spectral gradient method.

\section{Rebound-Aware Spectral Gradient Method}
\label{sec:algorithm}

Section~\ref{sec:hanoi-principle} establishes the ordering principle induced by spectral rebound in an idealized phase model: high-frequency components should be controlled before lower-frequency components are accurately targeted, and they may need to be revisited after lower-frequency steps. In a finite-memory method, however, the full spectral decomposition is unavailable. Each refresh provides only a small collection of Ritz candidates and estimates of their associated components.

This section translates the idealized phase logic into a practical memory-\(m\) method. At each refresh, the method makes three decisions. First, it identifies safe Ritz candidates and ranks them according to their estimated rebound relevance. Second, it specifies how strongly the selected component should be suppressed before proceeding to lower spectral scales. Third, it predicts how many times the corresponding reciprocal Ritz value should be reused before the reduced model is refreshed.

The resulting \emph{rebound-aware spectral gradient method}, abbreviated as RASG, combines a frequency--component ranking criterion with an adaptive phase-length estimator. The variants based on the Kato--Temple-inspired estimator and the long/short LMSD discrepancy are denoted by RASG-KT and RASG-LMSD, respectively. Without loss of generality, the presentation below uses Long LMSD to construct the primary Ritz candidates.

Accordingly, Section~\ref{subsec:ranking} constructs the rebound-aware ranking rule and the admissibility safeguard. Section~\ref{subsec:visibility} defines the frequency-dependent settlement level and adaptive phase length. Section~\ref{subsec:delta-theta} introduces the uncertainty estimators required by the phase-length prediction, and Section~\ref{subsec:alg-summary} combines these ingredients into the complete pseudo-memory algorithm.

\subsection{Weighted component estimate and admissibility filter}
\label{subsec:ranking}

At a spectral refresh, the reduced model generally contains several Ritz candidates, but only one is used to define the next phase. The idealized analysis of Section~\ref{sec:hanoi-principle} suggests that this choice should account for the obstruction that each component may create for subsequent lower-frequency phases. A component is especially relevant to rebound when it is both strongly represented in the current gradient history and located at a high spectral frequency: the former determines the amplitude available for re-amplification, while the latter determines its sensitivity to a later low-frequency step.

Approximation quality also affects this decision. A poorly resolved Ritz value may not suppress the intended component reliably and should not support a long phase. Following Gu and Du~\cite{gu2021modified}, a larger estimated component coefficient is associated with a smaller Ritz residual under the Lanczos conditions stated in Appendix~\ref{app:energy-residual}. Denoting the Ritz residual by \(r_i^k=Av_i^k-\theta_i^kv_i^k\), the energy--residual bound gives
\begin{equation}
\label{eq:energy-residual}
    \|r_i^k\|_2^2 \;\leq\; (\beta_{m,\,\mathrm{ind_r}}^k)^2
    \left( 1-\frac{t_i^{k,\,\mathrm{ind_r}}} {\|g^{k-m+\mathrm{ind_r}-1}\|_2^2} \right),
\end{equation}
where \(\beta_{m,\,\mathrm{ind_r}}^k\) is the last Lanczos subdiagonal coefficient for the process on \(V^k\) started at \(g^{k-m+\mathrm{ind_r}-1}\). Thus, a larger \(t_i^{k,\,\mathrm{ind_r}}\) indicates both a more strongly represented component and, under these conditions, a tighter Ritz-residual bound.

The quantity \(t_i^{k,\,\mathrm{ind_r}}\) generalizes the component coefficient \(t_i^k\) introduced in Section~\ref{subsec:moment-hankel-narrow}. Setting \(t_i^{k,1}=t_i^k\), define
\begin{equation}
\label{eq:component-energy-general}
    t_i^{k,\,\mathrm{ind_r}} \;=\; t_i^{k,1}
    \prod_{j=1}^{\mathrm{ind_r}-1} \bigl(1-\alpha_{k-m+j-1}\theta_i^k\bigr)^2,
    \qquad \mathrm{ind_r}\in\{1,\ldots,m+1\}.
\end{equation}
This quantity estimates the strength of the \(i\)-th Ritz component at a selected position in the reconstructed gradient-history window. The position parameter \(\mathrm{ind_r}\) allows the ranking rule to respond to a component that was reduced near the beginning of the window but became visible again after subsequent steps.

Frequency alone is insufficient because a negligible high-frequency component poses little immediate rebound risk, while component strength alone does not account for the larger amplification associated with high spectral ratios. We therefore define the \emph{weighted component estimate}
\begin{equation}
\label{eq:weighted-component-energy}
    \eta_i^k \;=\; (\theta_i^k)^{r_\theta} (t_i^{k,\,\mathrm{ind_r}})^{r_t},
    \qquad r_\theta,r_t>0.
\end{equation}
The exponent \(r_\theta\) controls the preference for high spectral frequencies, whereas \(r_t\) controls the sensitivity to the reconstructed component strength. Thus, \(\eta_i^k\) is a ranking surrogate for rebound relevance rather than an approximation to an exact spectral energy.

The ranking criterion determines which candidate is desirable, but it does not guarantee that the corresponding reciprocal Ritz value is safe. RASG therefore separates \emph{rebound priority} from \emph{spectral admissibility}: the score \(\eta_i^k\) orders the candidates, while the following filter removes candidates that may produce excessively large spectral steps.

Following Gu and Du~\cite{gu2021modified}, define
\begin{equation}
\label{eq:admissibility-threshold}
\begin{aligned}
    \omega_k &\;=\; \frac{1+\gamma}{2}
    \frac{(g^{k-m+1}-g^{k-m})^\top (x^{k-m+1}-x^{k-m})} {\|x^{k-m+1}-x^{k-m}\|_2^2} \\[6pt]
    &\;=\; \frac{1+\gamma}{2}
    \frac{\|g^{k-m}\|_2^2 -(g^{k-m+1})^\top g^{k-m}} {\alpha_{k-m} \|g^{k-m}\|_2^2},
\end{aligned}
\end{equation}
where \(\gamma\in(0,1]\) is a safeguard parameter. The second equality follows from \(x^{j+1}-x^j=-\alpha_jg^j\) and is therefore compatible with the pseudo-memory realization.

The \emph{admissible set} is
\begin{equation}
\label{eq:admissible-set}
    \mathcal{I}_k = \{i:\theta_i^k\geq\omega_k\}.
\end{equation}
By~\cite{gu2021modified}, the largest Ritz value satisfies \(\theta_m^k\geq\omega_k\). Hence \(\mathcal{I}_k\neq\emptyset\) in exact arithmetic.

The term appearing on the right-hand side of \eqref{eq:admissibility-threshold} also defines the anchor-based BB1 reciprocal step-size
\begin{equation}
\label{eq:rank-deficient-fallback}
    \theta_k^{\mathrm{BB1}}
    =
    \frac{
        \|g^{k-m}\|_2^2
        -(g^{k-m+1})^\top g^{k-m}
    }{
        \alpha_{k-m}\|g^{k-m}\|_2^2
    }
    =
    \frac{
        (g^{k-m})^\top Ag^{k-m}
    }{
        \|g^{k-m}\|_2^2
    }.
\end{equation}
This quantity is the Rayleigh quotient of \(A\) at the anchor gradient \(g^{k-m}\), and therefore
$\lambda_1\le\theta_k^{\mathrm{BB1}}\le\lambda_n$
and
$\theta_k^{\mathrm{BB1}}\ge\frac{1+\gamma}{2}\theta_k^{\mathrm{BB1}}=\omega_k$,
where the inequality uses \(\gamma\in(0,1]\). It plays the role of a safeguarded fallback whenever the reduced spectral extraction cannot be performed; see Section~\ref{subsec:alg-summary}.

The admissibility filter provides the safeguard used in Section~\ref{sec:convergence_analysis} to prove \(R\)-linear convergence. It imposes a lower bound on the selected reciprocal step-size and therefore prevents the spectral step-size from becoming excessively large. Related safeguards appear in \cite{Burdakov_2019,huang2021equipping,zhang2024cyclic}, and classical rules with established \(R\)-linear convergence, including BB, Yuan, and NY, also satisfy this condition.

At this point, RASG has an admissible candidate pool and a rebound-aware ordering of that pool. It remains to determine how far the selected component should be reduced and how long the corresponding step-size should be reused.
\subsection{Settlement threshold and phase-length estimator}
\label{subsec:visibility}

The ranking rule identifies the component whose current presence poses the greatest estimated rebound risk. A single application of its reciprocal Ritz value, however, may not suppress that component sufficiently before the method proceeds to lower spectral scales. RASG therefore treats the selected candidate as the target of a short \emph{phase}: its spectral step-size is reused until the component is predicted to reach a rebound-adjusted settlement level, or until the continuation cap forces the reduced model to be refreshed.

Let \(\epsilon=\Theta_{\mathrm{tol}}\|g^0\|_2\) be the absolute stopping tolerance. A natural reference level for one component is the budget \(\epsilon^2/n\), obtained by distributing \(\|g^k\|_2^2\leq\epsilon^2\) uniformly across the \(n\) spectral components. This flat budget does not account for the frequency-dependent amplification underlying rebound.

After component \(i\) has been suppressed, switching to a lower-frequency candidate \(\theta_j^k<\theta_i^k\) may re-amplify the component by a factor governed by their spectral ratio; see \eqref{eq:consecutive-amplification}. A high-frequency component should therefore satisfy a stricter settlement condition before the method moves to lower frequencies.

To represent the lowest-frequency scale currently available under the assumption that \(\theta_i^k\) is the target of the phase, define
\begin{equation}
\label{eq:theta-lower}
    \theta_{\min,i}^k \;=\; \min\Bigl\{ \min_{0\leq j<k}\alpha_j^{-1}, \theta_i^k \Bigr\}.
\end{equation}
Equation~\eqref{eq:intro-sd-bb-weighted-average} also shows that BB1-type reciprocal step-sizes are weighted spectral averages. Consequently, a persistent high-frequency component may mask lower spectral scales even after it has been substantially reduced.

Ranking and settlement may monitor different positions in the reconstructed history. In parallel with \(t_i^{k,\,\mathrm{ind_r}}\), define \(t_i^{k,\,\mathrm{ind_s}}\) through \eqref{eq:component-energy-general}, where \(\mathrm{ind_s}\in\{1,\ldots,m+1\}\) is selected independently. The \emph{rebound-adjusted settlement threshold} is
\begin{equation}
\label{eq:hanoi-threshold}
    \tau_i^{k,\,\mathrm{ind_s}} \;=\; \frac{\epsilon^2}{n}
    \left(\frac{\theta_{\min,i}^k}{\theta_i^k}\right)^r,
    \qquad r\geq2.
\end{equation}
The ratio \(\theta_{\min,i}^k/\theta_i^k\) is a reduced-model proxy for the separation between the current target and the lowest spectral scale that may be used later. Hence, the threshold becomes stricter when a high-frequency candidate is far from the available low-frequency scale. This construction is not intended to reproduce the exact future amplification factor; it converts the ordering principle of Section~\ref{sec:hanoi-principle} into a computable frequency-dependent budget.

Once \(t_i^{k,\,\mathrm{ind_s}}<\tau_i^{k,\,\mathrm{ind_s}}\), the current reduced model regards component \(i\) as sufficiently suppressed for lower-frequency candidates to become eligible in subsequent phases.

A settlement threshold alone does not determine the phase length. If a candidate were an exact eigenvalue of a quadratic model, one application of its reciprocal would eliminate the corresponding component. A Ritz candidate is only approximate, and the remaining component is governed by its relative spectral mismatch. RASG therefore uses the estimated relative Ritz uncertainty as a proxy for the per-step persistence of the selected component.

Define $\rho_i^k=|\Delta\theta_i^k|/\theta_i^k$,
where \(\Delta\theta_i^k\) is estimated as described in Section~\ref{subsec:delta-theta}. The quantities \(t_i^{k,\,\mathrm{ind_s}}\) and \(\tau_i^{k,\,\mathrm{ind_s}}\) are squared component estimates, while \(\rho_i^k\) measures the relative magnitude of the residual. A magnitude ratio \(\rho_i^k\) therefore corresponds to a per-step energy decay of \((\rho_i^k)^2\), and the number of steps required to reduce \(t_i^{k,\,\mathrm{ind_s}}\) to \(\tau_i^{k,\,\mathrm{ind_s}}\) is obtained by equating \((\rho_i^k)^{2\ell}\) with \(\tau_i^{k,\,\mathrm{ind_s}}/t_i^{k,\,\mathrm{ind_s}}\). When \(t_i^{k,\,\mathrm{ind_s}}>\tau_i^{k,\,\mathrm{ind_s}}\) and \(0<\rho_i^k<1\), the predicted continuation length is
\begin{equation}
\label{eq:continuation-score}
    \widehat{\kappa}_i^k \;=\; \max\!\left(
        \left\lceil \frac{\log(\tau_i^{k,\,\mathrm{ind_s}} /t_i^{k,\,\mathrm{ind_s}})} {2\log\rho_i^k} \right\rceil,\; 0
    \right).
\end{equation}
In all other cases we set \(\widehat{\kappa}_i^k=0\), so that the candidate cannot enter the continuation set. This quantity is an adaptive predictor rather than an exact component-energy recurrence, since both the component estimate and the Ritz uncertainty are inferred from the reduced model.

Let \(\kappa_{\mathrm{cont}}\geq1\) be the continuation cap. It limits the use of stale spectral information and guarantees that the reduced model is refreshed after finitely many repetitions. Define the \emph{continuation set}
$\mathcal{J}_k=\{i\in\mathcal{I}_k:1\leq\widehat{\kappa}_i^k\leq\kappa_{\mathrm{cont}}\}$.

If \(\mathcal{J}_k\neq\emptyset\), the admissible candidate with the largest rebound-aware score is selected:
$i_k=\arg\max_{i\in\mathcal{J}_k}\eta_i^k$ and $\ell_k=\widehat{\kappa}_{i_k}^k$.
If \(\mathcal{J}_k=\emptyset\), no admissible candidate supports a reliable multi-step continuation. The method then applies one step using the most strongly ranked admissible candidate:
$i_k=\arg\max_{i\in\mathcal{I}_k}\eta_i^k$ and \(\ell_k=1\).

Thus, the ranking score determines \emph{which} component is treated, the settlement threshold determines \emph{how far} it should be reduced, and the uncertainty estimate predicts \emph{how long} the phase should continue. When that prediction is unreliable, RASG takes one step and refreshes the reduced model rather than committing to a long phase.

\subsection{Estimating the Ritz uncertainty \(\Delta\theta_i^k\)}
\label{subsec:delta-theta}

The ranking score and settlement threshold determine the target component and its required suppression level, respectively. Neither directly specifies how rapidly the selected component will approach that level. This remaining task is performed by the phase-length estimator \eqref{eq:continuation-score}, which requires an estimate of the relative Ritz uncertainty \(|\Delta\theta_i^k|/\theta_i^k\).

The uncertainty estimate does not redefine the rebound ranking. Instead, it determines the confidence with which RASG commits to a multi-step phase. Two practical estimators are considered below. A third estimator, obtained from a first-order perturbation of the moment system, is recorded in Appendix~\ref{app:perturbation-estimator}; it is not used in the experiments because the associated confluent Vandermonde system becomes ill-conditioned when the Ritz values cluster.

The first estimator is motivated by the Kato--Temple inequality~\cite[Theorem~11.7.1]{parlett1998symmetric}. Under the corresponding separation assumptions, an eigenvalue \(\lambda\) associated with \(\theta_i^k\) satisfies a bound of the form \(|\lambda-\theta_i^k|\leq\|r_i^k\|_2^2/\delta_i^k\). Replacing the unknown spectral separation by the neighboring Ritz gap gives
\begin{equation}
\label{eq:kt-estimator}
    \Delta\theta_{i,\mathrm{KT}}^k \;=\; \frac{\|r_i^k\|_2^2}{\widehat{\delta}_i^k},
\end{equation}
where, for at least two distinct effective Ritz values,
\begin{equation}
\label{eq:neighboring-ritz-gap}
\widehat{\delta}_i^k =
\begin{cases}
\theta_2^k-\theta_1^k, & i=1,\\[2mm]
\min\{\theta_i^k-\theta_{i-1}^k,\, \theta_{i+1}^k-\theta_i^k\}, & 2\leq i\leq m-1,\\[2mm]
\theta_m^k-\theta_{m-1}^k, & i=m.
\end{cases}
\end{equation}
The residual norm is recoverable from the reduced matrices without storing full Ritz vectors, as shown in \eqref{eq:app-residual-formula} of Appendix~\ref{app:pseudo-memory}. Since the neighboring Ritz gap is only a computable separation proxy, \eqref{eq:kt-estimator} is referred to as a Kato--Temple-inspired estimator. If a neighboring Ritz gap is numerically zero, the corresponding candidate is not allowed to generate a multi-step KT phase.

The second estimator exploits the relation between BB1 and BB2 step-sizes, whose discrepancy reflects the departure of a secant direction from an invariant direction. By Remark~\ref{rmk:LMSDlongshort}, the pseudo-memory model provides \(m\) Long LMSD Ritz values \(\theta^k_{i,\mathrm{long}}\) and \(m\) Short LMSD Ritz values \(\theta^k_{i,\mathrm{short}}\).

A Cauchy--Schwarz argument in the \(A\)-inner product gives
\begin{equation}
\label{eq:long-short-ordering}
    \theta^k_{i,\mathrm{long}} \leq \theta^k_{i,\mathrm{short}},
    \qquad \alpha^k_{i,\mathrm{long}} \geq \alpha^k_{i,\mathrm{short}}.
\end{equation}
The ordered discrepancy
\begin{equation}
\label{eq:lmsd-estimator}
    \Delta\theta^k_{i,\mathrm{LMSD}} = \theta^k_{i,\mathrm{short}} - \theta^k_{i,\mathrm{long}} \geq0
\end{equation}
therefore provides a computable uncertainty estimate for the corresponding Long LMSD candidate.

The following proposition gives this discrepancy a residual interpretation; its proof is given in Appendix~\ref{app:long-short-proof}.

\begin{proposition}[Residual bound for the long/short estimator]
\label{prop:long-short-residual}
Let \(A\in\mathbb{S}^n_{++}\). By \eqref{eq:long-short-ordering}, \(\theta_{i,\mathrm{long}}^k \leq\theta_{i,\mathrm{short}}^k\). Let \(v_{i,\mathrm{long}}^k\) and \(v_{i,\mathrm{short}}^k\) be the corresponding unit Ritz vectors, and define
\begin{equation}
\label{eq:long-short-residual-definition}
    r_{i,\mathrm{short}}^k = Av_{i,\mathrm{short}}^k - \theta_{i,\mathrm{short}}^k v_{i,\mathrm{short}}^k,
    \qquad \delta_{v,i} = \min_{\sigma=\pm1} \bigl\| v_{i,\mathrm{short}}^k - \sigma v_{i,\mathrm{long}}^k \bigr\|_2.
\end{equation}
Then
\begin{equation}
\label{eq:long-short-bound}
    \|r_{i,\mathrm{short}}^k\|_2^2 \leq \lambda_n
    \Bigl( \theta_{i,\mathrm{short}}^k - \theta_{i,\mathrm{long}}^k + \|A\|_2\,\delta_{v,i}(2+\delta_{v,i}) \Bigr).
\end{equation}
In particular, when \(\delta_{v,i}\to0\), the long/short discrepancy directly controls the residual norm.
\end{proposition}

Proposition~\ref{prop:long-short-residual} identifies two proximity conditions: a small spectral discrepancy and directional alignment. The ordered long/short pairing used in \eqref{eq:lmsd-estimator} is most informative when these conditions hold. Near clustered Ritz values, the pairing may become ambiguous and the discrepancy should be regarded as a heuristic uncertainty indicator rather than a residual bound by itself.

The two estimators are summarized as
\begin{equation}
\label{eq:delta-theta-summary}
    \Delta\theta_i^k =
    \begin{cases}
        \displaystyle \frac{\|r_i^k\|_2^2}{\widehat{\delta}_i^k}, & \text{(Kato--Temple-inspired)}, \\[8pt]
        \theta_{i,\mathrm{short}}^k - \theta_{i,\mathrm{long}}^k, & \text{(long/short LMSD)}.
    \end{cases}
\end{equation}
In both cases, \(\rho_i^k=|\Delta\theta_i^k|/\theta_i^k\) is substituted into \eqref{eq:continuation-score}. The estimators do not change the rebound-aware score \(\eta_i^k\); they determine whether a ranked candidate can support a multi-step phase. If the resulting uncertainty is inadmissible, the method refreshes the reduced model after one step.

\subsection{Algorithm}
\label{subsec:alg-summary}

The preceding subsections define one complete rebound-aware phase. At a refresh point, the pseudo-memory model reconstructs the available Ritz candidates and component estimates. The admissibility filter removes unsafe candidates, the weighted component estimate ranks the remaining candidates, the settlement threshold specifies the required suppression level, and the uncertainty estimator determines the continuation length. After the phase, the scalar-product window is updated and the same decision is repeated.

This refresh--rank--settle--refresh cycle allows a previously suppressed high-frequency component to regain priority if it reappears after lower-frequency steps. No component is permanently removed from consideration; its estimated strength is recomputed whenever the reduced model is refreshed.

The complete RASG method consists of a main iteration
(Algorithm~\ref{alg:adaptive_hanoi_main}) and a pseudo-memory
selector (Algorithm~\ref{alg:adaptive_hanoi_selector}). Throughout,
the reduced-model dimension is fixed at \(m\), and the full-rank
condition \eqref{eq:no-breakdown} is imposed whenever the reduced
spectral model is constructed. If a refresh produces a numerically
rank-deficient window, the reduced spectral extraction is skipped,
one safeguarded BB1 fallback step is taken, and the scalar-product
window is advanced before the next refresh. The pseudo-memory formulation updates rolling scalar products rather
than recursively updating the moments. In the moment--Hankel form
used in the algorithms below, the moments associated with the
current anchor are recovered by
\eqref{eq:app-forward-substitution}.

\begin{algorithm}[!htbp]
\caption{RASG Main Framework}
\label{alg:adaptive_hanoi_main}
\begin{algorithmic}[1]
\Require $A\succ0$, $b$, $x^0$, memory $m$, $K_{\max}$,
         tolerance $\Theta_{\mathrm{tol}}$,
         cap $\kappa_{\mathrm{cont}}$, $\gamma\in(0,1]$,
         exponents $r_\theta,r_t>0$, indices
         $\mathrm{ind_r},\mathrm{ind_s}\in\{1,\ldots,m+1\}$,
         power $r\geq2$,
         $\mathtt{est}\in\{\mathtt{KT},\mathtt{LMSD}\}$
\Ensure approximate minimizer $x^k$

\State $g^0=Ax^0-b$;\;
       $\epsilon=\Theta_{\mathrm{tol}}\|g^0\|_2$;\;
       $\alpha_{\mathrm C}
       =\|g^0\|_2^2/((g^0)^\top Ag^0)$
\For{$j=0,\ldots,m-1$}
    \State $\alpha_j\leftarrow\alpha_{\mathrm C}$;\;
           $x^{j+1}=x^j-\alpha_jg^j$;\;
           $g^{j+1}=Ax^{j+1}-b$
\EndFor
\State Form $q^m$ from the initial window
       $\{g^0,\ldots,g^m\}$ through
       \eqref{eq:pseudo-memory-qk};\;
       $k\leftarrow m$

\While{$k<K_{\max}$ \textbf{and} $\|g^k\|_2>\epsilon$}
\State Recover $h_0^k,\ldots,h_{2m}^k$ from $q^k$ in
       \eqref{eq:pseudo-memory-qk} and the last $m$ step-sizes
       through \eqref{eq:app-forward-substitution};\;
       form $H_0^k,H_1^k,H_2^k$ through
       \eqref{eq:moment-Hankel-pair} and
       \eqref{eq:app-H2-definition}
    \If{$H_0^k$ is not numerically positive definite}
        \State Compute $\theta_k^{\mathrm{BB1}}$ from
               \eqref{eq:rank-deficient-fallback};\;
               $\alpha_k\leftarrow(\theta_k^{\mathrm{BB1}})^{-1}$
        \State $x^{k+1}=x^k-\alpha_kg^k$;\;
               $g^{k+1}=Ax^{k+1}-b$
        \State Update $q^{k+1}$ by evicting its two oldest scalar
               entries and appending $(g^{k+1})^\top g^k$ and
               $\|g^{k+1}\|_2^2$;\;
               $k\leftarrow k+1$;\;
               \textbf{continue}
    \EndIf
    \State $(i_k,\theta_{i_k}^k,\ell_k)\leftarrow
           \text{Algorithm~\ref{alg:adaptive_hanoi_selector}}(\cdot)$
    \State $\alpha\leftarrow(\theta_{i_k}^k)^{-1}$
    \For{$j=1,\ldots,\ell_k$}
        \State $x^{k+1}=x^k-\alpha g^k$;\;
               $g^{k+1}=Ax^{k+1}-b$
\State Update $q^{k+1}$ by evicting its two oldest scalar
       entries and appending $(g^{k+1})^\top g^k$ and
       $\|g^{k+1}\|_2^2$;\;
       $k\leftarrow k+1$
        \If{$k\geq K_{\max}$ \textbf{or}
            $\|g^k\|_2\leq\epsilon$}
            \State \textbf{break}
        \EndIf
    \EndFor
\EndWhile
\State \Return $x^k$
\end{algorithmic}
\end{algorithm}

Numerically, \(H_0^k\succ0\) is tested by Cholesky decomposition with a scale-dependent rank threshold: a refresh is declared rank deficient if the factorization fails or if a Cholesky pivot falls below \(\tau_{\mathrm{rank}}\) times the scale of \(H_0^k\).

\begin{algorithm}[!htbp]
\caption{Step-Size and Phase-Length Selection}
\label{alg:adaptive_hanoi_selector}
\begin{algorithmic}[1]
\Require memory size $m$;\;
         scalar-product vector $q^k$ from
         \eqref{eq:pseudo-memory-qk};\;
         last $m$ step-sizes
         $\{\alpha_{k-m},\ldots,\alpha_{k-1}\}$;\;
         $\epsilon,n,\kappa_{\mathrm{cont}}$;\;
         $r_\theta,r_t,\mathrm{ind_r},\mathrm{ind_s},r$;\;
         $\mathtt{est}\in\{\mathtt{KT},\mathtt{LMSD}\}$
\Ensure $(i_k,\theta_{i_k}^k,\ell_k)$

\State Recover $h_0^k,\ldots,h_{2m}^k$ from $q^k$ through
       \eqref{eq:app-forward-substitution};\;
       form the Hankel matrices $H_0^k,H_1^k,H_2^k$
       through \eqref{eq:moment-Hankel-pair} and
       \eqref{eq:app-H2-definition}
\State Solve $H_1^ky=\theta H_0^ky$;\;
       sort the Long LMSD Ritz values
       $\theta_1^k\leq\cdots\leq\theta_m^k$;\;
       if $\mathtt{est}=\mathtt{LMSD}$, also solve
       $H_2^ky=\theta H_1^ky$ and sort the Short LMSD
       Ritz values $\theta_{i,\mathrm{short}}^k$
\State Compute $t_i^{k,1}$ through
       \eqref{eq:app-anchor-weight};\;
       compute $t_i^{k,\mathrm{ind}}$ for
       $\mathrm{ind}\in\{\mathrm{ind_r},\mathrm{ind_s}\}$
       through \eqref{eq:component-energy-general};\;
       set $\eta_i^k$ through \eqref{eq:weighted-component-energy}
\State Compute $\Delta\theta_i^k$ through
       \eqref{eq:kt-estimator} with the residual
       $\|r_i^k\|_2^2$ evaluated by \eqref{eq:app-residual-formula}
       if $\mathtt{est}=\mathtt{KT}$,
       or through \eqref{eq:lmsd-estimator} otherwise
\State Set $\mathcal{I}_k=\{\theta_i^k\geq\omega_k\}$
       with $\omega_k$ from \eqref{eq:admissibility-threshold}
\For{$i\in\mathcal{I}_k$}
    \State $\theta_{\min,i}^k\leftarrow
           \min\bigl\{\min_{0\leq j<k}\alpha_j^{-1},\,
           \theta_i^k\bigr\}$
           through \eqref{eq:theta-lower}
    \State $\tau_i^{k,\mathrm{ind_s}}\leftarrow
           (\epsilon^2/n)(\theta_{\min,i}^k/\theta_i^k)^r$;\;
           $\rho_i^k\leftarrow|\Delta\theta_i^k|/\theta_i^k$
    \If{$t_i^{k,\mathrm{ind_s}}>\tau_i^{k,\mathrm{ind_s}}$
        \textbf{and} $0<\rho_i^k<1$}
        \State $\widehat{\kappa}_i^k\leftarrow$ formula
               \eqref{eq:continuation-score}
    \Else
        \State $\widehat{\kappa}_i^k\leftarrow0$
    \EndIf
\EndFor
\State $\mathcal{J}_k=\{i\in\mathcal{I}_k:
       1\leq\widehat{\kappa}_i^k\leq\kappa_{\mathrm{cont}}\}$
\If{$\mathcal{J}_k\neq\emptyset$}
    \State $i_k\leftarrow\arg\max_{i\in\mathcal{J}_k}\eta_i^k$;\;
           $\ell_k\leftarrow\widehat{\kappa}_{i_k}^k$
\Else
    \State $i_k\leftarrow\arg\max_{i\in\mathcal{I}_k}\eta_i^k$;\;
           $\ell_k\leftarrow1$
\EndIf
\State \Return $(i_k,\theta_{i_k}^k,\ell_k)$
\end{algorithmic}
\end{algorithm}

Algorithm~\ref{alg:adaptive_hanoi_main} first performs $m$ warm-start iterations with a fixed Cauchy step, thereby constructing the initial scalar-product window associated with $\{g^0,\ldots,g^m\}$. It then sets $k=m$ and repeatedly calls Algorithm~\ref{alg:adaptive_hanoi_selector}, which returns a Ritz candidate and a continuation length \(\ell_k\). The corresponding step-size is reused for \(\ell_k\) iterations before the reduced model is refreshed, and this repeated reassessment is the mechanism that allows a previously suppressed high-frequency component to be selected again after rebound. The method has six tunable parameters: the ranking index \(\mathrm{ind_r}\) with exponents \(r_\theta,r_t\) determines the ranking score \(\eta_i^k\); the settlement index \(\mathrm{ind_s}\) with exponent \(r\) determines the settlement threshold; and \(\kappa_{\mathrm{cont}}\) bounds the continuation length. Recommended values are reported in Section~\ref{sec:numerical-experiments}.

\section{Convergence Analysis}
\label{sec:convergence_analysis}

This section establishes global \(R\)-linear convergence of Algorithm~\ref{alg:adaptive_hanoi_main} for the strictly convex quadratic objective~\eqref{eq:intro-quadratic}. The analysis is conducted in the eigenbasis of \(A\), following the decomposition in Section~\ref{subsec:intro-related-work}. 

The RASG method permits a reciprocal step-size \(\theta_k\), selected at a refresh iteration, to be reused for \(\ell_k\) consecutive iterations. Consequently, the gradient window from which \(\theta_k\) is computed may precede the iterate to which it is applied by at most
\[
    \tilde{m} := m+\kappa_{\mathrm{cont}}-1
\]
iterations. The analysis shows that this bounded information lag does not destroy componentwise suppression: the admissibility filter enforces contraction of each unresolved spectral component within a uniformly bounded iteration window.

The proof follows the componentwise framework used in the analysis of spectral gradient methods \cite{dai2002r,gu2021modified}, with the admissibility condition \eqref{eq:admissible-set} providing the essential safeguard. Rather than imposing a direct forward-decay estimate for every admissible step-size, the argument traces each applied step-size backward to the refresh index at which it was computed. The admissibility threshold at the corresponding anchor gradient then yields a uniform contraction estimate for the dominant unresolved component. The bounded reuse length is absorbed into the enlarged history parameter \(\tilde{m}\).

The convergence argument covers both regular refreshes (Ritz step-sizes) and the rank-deficient BB1 fallback defined in \eqref{eq:rank-deficient-fallback}; the latter inherits the same spectral and admissibility bounds from \eqref{eq:admissibility-threshold}.

\medskip
\noindent\textit{Setup.}
Recall from \eqref{eq:intro-component-dynamics} that the eigenbasis components \(d_i^k:=q_i^\top g^k\) satisfy
\begin{equation}
\label{eq:eigen_component_update}
    d_i^{k+1} = \left(1-\frac{\lambda_i}{\theta_k}\right)d_i^k, \qquad i=1,\ldots,n,
\end{equation}
where \(\theta_k=\alpha_k^{-1}\) is the reciprocal step-size applied at iteration \(k\).

Define the partial component energy
$D(k,l):=\sum_{i=1}^{l}(d_i^k)^2$ for \(l=0,1,\ldots,n\),
with the convention \(D(k,0)=0\).

Introduce the constants
\begin{equation}
\label{eq:sc_def}
    s := \frac{2}{\gamma}, \qquad
    c := \max\!\left\{ 1-\frac{\lambda_1}{\lambda_n},\; \frac{1}{1+\gamma} \right\}, \qquad
    \delta := \max\!\left\{ \frac{\lambda_n}{\lambda_1}-1,\; 1 \right\}.
\end{equation}
These constants depend only on the spectrum of \(A\) and the safeguard parameter \(\gamma\in(0,1]\) from \eqref{eq:admissibility-threshold}. By construction, \(\delta\geq1\) and \(\frac12\leq c<1\).

Given \(\varepsilon_l\in(0,\gamma/4]\) and an integer \(m_l\geq0\), define
\begin{equation}
\label{eq:Delta_l_def}
    \Delta_l := \left\lceil \frac{ \log\!\left( s\,\varepsilon_l\, \delta^{-2(m_l+\tilde{m})} \right) } { 2\log c } \right\rceil +\tilde{m}-1.
\end{equation}
Because \(s\varepsilon_l\leq1/2\), \(\delta\geq1\), and \(c\in(0,1)\), the quantity in \eqref{eq:Delta_l_def} is well defined and \(\Delta_l\geq\tilde{m}\).

The analysis proceeds through four lemmas followed by the main convergence proposition. The proofs are given in Appendix~\ref{app:convergence-proofs}.

\begin{lemma}
\label{lem:single_step_bound}
For every \(i=1,\ldots,n\) and \(k\geq0\),
\[
    (d_i^{k+1})^2 \leq \delta^2(d_i^k)^2,
    \qquad \|g^{k+1}\|_2 \leq \delta\|g^k\|_2.
\]
\end{lemma}

\begin{lemma}
\label{lem:sweeping_l_plus_1}
Let \(0\leq l<n\) and \(k\geq m\). Suppose that there exist \(\varepsilon_l\in(0,\gamma/4]\) and an integer \(m_l\geq0\) such that
\[
    D(k+j,l) \leq \varepsilon_l\|g^k\|_2^2, \qquad \forall\,j\geq m_l.
\]
Then there exists \(j_0 \in [m_l,\,m_l+\Delta_l+1]\) such that
\[
    (d_{l+1}^{k+j_0})^2 \leq s\,\varepsilon_l\|g^k\|_2^2.
\]
\end{lemma}

Lemma~\ref{lem:sweeping_l_plus_1} formalizes the backward-tracing argument. For an applied step-size, the corresponding refresh index and anchor gradient remain within a bounded distance of the target iteration. The information lag is therefore absorbed into the window length \(\Delta_l\), rather than accumulating as an unbounded perturbation.

\begin{lemma}
\label{lem:induction_l_plus_1}
Under the hypotheses of Lemma~\ref{lem:sweeping_l_plus_1}, define
\begin{equation}
\label{eq:eps_induction}
    \varepsilon_{l+1} := \left( 1+s\,\delta^{2(\tilde{m}+1)} \right)\varepsilon_l,
    \qquad m_{l+1} := m_l+\Delta_l+1.
\end{equation}
Then
\[
    D(k+j,l+1) \leq \varepsilon_{l+1}\|g^k\|_2^2, \qquad \forall\,j\geq m_{l+1}.
\]
\end{lemma}

Lemma~\ref{lem:induction_l_plus_1} converts the pointwise suppression guaranteed by Lemma~\ref{lem:sweeping_l_plus_1} into a uniform bound for all subsequent iterations. The factor \(\delta^{2(\tilde{m}+1)}\) accounts for the largest possible amplification before the anchor associated with an applied step-size is guaranteed to lie after the most recent suppression event.

With the convention \(D(k,0)=0\), the induction starts at \(l=0\). The initial tolerance is chosen sufficiently small that all intermediate values \(\varepsilon_l\) remain in \((0,\gamma/4]\) and the terminal value satisfies \(\varepsilon_n\leq1/2\). Successive applications of Lemma~\ref{lem:induction_l_plus_1} then produce a uniform partial-energy bound for all \(n\) components.

\begin{lemma}
\label{lem:global_halving}
There exists a positive integer \(M\), depending only on \(n\), \(\delta\), \(\gamma\), \(c\), and \(\tilde{m}\), such that
\[
    \|g^{k+j}\|_2^2 \leq \frac12\|g^k\|_2^2, \qquad \forall\,j\geq M,\quad k\geq m.
\]
\end{lemma}

\begin{proposition}[\(R\)-linear convergence]
\label{prop:r_linear_convergence}
Let \(\{x^k\}\) be the sequence generated by Algorithm~\ref{alg:adaptive_hanoi_main} when applied to problem~\eqref{eq:intro-quadratic}. Then \(\{x^k\}\) converges \(R\)-linearly to the unique minimizer \(x^*\). More precisely, there exist constants \(C>0\) and \(\rho\in(0,1)\), depending only on \(A\), \(\gamma\), \(m\), and \(\kappa_{\mathrm{cont}}\), such that
\[
    \|x^k-x^*\|_2 \leq C\rho^k, \qquad \forall\,k\geq0.
\]
\end{proposition}

\section{Numerical Experiments}
\label{sec:numerical-experiments}

We compare two variants of the proposed rebound-aware spectral gradient method: \textbf{RASG-KT}, which uses the Kato--Temple-inspired estimator of Ritz uncertainty, and \textbf{RASG-LMSD}, which uses the long/short LMSD discrepancy. The comparison includes five representative methods: BB1, NY~\cite{xie2026new}, HDL~\cite{huang2021equipping}, MLMSD~\cite{gu2021modified}, and L-BFGS. The implementations and parameter settings are specified below.

\begin{itemize}
\item \textbf{NY}~\cite{xie2026new}: a cyclic gradient method with cycle length \(T=7\). For general convex problems, each cycle samples three approximated Cauchy steps \(\alpha^{\mathrm{ASD}}\) at phases \(0,1,2\) via quadratic interpolation~\cite{xie2025adaptive}, and recovers the common step-size \(\alpha^{\mathrm{ANY}}\) through Cardano's formula applied to the resulting cubic. The remaining iterations reuse \(\alpha^{\mathrm{ANY}}\) until the next refresh. On the log-determinant problem only, we replace the paper's initial step \(\beta_0=1/\|g^k\|_\infty\) by \(\beta_0=10/\|g^k\|_2\); without this rescaling the method stalls on that problem. All other experiments use the paper's setting.
    \item \textbf{HDL}~\cite{huang2021equipping}: the special HDL step-size \(\alpha^{\mathrm{new}}\) from Example~\ref{ex:special-Huang}. At each iteration, the method computes \(\alpha^{\mathrm{BB1}}\) and \(\alpha^{\mathrm{BB2}}\), and switches between the long step \(\alpha^{\mathrm{BB1}}\) and the short step \(\min\{\alpha^{\mathrm{BB2}},\alpha^{\mathrm{new}}\}\) according to whether \(\alpha^{\mathrm{BB2}}/\alpha^{\mathrm{BB1}}<\tau\). For quadratic problems, we set \(\tau=0.5\), as suggested in the original paper. For general convex problems, we follow Algorithm~3.1 of~\cite{huang2021equipping}.
    \item \textbf{MLMSD}~\cite{gu2021modified}: the modified limited-memory steepest-descent method with memory size \(m\), reference-gradient index \(\mathrm{ind}=m\), and the admissibility threshold in \eqref{eq:admissibility-threshold}. We test \(m\in\{2,3\}\) for quadratic problems and \(m=2\) for general convex problems.
    \item \textbf{L-BFGS}: the standard limited-memory BFGS method with memory \(m\in\{2,3\}\) for quadratic problems and \(m=2\) for general convex problems. In the quadratic experiments, the two-loop recursion is initialized with unit scaling and followed by a unit step, providing a gradient-only comparison without additional line-search evaluations. The unit step setting follows~\cite{fletcher2012limited}. In the general convex experiments, L-BFGS is equipped with the same GLL nonmonotone line search used for the other methods.
    \item \textbf{RASG-KT / RASG-LMSD}: both variants use memory \(m\in\{2,3\}\) for quadratic problems and \(m=2\) for general convex problems. The parameters \(r_\theta,r_t\) weight spectral frequency against component strength in~\eqref{eq:weighted-component-energy}; the exponent \(r\) controls how strictly high-frequency candidates are settled in~\eqref{eq:hanoi-threshold}; the cap \(\kappa_{\mathrm{cont}}\) bounds how long a selected step-size is reused; and \(\mathrm{ind_r},\mathrm{ind_s}\) select the history positions used for ranking and settlement. The ranking score is \(\eta_i^k=(\theta_i^k)^2\,t_i^{k,\mathrm{ind}_r}\) in all experiments. Under the usual no-breakdown condition, we use
    \begin{itemize}
        \item RASG-LMSD: \(r_\theta=2\), \(r_t=1\), \(r=2\), \(\mathrm{ind}_r=\mathrm{ind}_s=m\), and \(\kappa_{\mathrm{cont}}=3\);
        \item RASG-KT: \(r_\theta=2\), \(r_t=1\), \(r=2\), \(\mathrm{ind}_r=m\), \(\mathrm{ind}_s=m+1\), and \(\kappa_{\mathrm{cont}}=2\).
    \end{itemize}
\end{itemize}

\subsection{Quadratic case}

\noindent\textit{Per-iteration complexity.}
Table~\ref{tab:complexity} summarizes the dominant arithmetic overhead per iteration beyond the shared matrix--vector product \(Ag^k\) and the \(O(n)\) vector operations common to the methods under consideration.

\begin{table}[t]
\caption{Arithmetic overhead per iteration beyond the shared matrix--vector product \(Ag^k\) and \(O(n)\) vector operations. NY is equivalent to the scalar-moment LMSD construction with \(m=3\); see Remark~\ref{rem:yuan-special-case}. Its \(O(m^3)\) Hankel solve reduces to \(O(1)\) because \(m=3\) is fixed. The \(O(m^3)\) reduced-model costs of RASG-KT and RASG-LMSD are amortized over a continuation phase and become \(O(m^3/\ell_k)\) per gradient step when \(\ell_k>1\).}
\label{tab:complexity}
\centering
\begin{tabular}{lc}
\toprule
\textbf{Method} & \textbf{Additional cost per iteration} \\
\midrule
BB1                         & $O(1)$       \\
NY                          & $O(1)$       \\
HDL                         & $O(1)$       \\
MLMSD (limited-memory)      & $O(nm^2)$    \\
MLMSD (pseudo-memory)       & $O(m^3)$     \\
RASG-KT                     & $O(m^3)$     \\
RASG-LMSD                   & $O(m^3)$     \\
L-BFGS                      & $O(mn)$      \\
\bottomrule
\end{tabular}
\end{table}

For \(m\leq3\) and \(n=1000\), the dense-problem cost hierarchy is
\[
    O(m^3) \ll O(mn) \leq O(nm^2) \ll O(n^2).
\]
The scalar-moment overhead is therefore negligible relative to a dense matrix--vector product. For sparse problems with \(\operatorname{nnz}(A)=O(n)\), the \(O(mn)\) operations of L-BFGS are comparable to the matrix--vector product, whereas the \(O(m^3)\) reduced-model cost remains independent of the ambient dimension.

\medskip
\noindent\textit{Test problems.}
Following Dai and Yuan~\cite{dai2003alternate}, each problem instance is generated as
$A=QDQ^\top$ with
$Q=(I-2w_1w_1^\top)(I-2w_2w_2^\top)(I-2w_3w_3^\top)$,
where \(w_1,w_2,w_3\in\mathbb{R}^n\) are independent unit vectors drawn uniformly from the sphere, \(D=\operatorname{diag}(d_1,\ldots,d_n)\), and \(b_i\sim\mathrm{Uniform}[-10,10]\). The initial point is sampled according to \(x^0\sim\mathcal{N}(0,I)\). All experiments use dimension \(n=1000\), condition number \(\mathrm{cond}=10^6\), and \(100\) independently generated instances for each problem class.

The four problem classes differ in the distribution of the eigenvalues \(d_1,\ldots,d_n\). In each case, the endpoints are fixed as \(d_1=1\), \(d_n=\mathrm{cond}\).
\begin{description}
    \item[Problem~1 (Uniform) \cite{de2014efficient,Gonzaga_2015,sun2020new,gu2021modified}.] For \(1<i<n\), \(d_i\sim\mathrm{Uniform}[1,\mathrm{cond}]\).
    \item[Problem~2 (Marchenko--Pastur) \cite{Marenko_1967,gu2021modified}.] Set \(a=(1-c)^2\), \(b=(1+c)^2\), \(c=\frac12\). For \(1<i<n\),
    \[
        d_i = \frac{d_1b-d_na}{b-a} + \frac{d_n-d_1}{b-a}\xi_i, \qquad
        \xi_i\sim p_c(x) = \frac{\sqrt{(b-x)(x-a)}}{2\pi xc^2}.
    \]
    This construction follows the limiting spectral distribution of sample covariance matrices.
    \item[Problem~3 (Two-block bimodal) \cite{di2018steplength,sun2020new,gu2021modified}.] For \(1<i<n\),
    \[
        d_i \sim \begin{cases}
            \mathrm{Uniform}\bigl[1,\,1+0.2(\mathrm{cond}-1)\bigr], & 1<i\leq\lfloor n/2\rfloor, \\[4pt]
            \mathrm{Uniform}\bigl[1+0.8(\mathrm{cond}-1),\,\mathrm{cond}\bigr], & \lfloor n/2\rfloor<i<n.
        \end{cases}
    \]
    The eigenvalues form two separated groups near the endpoints of the spectral interval.
    \item[Problem~4 (Two-block asymmetric) \cite{gu2021modified}.] For \(1<i<n\),
    \[
        d_i \sim \begin{cases}
            \mathrm{Uniform}[1,100], & 1<i\leq\lfloor n/2\rfloor, \\[4pt]
            \mathrm{Uniform}[\mathrm{cond}/2,\,\mathrm{cond}], & \lfloor n/2\rfloor<i<n.
        \end{cases}
    \]
    The lower block remains fixed while the upper block scales with \(\mathrm{cond}\), producing a strongly asymmetric spectral gap.
\end{description}

\medskip
\noindent\textit{Results.}
Performance is measured by the number of iterations required to satisfy \(\|g^k\|_2/\|g^0\|_2 \leq \epsilon\), \(\epsilon=10^{-8}\). A run is classified as unsuccessful if this criterion is not attained within \(50{,}000\) iterations. For the construction of the performance profiles, unsuccessful runs are assigned an iteration count of \(500{,}000\), placing them beyond the displayed range. Figures~\ref{fig:pp-m2}--\ref{fig:pp-all} report the performance profiles of Dolan and Mor\'{e}~\cite{dolan2002benchmarking} over the \(100\) independently generated instances of each problem class.

\begin{figure}[!htbp]
    \centering
    \includegraphics[width=0.48\textwidth]{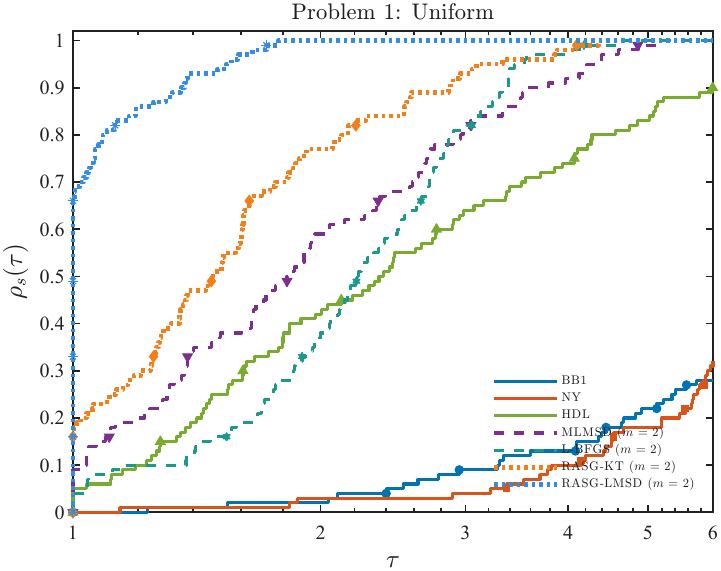}\hfill
    \includegraphics[width=0.48\textwidth]{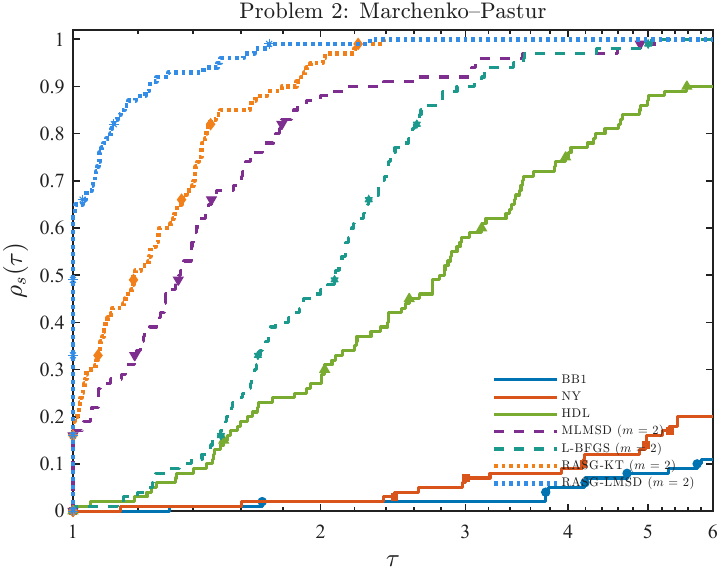}\\[4pt]
    \includegraphics[width=0.48\textwidth]{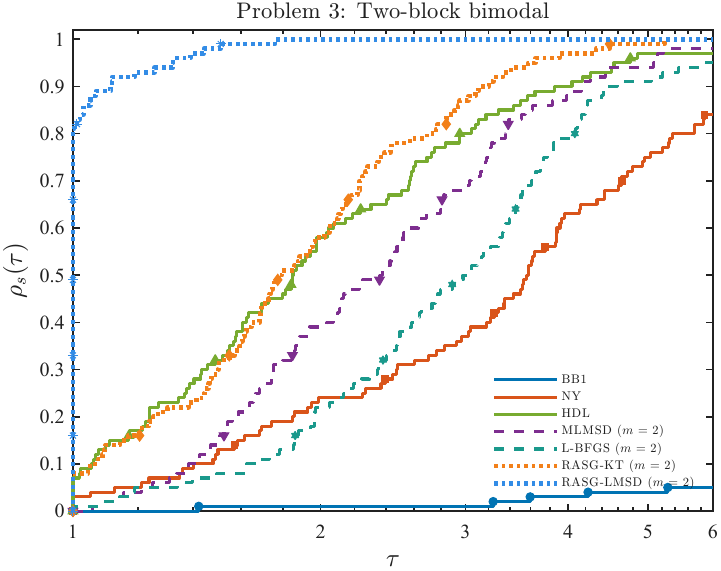}\hfill
    \includegraphics[width=0.48\textwidth]{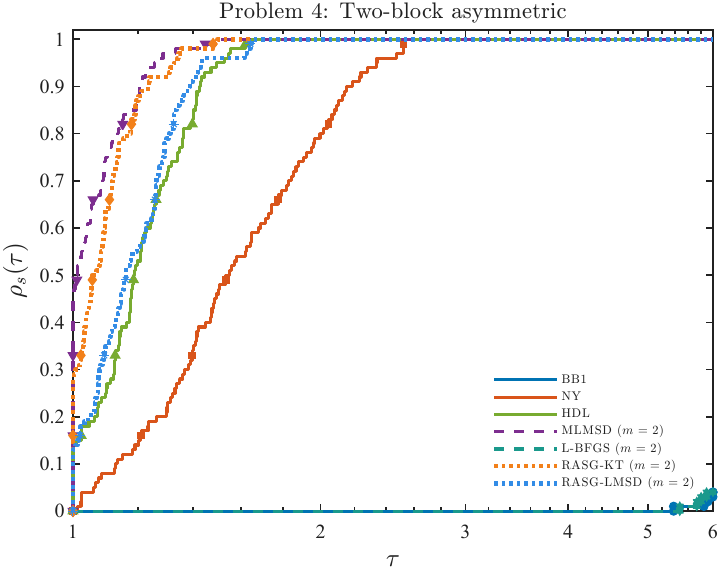}
    \caption{Performance profiles for Problems~1--4 with \(m=2\). The comparison includes BB1, NY, HDL, MLMSD, RASG-KT, RASG-LMSD, and L-BFGS. The stopping tolerance is \(\epsilon=10^{-8}\), and the performance-ratio axis is displayed on a logarithmic scale.}
    \label{fig:pp-m2}
\end{figure}

\begin{figure}[!htbp]
    \centering
    \includegraphics[width=0.48\textwidth]{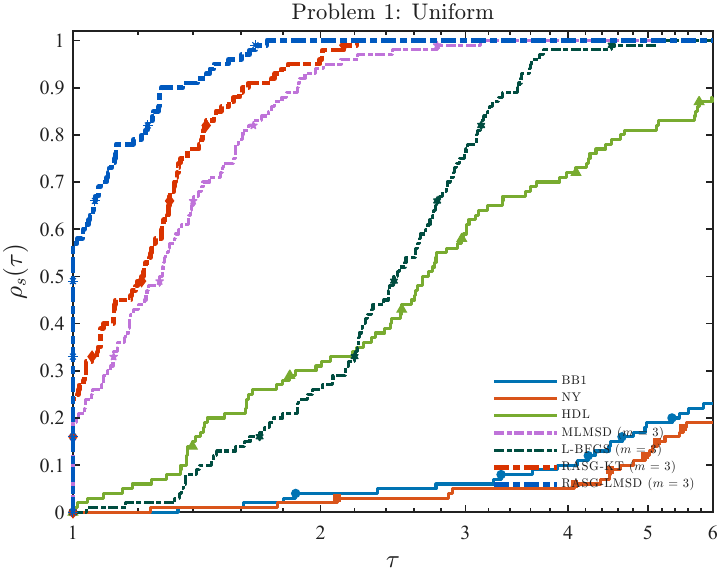}\hfill
    \includegraphics[width=0.48\textwidth]{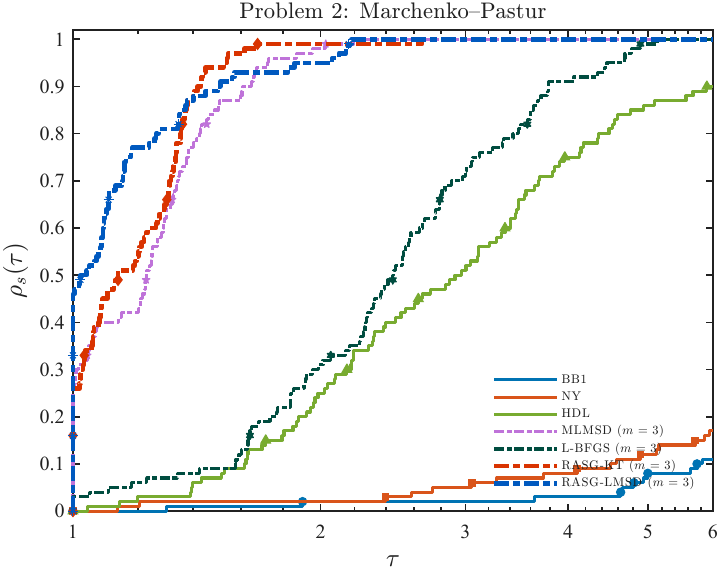}\\[4pt]
    \includegraphics[width=0.48\textwidth]{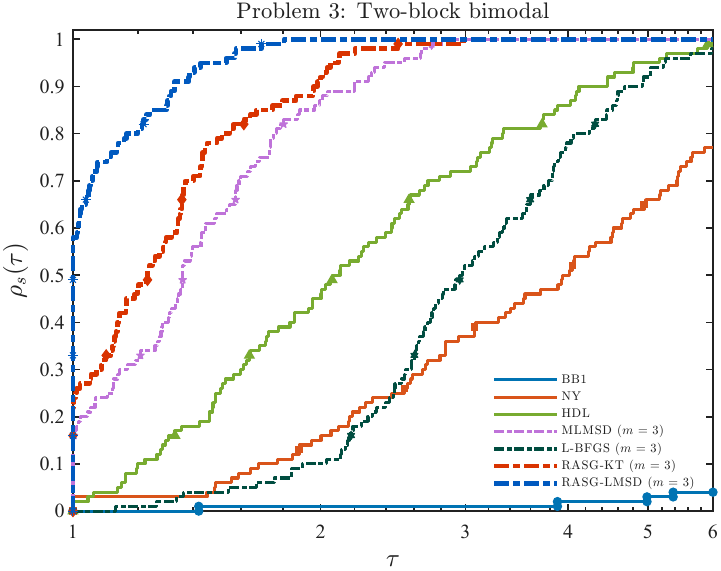}\hfill
    \includegraphics[width=0.48\textwidth]{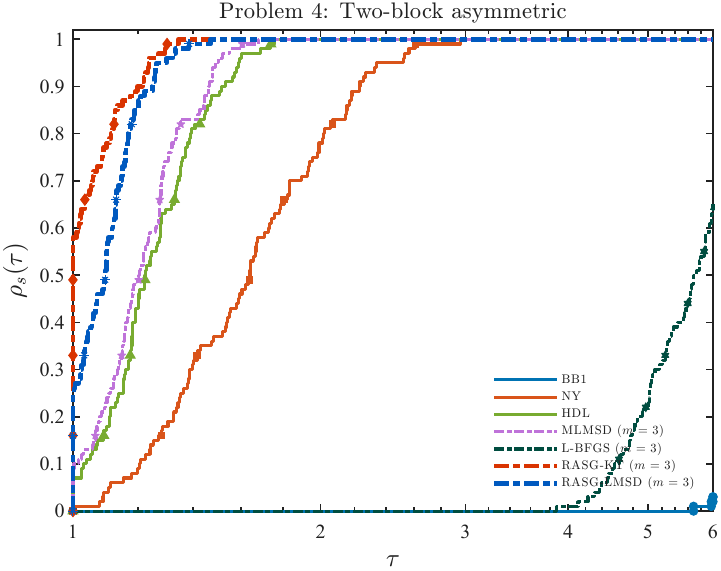}
    \caption{Performance profiles for Problems~1--4 with \(m=3\) under the setting of Figure~\ref{fig:pp-m2}. The comparison includes BB1, NY, HDL, MLMSD, RASG-KT, RASG-LMSD, and L-BFGS with memory \(m=3\).}
    \label{fig:pp-m3}
\end{figure}

\begin{figure}[!htbp]
    \centering
    \includegraphics[width=0.48\textwidth]{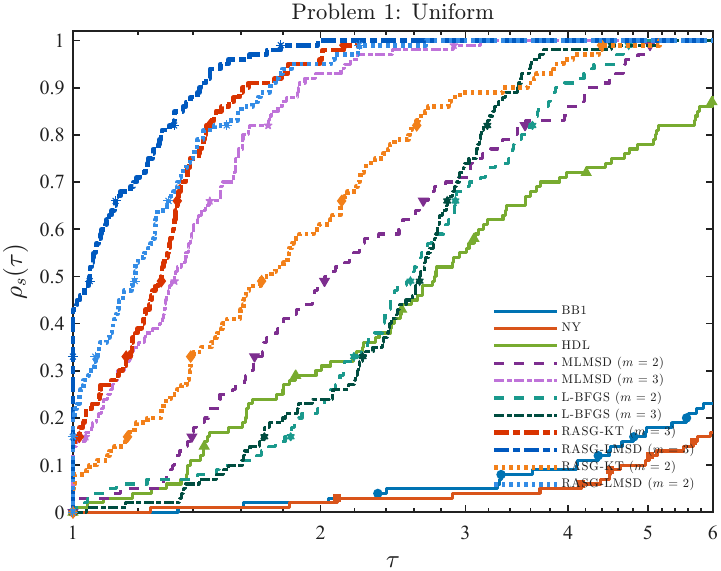}\hfill
    \includegraphics[width=0.48\textwidth]{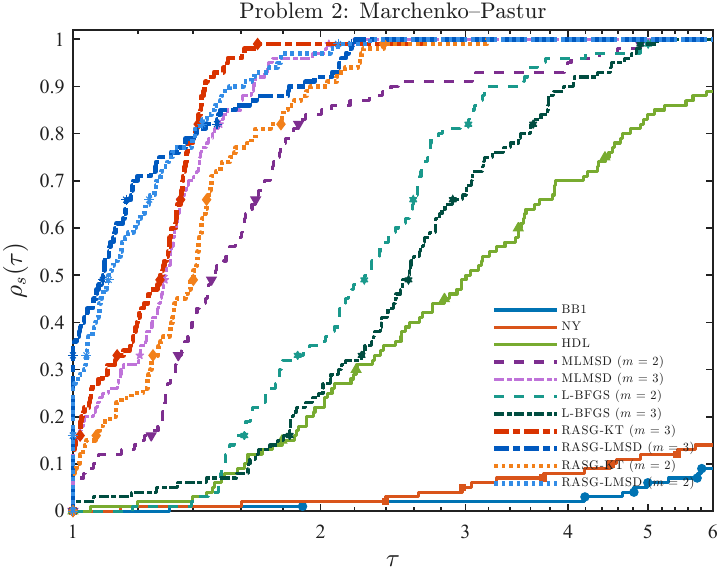}\\[4pt]
    \includegraphics[width=0.48\textwidth]{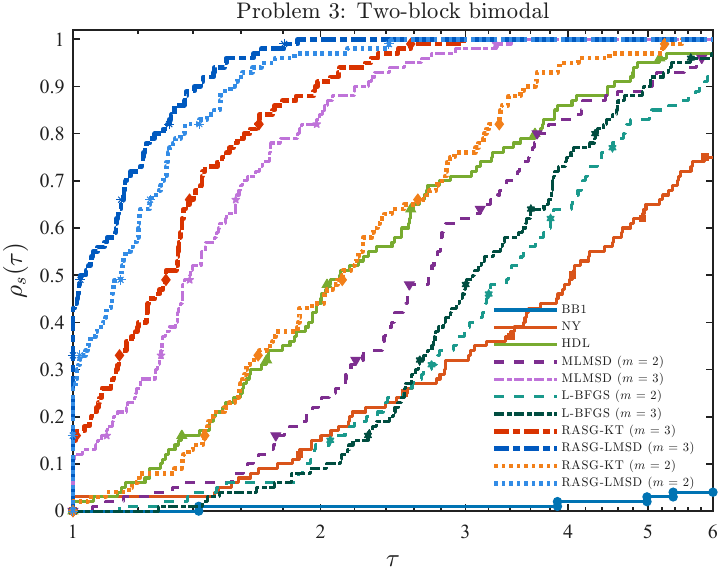}\hfill
    \includegraphics[width=0.48\textwidth]{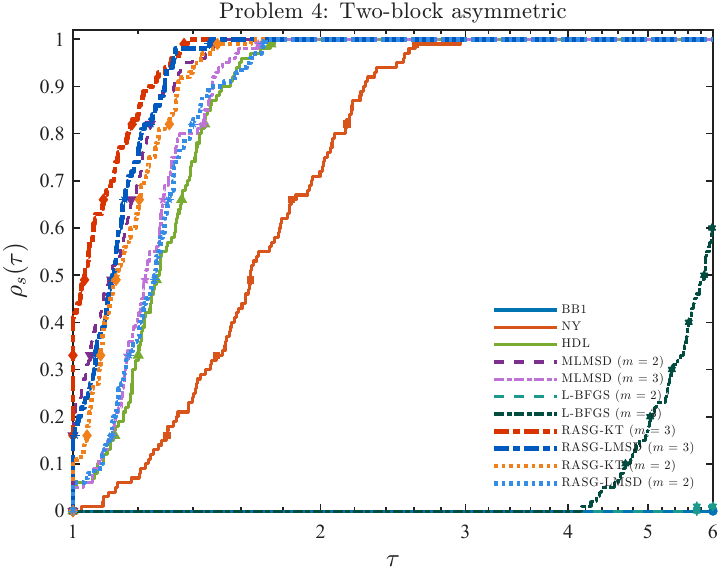}
    \caption{Performance profiles for Problems~1--4 including all \(11\) method--memory configurations under the setting of Figure~\ref{fig:pp-m2}.}
    \label{fig:pp-all}
\end{figure}

For \(m=2\), Figure~\ref{fig:pp-m2} shows that both RASG variants outperform BB1, NY, HDL, and L-BFGS across the four tested spectral distributions. RASG-LMSD attains the highest profile on Problems~1--3, followed by RASG-KT. On Problem~4, MLMSD attains the largest profile value; both RASG variants nevertheless solve all instances within approximately \(1.5\) times the iteration count of the best solver.

For \(m=3\), Figure~\ref{fig:pp-m3} shows that RASG-LMSD and RASG-KT again attain the two strongest profiles on Problems~1--3. On Problem~4, RASG-KT ranks first and RASG-LMSD ranks second, with both variants outperforming MLMSD at the same memory size. The aggregate profiles in Figure~\ref{fig:pp-all} indicate that the performance of rebound-aware selection is robust across the tested spectral distributions and generally improves when the reduced-model dimension is increased from \(m=2\) to \(m=3\).

\subsection{General convex case}
\label{sec:convex-case}

The comparison is extended to two large-scale convex problems. The theoretical analysis of Section~\ref{sec:convergence_analysis} concerns strictly convex quadratics, where the finite-moment identities hold exactly; for general convex objectives these identities become local approximations, and the pseudo-memory pipeline is used as a heuristic reduced model. The experiments here explore whether the rebound-aware ordering principle retains any advantage in this regime. All methods are equipped with the GLL nonmonotone line search~\cite{grippo1986nonmonotone}, using \(\sigma=10^{-4}\), \(\beta=0.5\), \(M=10\), and at most \(50\) backtracking reductions per iteration. For every method and test problem, the trial step-size is restricted to \([\alpha_{\min},\alpha_{\max}]=[10^{-10},10^6]\). The stopping criterion is \(\|g^k\|_2/\|g^0\|_2\leq10^{-6}\), and all memory-based methods use \(m=2\). A run is terminated if the prescribed time limit is reached before the stopping criterion is satisfied.

For nonquadratic objectives, the fixed-Hessian moment identities
underlying the quadratic construction no longer hold exactly.
Following the pseudo-memory implementation of MLMSD, we construct
an approximation to the small gradient Gram matrix
\((G^k)^\top G^k\) associated with the current history window. Its
diagonal entries and consecutive-gradient inner products are
obtained directly from the rolling scalar-product window, while the
unavailable nonconsecutive-gradient inner products are estimated
using the pseudo-memory formulas and the recent accepted GLL
step-sizes. These formulas are exact for quadratic objectives and
are used as local approximations for general smooth objectives.

The resulting small LMSD model is processed by the same Cholesky
factorization and SVD procedure used in MLMSD~\cite{gu2021modified}. Its
singular values and right singular vectors are used as empirical
reduced spectral candidates and component directions, respectively.
The admissibility filter, ranking score, settlement threshold, and
continuation cap are then applied to these reduced quantities, and
the resulting trial step is passed to the GLL nonmonotone line
search. For general smooth objectives, no exact Ritz or
finite-moment interpretation is asserted, and we do not establish a
separate convergence theorem for this extension. All three methods
(RASG-KT, RASG-LMSD, and MLMSD) are implemented in their
pseudo-memory form.

The two test problems have distinct computational and structural characteristics.

\medskip
\noindent\textbf{Problem 1: Weighted exponential} \cite{Raydan_1997,gu2021modified}. We consider the shifted objective
\[
    \min_{x\in\mathbb{R}^n} \sum_{i=1}^{n} w_i\bigl(\exp(x_i)-x_i-1\bigr),
\]
where \(w_i=i/10\), \(i=1,\ldots,n\), \(n=10^6\), and \(x^0=\mathbf{1}\). The additive constant does not affect the gradient or the generated iterates. The unique minimizer is \(x^*=\mathbf{0}\), with \(f(x^*)=0\).

\medskip
\noindent\textbf{Problem 2: Log-determinant optimization with Frobenius-norm regularization} \cite{lam2009sparsistency,van2016ridge}. We consider
\[
    \min_{X\succ0} -\log\det(X) + \operatorname{tr}(SX) + \lambda\|X\|_F^2,
\]
where \(X\in\mathbb{R}^{p\times p}\) is symmetric positive definite, \(p=1000\), and the scalar dimension is \(n=p^2=10^6\). We set \(\lambda=10^{-8}\), \(X^0=I\). The matrix \(S\) is generated as \(S=Q\operatorname{diag}(s_i)Q^\top\), where \(Q\) is orthogonal and the values \(s_i\) are independently sampled from the uniform distribution on \([0.1,1]\). The optimality condition \(-X^{-1}+S+2\lambda X=0\) diagonalizes in the eigenbasis of \(S\), giving the per-eigenvalue quadratic \(2\lambda x_i^2+s_ix_i-1=0\). The unique positive root \(x_i=(-s_i+\sqrt{s_i^{\,2}+8\lambda})/(4\lambda)\) determines the exact minimizer \(X^*=Q\operatorname{diag}(x_i)Q^\top\), which is used to evaluate \(f^*\) in the objective-gap panels.

\begin{figure}[!htbp]
\centering
\begin{minipage}[b]{0.48\textwidth}
  \centering
  \includegraphics[width=\textwidth]{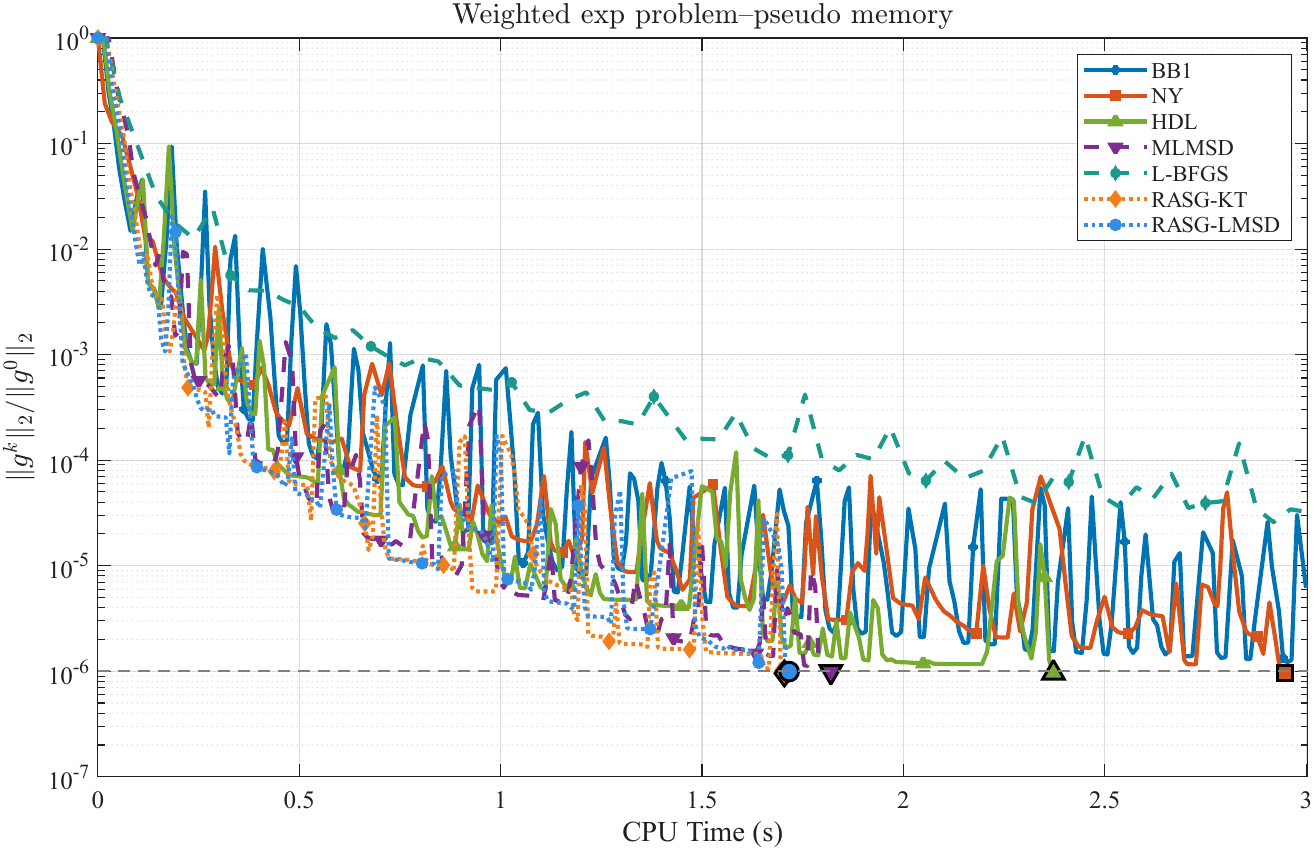}
  \subcaption{ \(\|g^k\|_2/\|g^0\|_2\).}
  \label{fig:weighted_pm_gnorm}
\end{minipage}
\hfill
\begin{minipage}[b]{0.48\textwidth}
  \centering
  \includegraphics[width=\textwidth]{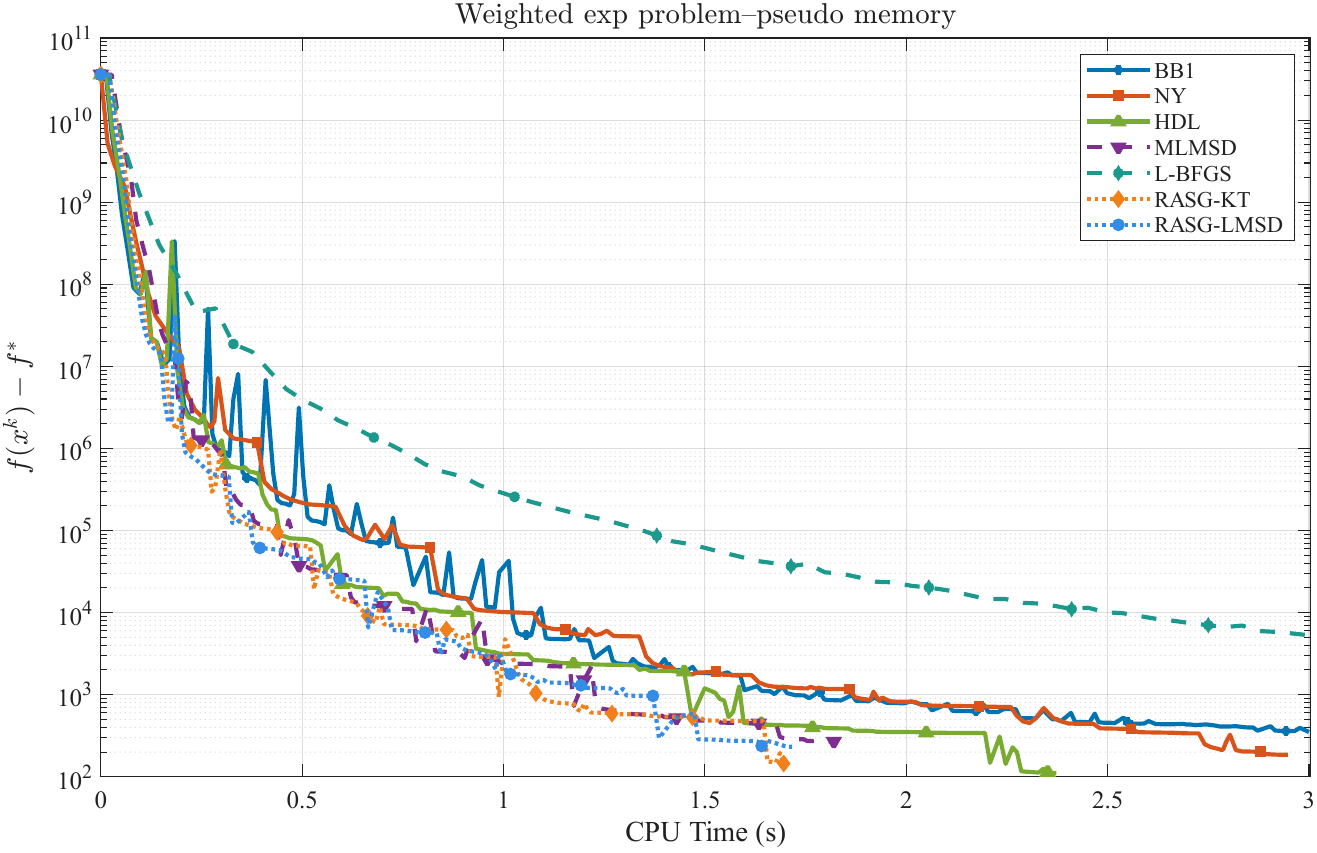}
  \subcaption{ \(f(x^k)-f^*\).}
  \label{fig:weighted_pm_gap}
\end{minipage}
\caption{CPU-time convergence histories for the weighted exponential problem with \(n=10^6\).}
\label{fig:weighted}
\end{figure}

\begin{figure}[!htbp]
\centering
\begin{minipage}[b]{0.48\textwidth}
  \centering
  \includegraphics[width=\textwidth]{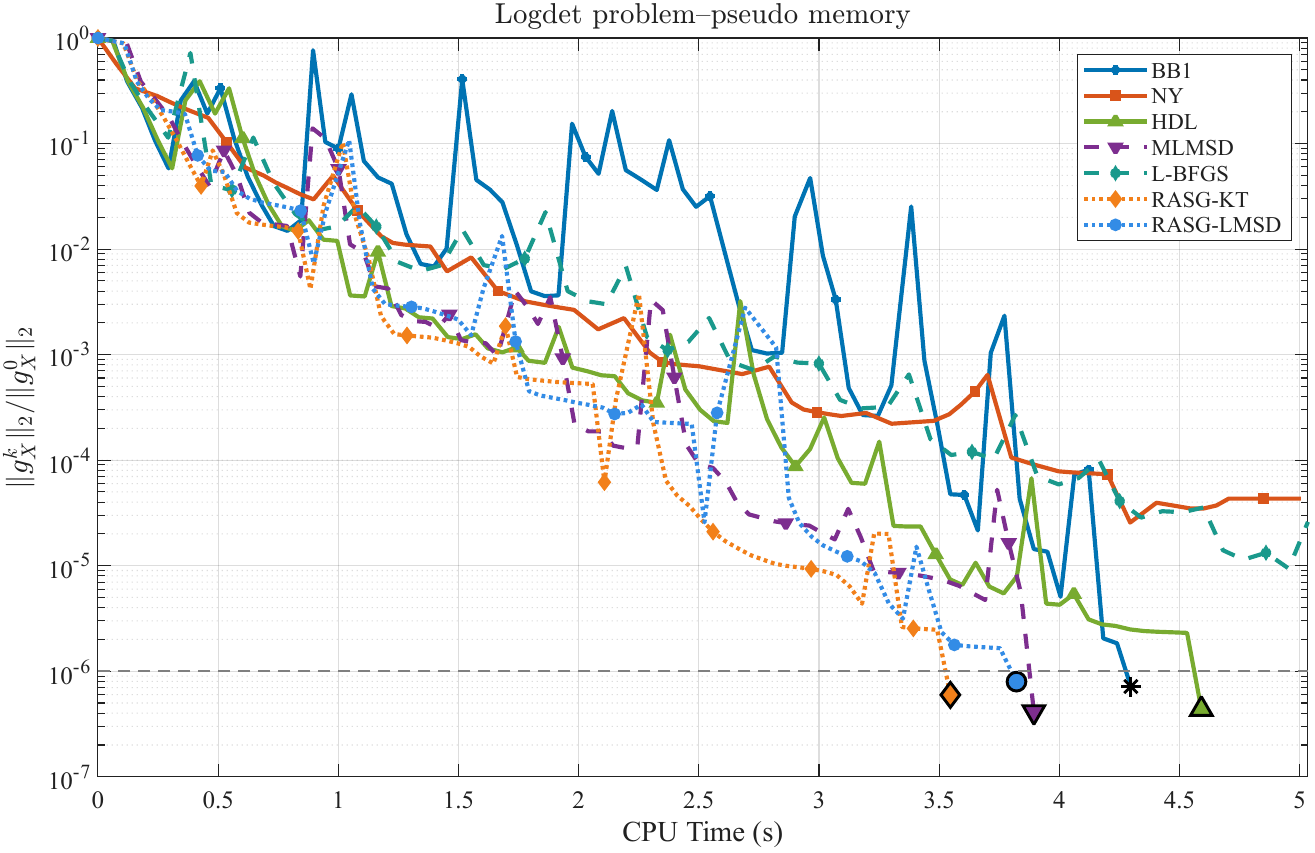}
  \subcaption{ \(\|g^k\|_2/\|g^0\|_2\).}
  \label{fig:logdet_pm_gnorm}
\end{minipage}
\hfill
\begin{minipage}[b]{0.48\textwidth}
  \centering
  \includegraphics[width=\textwidth]{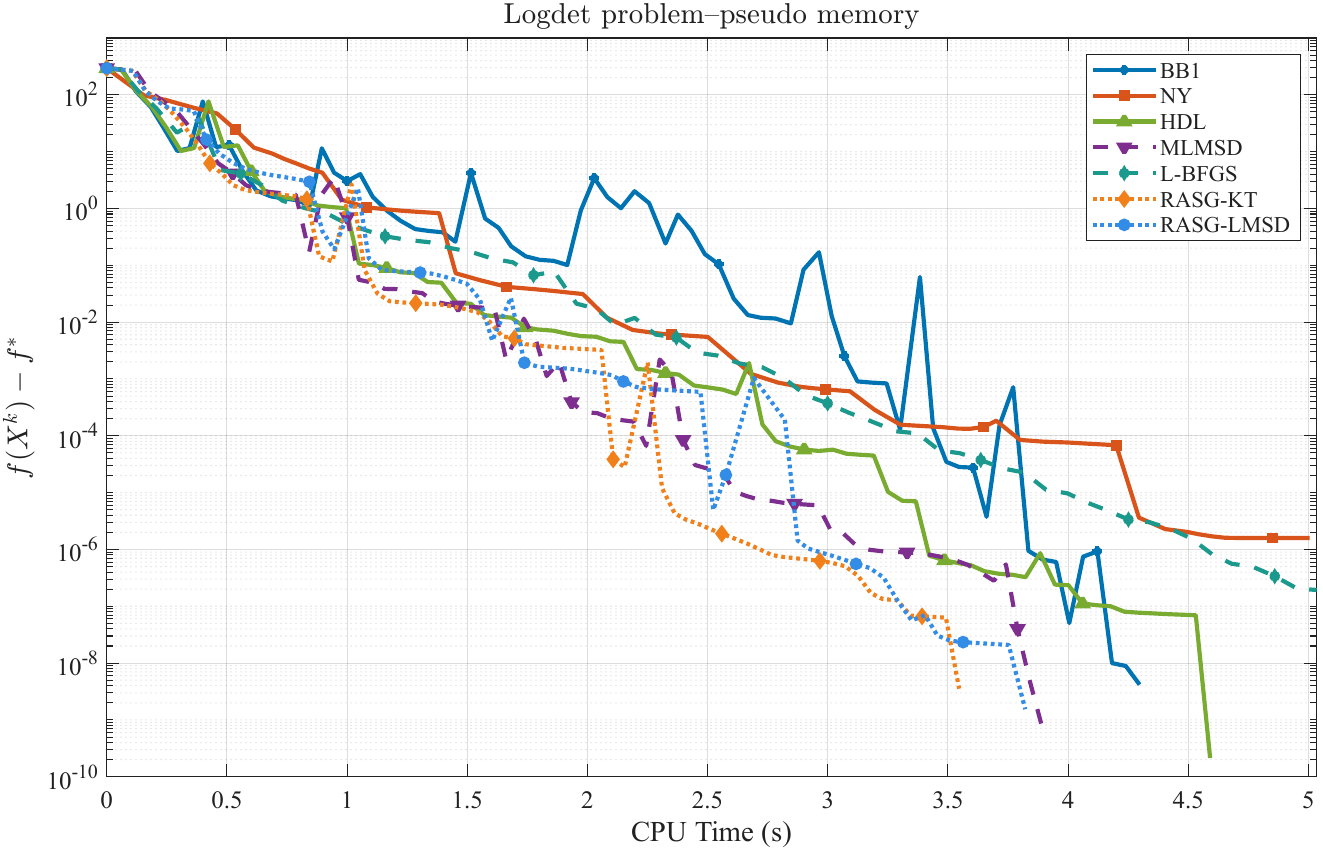}
  \subcaption{\(f(X^k)-f^*\).}
  \label{fig:logdet_pm_gap}
\end{minipage}
\caption{CPU-time convergence histories for the regularized log-determinant problem with \(p=1000\) and scalar dimension \(n=10^6\).}
\label{fig:logdet}
\end{figure}

Figures~\ref{fig:weighted} and~\ref{fig:logdet} report the CPU-time convergence histories for the two general convex problems. In each figure, the left panel reports the relative gradient norm and the right panel the objective gap. Both figures compare the pseudo-memory realizations of RASG-KT, RASG-LMSD, and MLMSD.

For the weighted exponential problem, each gradient evaluation has \(O(n)\) complexity. In the pseudo-memory comparison, RASG-LMSD and RASG-KT attain the shortest reported solution times. This reduction is consistent with replacing the dimension-dependent \(O(nm^2)\) reduced-space operations of the explicit-history implementation by the \(O(m^3)\) scalar-moment computation. For this high-dimensional problem, the reduced algebraic overhead compensates for the approximation introduced by the pseudo-memory realization.

For the log-determinant problem, evaluation of the gradient requires a matrix factorization and has \(O(p^3)=O(n^{3/2})\) complexity. The pseudo-memory RASG variants attain the shortest reported solution times and reach small final residuals in this experiment.
\section{Conclusion}
\label{sec:conclusion}

This work develops a structured framework for spectral step-size selection in gradient methods for strictly convex quadratic optimization. The analysis identifies a finite-moment connection among HDL determinant pencils, narrow realizations, and generalized Gu--Du recovery, placing the long, short, and mean LMSD constructions within a common spectral representation.

The component dynamics reveal a rebound mechanism whereby a step targeting a lower spectral scale may reactivate high-frequency components that have already been reduced. Under explicitly stated transition rules, this mechanism yields exact phase-count recurrences for one-target, block, and recomputed memory-\(m\) models. These recurrences characterize the idealized scheduling structure and motivate the practical ordering rule; they are not iteration-complexity bounds for the finite-precision adaptive method.

The resulting rebound-aware spectral gradient method, RASG, combines Ritz-candidate ranking, component-based settlement monitoring, and adaptive step-size reuse. Its two variants, RASG-KT and RASG-LMSD, differ in the estimator used to quantify Ritz uncertainty. Under the admissibility safeguard, the RASG iterates converge globally \(R\)-linearly for strictly convex quadratic objectives despite the bounded information lag introduced by multi-step reuse.

The numerical results indicate that rebound-aware scheduling is effective across the tested ill-conditioned spectral distributions. The large-scale convex experiments further suggest that the pseudo-memory realization can reduce dimension-dependent overhead when only a low-dimensional reduced model is retained. This nonquadratic extension remains empirical: the fixed-Hessian
moment identities underlying the quadratic analysis no longer hold
exactly, and the relation between the pseudo-memory reduced model
and local curvature has yet to be characterized.

\appendix
\section{Pseudo-memory moment recovery and residual evaluation}
\label{app:pseudo-memory}

This appendix gives a moment--Hankel realization of the
pseudo-memory representation introduced in
Section~\ref{subsec:moment-hankel-narrow}. It shows that the moment
window \(h_0^k,\ldots,h_{2m}^k\) can be reconstructed from scalar
products of consecutive gradients and the recent step-sizes, so
that the full gradient block
\(G^k=[g^{k-m},\ldots,g^{k-1}]\) need not be stored. Apart from the
current iterate, this realization retains only the previous and
current gradients, together with \(O(m)\) scalar products and
step-sizes.

Let \(w^k=g^{k-m}\) and \(h_j^k=(w^k)^\top A^jw^k\) for \(j\geq0\),
and define the history polynomials
\(p_0^k(x)=1\) and
\(p_i^k(x)=\prod_{t=0}^{i-1}(1-\alpha_{k-m+t}x)\) for \(i\geq1\).
The quadratic recurrence \(g^{k-m+i}=p_i^k(A)w^k\) implies that
\((g^{k-m+i})^\top g^{k-m+j}=(w^k)^\top p_i^k(A)p_j^k(A)w^k\)
for \(0\leq i,j\leq m\); since \(p_i^kp_j^k\) has degree at most
\(i+j\), every such scalar product is a linear combination of
\(h_0^k,\ldots,h_{i+j}^k\).

The \(2m+1\) scalar products
\begin{equation}
\label{eq:app-q-vector}
    q^k
    =
    \bigl(
        \|g^{k-m}\|_2^2,\,
        (g^{k-m+1})^\top g^{k-m},\,
        \|g^{k-m+1}\|_2^2,\,
        \ldots,\,
        (g^k)^\top g^{k-1},\,
        \|g^k\|_2^2
    \bigr)^\top
\end{equation}
satisfy \(q^k=S^kh^k\), where
\(h^k=(h_0^k,\ldots,h_{2m}^k)^\top\) and
\(S^k\in\mathbb{R}^{(2m+1)\times(2m+1)}\) is lower triangular.
Writing \([x^j]\phi(x)\) for the coefficient of \(x^j\) in a
polynomial \(\phi\), the entries of \(S^k\) are
\(S_{1,j+1}^k=[x^j]\,1\) and, for \(i=1,\ldots,m\),
\[
    S_{2i,j+1}^k=[x^j]\,p_i^k(x)p_{i-1}^k(x),
    \qquad
    S_{2i+1,j+1}^k=[x^j]\,\bigl(p_i^k(x)\bigr)^2,
\]
for \(j=0,\ldots,2m\). Since the polynomial associated with row
\(r\) has degree at most \(r-1\), \(S_{r,j}^k=0\) whenever \(j>r\),
so \(S^k\) is lower triangular.

The rows of \(S^k\) may be generated recursively. Setting
\(S_{r,1}^k=1\) for \(r=1,\ldots,2m+1\) and zero outside the active
lower-triangular range, then for \(i=1,\ldots,m\),
\begin{equation}
\label{eq:app-S-recursion}
\begin{aligned}
    S_{2i,j}^k
    &=
    S_{2i-1,j}^k
    -
    \alpha_{k-m+i-1}S_{2i-1,j-1}^k,
    && j=2,\ldots,2i,\\
    S_{2i+1,j}^k
    &=
    S_{2i,j}^k
    -
    \alpha_{k-m+i-1}S_{2i,j-1}^k,
    && j=2,\ldots,2i+1.
\end{aligned}
\end{equation}
Here row \(2i\) represents
\(p_i^k(x)p_{i-1}^k(x)=(1-\alpha_{k-m+i-1}x)(p_{i-1}^k(x))^2\),
and row \(2i+1\) represents
\((p_i^k(x))^2=(1-\alpha_{k-m+i-1}x)p_i^k(x)p_{i-1}^k(x)\).
The diagonal entries are the leading coefficients
\(S_{2i,2i}^k=-\alpha_{k-m+i-1}\prod_{t=0}^{i-2}\alpha_{k-m+t}^2\)
and
\(S_{2i+1,2i+1}^k=\prod_{t=0}^{i-1}\alpha_{k-m+t}^2\),
with the empty product interpreted as one. When the recent
step-sizes are nonzero, \(S^k\) is nonsingular and the moment vector
is recovered by forward substitution:
\begin{equation}
\label{eq:app-forward-substitution}
    h_0^k=q_1^k,
    \qquad
    h_{r-1}^k
    =
    \frac{
        q_r^k
        -
        \sum_{\ell=0}^{r-2}S_{r,\ell+1}^k h_\ell^k
    }{
        S_{r,r}^k
    },
    \qquad
    r=2,\ldots,2m+1.
\end{equation}
The right-hand side of \eqref{eq:app-forward-substitution} is formed
from scalar products of gradients rather than from the preceding
moment vector \(h^{k-1}\): at iteration \(k\), the newest entries
\((g^k)^\top g^{k-1}\) and \(\|g^k\|_2^2\) are computed from the two
retained gradients, while the remaining entries of \(q^k\) are
retained from the preceding window.

When the window advances by one iteration, the scalar-product vector
is updated as
\(q^{k+1}=(q_3^k,\ldots,q_{2m+1}^k,(g^{k+1})^\top g^k,\|g^{k+1}\|_2^2)^\top\),
so the rolling update requires only the previous and current
gradients, together with \(O(m)\) stored scalars. The moment vectors
associated with two consecutive anchors satisfy the consistency
identity
\[
    h_j^{k+1}
    =
    h_j^k
    -
    2\alpha_{k-m}h_{j+1}^k
    +
    \alpha_{k-m}^2h_{j+2}^k,
\]
which describes how the exact moment sequence changes when the
anchor is shifted. It is not used as the numerical update rule,
because recovering the full new window recursively would require
moments beyond \(h_{2m}^k\). Instead, the complete moment window can be reconstructed from the
updated scalar-product vector \(q^{k+1}\) through
\eqref{eq:app-forward-substitution}.

Once \(h_0^k,\ldots,h_{2m}^k\) have been recovered, define
\begin{equation}
\label{eq:app-H2-definition}
    H_0^k
    =
    \bigl[h_{r+s}^k\bigr]_{r,s=0}^{m-1},
    \qquad
    H_1^k
    =
    \bigl[h_{r+s+1}^k\bigr]_{r,s=0}^{m-1},
    \qquad
    H_2^k
    =
    \bigl[h_{r+s+2}^k\bigr]_{r,s=0}^{m-1}.
\end{equation}
The largest moment index required by \(H_2^k\) is \(2m\), so all
three matrices are available from the recovered moment window.
Writing \(K_m(w^k)=[\,w^k\ \ Aw^k\ \ \cdots\ \ A^{m-1}w^k\,]\),
one has
\[
    H_0^k=K_m(w^k)^\top K_m(w^k),
    \quad
    H_1^k=K_m(w^k)^\top A K_m(w^k),
    \quad
    H_2^k=K_m(w^k)^\top A^2 K_m(w^k).
\]

Let \((\theta,y)\) be a generalized eigenpair satisfying
\(H_1^k y=\theta H_0^k y\), and define the formal Ritz vector
\(s=K_m(w^k)y\), with squared norm
\(\|s\|_2^2=y^\top H_0^k y\). For the normalized Ritz vector
\(v=s/\|s\|_2\), the squared residual norm is
\begin{equation}
\label{eq:app-residual-formula}
    \|Av-\theta v\|_2^2
    =
    \frac{
        y^\top
        \bigl(
            H_2^k
            -
            2\theta H_1^k
            +
            \theta^2H_0^k
        \bigr)y
    }{
        y^\top H_0^k y
    }.
\end{equation}
Indeed, \(\|As-\theta s\|_2^2=y^\top(H_2^k-2\theta H_1^k+\theta^2H_0^k)y\),
and division by \(\|s\|_2^2=y^\top H_0^k y\) gives
\eqref{eq:app-residual-formula}. If the generalized eigenvector is
normalized so that \(y^\top H_0^k y=1\), then \(s\) is already a
unit Ritz vector and \eqref{eq:app-residual-formula} reduces to
\begin{equation}
\label{eq:app-residual-formula-normalized}
    \|As-\theta s\|_2^2
    =
    y^\top
    \bigl(
        H_2^k
        -
        2\theta H_1^k
        +
        \theta^2H_0^k
    \bigr)y.
\end{equation}
Thus, the Ritz residual required by the Kato--Temple-inspired
estimator can be evaluated entirely from the recovered scalar
moments, without forming or storing the full Ritz vector.

Under the same normalization, the anchor component weight
introduced in \eqref{eq:component-energy} admits the
pseudo-memory representation
\begin{equation}
\label{eq:app-anchor-weight}
    t_i^{k,1}
    =
    \bigl(w^k{}^\top s_i\bigr)^2
    =
    \Bigl(\sum_{j=0}^{m-1}h_j^k\,y_{i,j}\Bigr)^2
    =
    \bigl(e_1^\top H_0^k y_i\bigr)^2,
\end{equation}
which is used to evaluate the generalized component
coefficients \(t_i^{k,\mathrm{ind}}\) in
\eqref{eq:component-energy-general}.

\section{Derivations and Proofs for Section~\ref{sec:general-framework}}
\label{app:gudu-huang-proofs}

\subsection{Yuan and NY step-sizes as Krylov--Ritz instances}
\label{subsec:app-yuan}

\noindent\textbf{Yuan step-size for \(m=2\).}
Consider two consecutive gradients generated by a steepest-descent step \(x^1=x^0-\alpha_0^{\mathrm{SD}}g^0\). Let \(Q\) be an orthonormal basis of \(\operatorname{span}\{g^0,g^1\}\). In the basis used in~\cite{yuan2006new}, the projected matrix has the form
\begin{equation}
\label{eq:yuan-projected-matrix}
    T = Q^\top AQ =
    \begin{pmatrix}
        \bigl(\alpha_0^{\mathrm{SD}}\bigr)^{-1} & -\|g^1\|_2/\|s^0\|_2 \\[4pt]
        -\|g^1\|_2/\|s^0\|_2 & \bigl(\alpha_1^{\mathrm{SD}}\bigr)^{-1}
    \end{pmatrix},
    \qquad s^0 = x^1-x^0.
\end{equation}
The reciprocal Yuan step-size is the larger eigenvalue of this two-dimensional projected matrix. Thus, if \(\theta_1\leq\theta_2\) are the Ritz values of \(T\), then \(\alpha^{\mathrm{Y}}=\theta_2^{-1}\). For a two-dimensional strictly convex quadratic, the Yuan construction combined with the prescribed steepest-descent steps yields finite termination in exact arithmetic; see~\cite{yuan2006new}.

\medskip
\noindent\textbf{NY spectral quantities for \(m=3\).}
The three-dimensional construction of~\cite{xie2026new} uses the projected matrix on \(\operatorname{span}\{g^{k-2},g^{k-1},g^k\}\). After two steepest-descent steps, \((g^{k-2})^\top g^{k-1}=0\) and \((g^{k-1})^\top g^k=0\). Applying Gram--Schmidt gives an orthonormal basis \(Q\) and the three-dimensional projected matrix \(T=Q^\top AQ\). Its eigenvalues are the three Ritz values of \(A\) on the current gradient-history subspace. Following the descending convention of~\cite{xie2026new}, write \(\mu_1\geq\mu_2\geq\mu_3>0\). They are the roots of \(\mu^3-k_1\mu^2+k_2\mu-k_3=0\), where the coefficients are determined by the preceding steepest-descent step-sizes and the projected quantities specified in~\cite[eq.~(2.11)]{xie2026new}. The corresponding reciprocal step-sizes are \(\alpha^{\mathrm{NY}(i)}=\mu_i^{-1}\), \(i=1,2,3\), and satisfy \(\alpha^{\mathrm{NY}(1)}\leq\alpha^{\mathrm{NY}(2)}\leq\alpha^{\mathrm{NY}(3)}\).

Under the ascending Ritz-value convention used in the main text, \(\theta_1\leq\theta_2\leq\theta_3\), the two orderings are related by \(\theta_1=\mu_3\), \(\theta_2=\mu_2\), \(\theta_3=\mu_1\). Consequently, \(\theta_3^{-1}=\alpha^{\mathrm{NY}(1)}\) is the shortest of the three NY step-sizes, whereas \(\theta_1^{-1}=\alpha^{\mathrm{NY}(3)}\) is the longest. This correspondence should be distinguished from the algorithmic rule that determines which of the three spectral step-sizes is reused within an NY cycle.

Both the Yuan and NY constructions can therefore be interpreted through low-dimensional Krylov--Ritz projections. The former uses the larger Ritz value of a two-dimensional projected problem, while the latter uses the three Ritz values of a three-dimensional projected problem.

\subsection{Proofs for Section~\ref{subsec:gudu-compatibility-rigidity}}

\begin{proof}[Proof of Proposition~\ref{prop:gudu-node-compatibility-symbol}]
The exact atomic moment representation gives
$M_0 = \sum_{i=1}^m t_iP(a_i)$ and $M_1 = \sum_{i=1}^m t_ia_iP(a_i)$.
Therefore, $M_1-\theta M_0 = \sum_{i=1}^m t_i(a_i-\theta)P(a_i)$. Setting \(\theta=a_s\) yields
$M_1-a_sM_0 = \sum_{\substack{i=1\\i\neq s}}^m t_i(a_i-a_s)P(a_i)$.
Thus \(a_s\) is a generalized eigenvalue if and only if the matrix on the right-hand side is singular, which is equivalent to \eqref{eq:node-compatibility-condition}.
\end{proof}

\begin{proof}[Proof of Proposition~\ref{prop:gudu-gen-equals-nar}]
Write \(P(x)=\sum_{j=0}^{2m-1}P^{(j)}x^j\). Define the linear map
$\mathcal T(y_0,\ldots,y_{2m-1}) := \sum_{j=0}^{2m-1}P^{(j)}y_j$.
It follows that
$M_0 = \Lambda_h(P) = \sum_{j=0}^{2m-1}P^{(j)}h_j = \mathcal T(h_0,\ldots,h_{2m-1})$,
and
$M_1 = \Lambda_h(xP) = \sum_{j=0}^{2m-1}P^{(j)}h_{j+1} = \mathcal T(h_1,\ldots,h_{2m})$.
Hence the representation satisfies \eqref{eq:narrow-realization-window}.
\end{proof}

\begin{proof}[Proof of Proposition~\ref{prop:psd-universal-recovery-rank-one}]
Fix pairwise distinct points \(x_1,\ldots,x_{m-1}\in\Omega\) and positive weights \(t_1,\ldots,t_{m-1}\). For every \(\lambda\in\Omega\setminus\{x_1,\ldots,x_{m-1}\}\), apply the exact-recovery assumption to the \(m\) nodes \(x_1,\ldots,x_{m-1},\lambda\). The recovery condition at the node \(\lambda\) gives
$\det(\sum_{i=1}^{m-1} t_i(x_i-\lambda)P(x_i)) = 0$.
The left-hand side is a polynomial in \(\lambda\) and vanishes on an infinite subset of \(\Omega\). It is therefore identically zero. Its leading coefficient gives
$\det(\sum_{i=1}^{m-1}t_iP(x_i)) = 0$.
Now fix pairwise distinct points \(x_1,\ldots,x_m\in\Omega\) and define
$Q(t_1,\ldots,t_m) = \det(\sum_{i=1}^m t_iP(x_i))$.
This is a homogeneous polynomial of degree \(m\). By the above identity, it vanishes whenever any one of the variables \(t_j\) is zero. Hence each variable divides \(Q\), and homogeneity gives
$Q(t_1,\ldots,t_m) = c\,t_1\cdots t_m$.
For positive weights, the matrix \(\sum_{i=1}^m t_iP(x_i)\) is positive semidefinite and nonsingular, hence positive definite. Therefore, \(c>0\).

Fix an arbitrary point \(x_0\in\Omega\) and choose pairwise distinct points \(x_1,\ldots,x_{m-1}\in\Omega\setminus\{x_0\}\). Applying the above factorization to the points \(x_0,x_1,\ldots,x_{m-1}\) gives
$\det(tP(x_0) + \sum_{i=1}^{m-1}P(x_i)) = ct$, \(t>0\).
Let \(B = P(x_0) + \sum_{i=1}^{m-1}P(x_i)\). By the nonsingularity assumption and positive semidefiniteness, \(B\succ0\). Setting \(t=1+s\) gives the polynomial identity
$\det(B+sP(x_0)) = c(1+s)$.
On the other hand,
\[
    \det\!\left( B+sP(x_0) \right) = \det(B) \det\!\left( I+sB^{-1/2}P(x_0)B^{-1/2} \right) = \det(B) \prod_{j=1}^{r} (1+s\lambda_j),
\]
where \(r = \operatorname{rank}P(x_0)\) and \(\lambda_1,\ldots,\lambda_r\) are the positive eigenvalues of \(B^{-1/2}P(x_0)B^{-1/2}\). The degree of this polynomial in \(s\) is \(r\). Since the right-hand side \(c(1+s)\) has degree one, \(\operatorname{rank}P(x_0)=1\). Because \(x_0\in\Omega\) was arbitrary, the result follows.
\end{proof}

\begin{proof}[Proof of Corollary~\ref{cor:positive-triple-equivalence}]
The positivity, regularity, and universal node-recovery assumptions allow Proposition~\ref{prop:psd-universal-recovery-rank-one} to be applied to the effective symbol. Hence \(\operatorname{rank}P(x)=1\) for \(x\in\Omega\). The degree condition \(\deg P\leq2m-1\) and Proposition~\ref{prop:gudu-gen-equals-nar} imply that the same representation is narrow. Finally, the assumption \(G\in\mathcal{GD}_m\) supplies positive nodes and weights satisfying the selected moment equations. The four stated properties therefore hold for the same effective representation.
\end{proof}

\begin{proof}[Proof of Proposition~\ref{prop:psd-nar-genone-in-gudu}]
For the exact atomic moment sequence,
$M_0 = \Lambda_h(P) = \sum_{i=1}^m t_iP(a_i)$ and $M_1 = \Lambda_h(xP) = \sum_{i=1}^m t_ia_iP(a_i)$.
For each \(s=1,\ldots,m\),
$M_1-a_sM_0 = \sum_{\substack{i=1\\i\neq s}}^m t_i(a_i-a_s)P(a_i)$.
Since \(P(a_i)=q(a_i)u(a_i)u(a_i)^\top\) has rank at most one, the right-hand side is a sum of at most \(m-1\) rank-one matrices. Therefore, \(\operatorname{rank}(M_1-a_sM_0)\leq m-1\), and \(\det(M_1-a_sM_0)=0\). Thus every \(a_s\) is a generalized eigenvalue of the pencil. Because the pencil is regular and the nodes are pairwise distinct, these \(m\) nodes exhaust its finite generalized eigenvalues.

The same nodes and weights satisfy \(h_j=\sum_{i=1}^m t_ia_i^j\), \(j=0,\ldots,2m-1\), which are precisely the standard long Gu--Du moment equations. Therefore, on the specified positive atomic moment class, the generalized eigenvalues of the rank-one pencil coincide with the nodes recovered by the standard Gu--Du construction.
\end{proof}

\section{Proofs for Section~\ref{sec:hanoi-principle}}
\label{app:hanoi-proofs}

\subsection{Proof of Proposition~\ref{prop:hanoi-principle}}
\label{app:StandardHanoi}

\begin{proof}
We proceed by induction on \(k\). For \(k=1\), the component \(d_n\) is initially unresolved by Assumption~1. Since each phase can settle only one targeted component, at least one phase is necessary. One phase targeting \(\lambda_n\) clears \(d_n\) by Assumption~3, so \(N_1=1\).

Assume that \(N_{k-1}=2^{k-1}-1\). Consider the task of clearing the \(k\) highest-frequency components \(d_{n-k+1},\ldots,d_n\). By Assumption~4, before \(d_{n-k+1}\) becomes accessible, the higher-frequency components \(d_{n-k+2},\ldots,d_n\) must be settled. By the induction hypothesis, this requires at least \(N_{k-1}\) phases.

One additional phase is then necessary to target and settle \(d_{n-k+1}\). By Assumption~3, this phase raises \(|d_j|>\delta_j\) for all \(j>n-k+1\), so the higher-frequency components \(d_{n-k+2},\ldots,d_n\) become unresolved again. Settling them again requires at least another \(N_{k-1}\) phases. Hence \(N_k\geq2N_{k-1}+1\).

Conversely, this bound is attainable: first clear \(d_{n-k+2},\ldots,d_n\) using an optimal \((k-1)\)-component schedule, then target \(\lambda_{n-k+1}\) once, and finally clear \(d_{n-k+2},\ldots,d_n\) again using the same optimal \((k-1)\)-component schedule. Therefore \(N_k=2N_{k-1}+1\). Together with \(N_1=1\), this gives \(N_k=2^k-1\).
\end{proof}

\subsection{Proof of Corollary~\ref{cor:block-hanoi}}
\label{app:BlockHanoi}

\begin{proof}
Treat each block as one aggregate component in Proposition~\ref{prop:hanoi-principle}. The number of block phases therefore satisfies \(C_k=2C_{k-1}+1\), \(C_1=1\), and hence \(C_K=2^K-1\).

Index the blocks from the lowest-frequency block to the highest-frequency block. Under the complete-rebound schedule, the lowest-frequency block is targeted once and block \(j\) is targeted \(2^{j-1}\) times. The lowest-frequency block contains \(r\) components, while every other block contains \(m\) components. Therefore
$T_n = r+m\sum_{j=2}^{K}2^{j-1} = r+m(2^K-2)$.
This proves \eqref{eq:block-hanoi-step-count}.
\end{proof}

\subsection{Proof of Proposition~\ref{prop:practical-hanoi}}
\label{app:practicalHanoi}

\begin{proof}
For \(0\leq k\leq m\), all \(k\) components of the subproblem are simultaneously represented in the memory-\(m\) reduced model. Since each phase targets one selected component, the \(k\) components can be settled in \(k\) phases. Hence \(f_m(k)=k\) for \(0\leq k\leq m\).

Let \(n>m\). By the recursive transition assumptions in Proposition~\ref{prop:practical-hanoi}, exposing the lowest-frequency component of an \(n\)-component subproblem requires an optimal preparatory schedule for an \((n-m)\)-component subproblem. This stage requires \(f_m(n-m)\) phases. One additional phase then targets and settles the newly accessible lowest-frequency component.

After this phase, the remaining unresolved components form an \((n-1)\)-component subproblem governed by the same transition rules. Settling this continuation subproblem requires \(f_m(n-1)\) further phases. Consequently,
$f_m(n) = f_m(n-m)+1+f_m(n-1)$, \(n>m\).
For \(m=1\), the recurrence becomes \(f_1(n)=2f_1(n-1)+1\), \(f_1(1)=1\). Therefore \(f_1(n)=2^n-1\), and \(\rho_1=2\).

Now let \(m\geq2\). Define \(g_m(n)=f_m(n)+1\). Then the recurrence becomes \(g_m(n)=g_m(n-1)+g_m(n-m)\). The associated characteristic polynomial is \(p_m(x)=x^m-x^{m-1}-1\). Since \(p_m(1)=-1<0\), \(p_m(2)=2^{m-1}-1>0\), and \(p_m'(x)=x^{m-2}(mx-(m-1))>0\) for \(x>1\), the polynomial \(p_m\) has a unique positive root \(\rho_m\in(1,2)\). Equivalently, \(\rho_m^m=\rho_m^{m-1}+1\).

The recurrence for \(g_m(n)\) has nonnegative coefficients, and its companion matrix is primitive. By the Perron--Frobenius theorem, \(\rho_m\) is its simple dominant eigenvalue. Since the initial values are positive, \(g_m(n)=\Theta(\rho_m^n)\). It follows that \(f_m(n)=\Theta(\rho_m^n)\).

Finally, the defining equation for \(\rho_m\) is equivalent to \(\rho_m^{m-1}(\rho_m-1)=1\). For fixed \(\rho>1\), the left-hand side is strictly increasing in \(m\), whereas for fixed \(m\) it is strictly increasing in \(\rho\). Therefore \(\rho_{m+1}<\rho_m\). If \(\rho_m\) converged to a limit strictly larger than \(1\), then \(\rho_m^{m-1}(\rho_m-1)\) would diverge as \(m\to\infty\), contradicting its equality to \(1\). Hence \(\rho_m\downarrow1\).
\end{proof}

\section{Proofs for Section~\ref{sec:algorithm}}
\label{app:algorithm-proofs}
\subsection{Energy--residual bound}
\label{app:energy-residual}

\begin{proposition}
\label{prop:energy-residual}
Let \(\mathcal{V}^k = \mathrm{span}(G^k)\) be the subspace spanned by the gradient block \(G^k = [g^{k-m},\ldots,g^{k-1}]\), let \((\theta_i^k, v_i^k)\) be a Ritz pair extracted from \(\mathcal{V}^k\) via the Rayleigh--Ritz procedure, and let \(\mathrm{ind} \in \{1,\ldots,m\}\). Assume that the \(m\)-step Lanczos process started at
\(g^{k-m+\mathrm{ind}-1}\) generates the same subspace
\(\mathcal V^k\). Run the Lanczos process with starting vector
\(g^{k-m+\mathrm{ind}-1}\), producing an orthonormal basis
\(\widetilde V_m^k\) with
   $ \widetilde v_1^k
    =
    \frac{g^{k-m+\mathrm{ind}-1}}
    {\|g^{k-m+\mathrm{ind}-1}\|_2},$
and let \(\beta_{m,\mathrm{ind}}^k\) be the last subdiagonal
coefficient of the resulting tridiagonal matrix
\(\widetilde T_m^k\). Then
\begin{equation}
\label{eq:energy-residual-bound}
    \|r_i^k\|_2^2 \leq \bigl(\beta_{m,\,\mathrm{ind}}^k\bigr)^2
    \!\left(1 - \frac{t_i^{k,\,\mathrm{ind}}} {\|g^{k-m+\mathrm{ind}-1}\|_2^2}\right),
\end{equation}
where \(t_i^{k,\,\mathrm{ind}}\) is defined in~\eqref{eq:component-energy-general}.
\end{proposition}

\begin{proof}
Under the stated assumption, \(\tilde{V}_m^k\) is an orthonormal basis for \(\mathcal{V}^k\). The Ritz pair \((\theta_i^k, v_i^k)\) is independent of the choice of starting vector; write \(v_i^k = \tilde{V}_m^k s_i^k\) with \(\|s_i^k\|_2 = 1\). The Lanczos three-term recurrence reads
$A\tilde{V}_m^k = \tilde{V}_m^k \tilde{T}_m^k + \beta_{m,\,\mathrm{ind}}^k\,\tilde{v}_{m+1}^k e_m^\top$,
and the eigenpair condition is \(\tilde{T}_m^k s_i^k = \theta_i^k s_i^k\). Multiplying the recurrence by \(s_i^k\) and using the eigenpair condition gives
$r_i^k = A v_i^k - \theta_i^k v_i^k = \beta_{m,\,\mathrm{ind}}^k\,(e_m^\top s_i^k)\,\tilde{v}_{m+1}^k$.
Since \(\tilde{v}_{m+1}^k\) is a unit vector orthogonal to \(\tilde{V}_m^k\), taking squared norms yields
$\|r_i^k\|_2^2 = \bigl(\beta_{m,\,\mathrm{ind}}^k\bigr)^2 (e_m^\top s_i^k)^2$.
It remains to relate \((e_m^\top s_i^k)^2\) to the component weight. Because \(\tilde{v}_1^k = g^{k-m+\mathrm{ind}-1}/\|g^{k-m+\mathrm{ind}-1}\|_2\), we have
$(e_1^\top s_i^k)^2 = t_i^{k,\,\mathrm{ind}} / \|g^{k-m+\mathrm{ind}-1}\|_2^2$,
where the equality follows from the definition of \(t_i^{k,\,\mathrm{ind}}\) in~\eqref{eq:component-energy-general}. Since \(\|s_i^k\|_2 = 1\), the Pythagorean identity gives
$(e_m^\top s_i^k)^2 \leq 1 - (e_1^\top s_i^k)^2 = 1 - t_i^{k,\,\mathrm{ind}} / \|g^{k-m+\mathrm{ind}-1}\|_2^2$.
Substituting this bound into the expression for \(\|r_i^k\|_2^2\) proves~\eqref{eq:energy-residual-bound}.
\end{proof}

\subsection{First-order perturbation estimator for \(\Delta\theta_i^k\)}
\label{app:perturbation-estimator}

This appendix derives the first-order perturbation estimator mentioned in Section~\ref{subsec:delta-theta} and explains why it is retained as a theoretical reference only.

Let \(h_\ell = (g^{k-m})^\top A^\ell g^{k-m}\), \(\ell = 0, 1, \ldots, 2m-1\), be the moment sequence generated by the baseline gradient \(g^{k-m}\). In the ideal dominant spectral model these moments satisfy
\begin{equation}
\label{eq:ideal-moment-model}
    h_\ell = \sum_{i=1}^m (\theta_i^k)^\ell\, t_i^k.
\end{equation}
Let \(\pi^{k-m}\) index the true eigenvalues approximated by the current Ritz values, and write \(\theta_i^k + \Delta\theta_i^k = \lambda_{\pi_i^{k-m}}\) and \(t_i^k + \Delta t_i^k = (g^{k-m}_{\pi_i^{k-m}})^2\). The corresponding ideal dominant moments satisfy
$h_\ell + \Delta h_\ell = \sum_{i=1}^m (\theta_i^k + \Delta\theta_i^k)^\ell (t_i^k + \Delta t_i^k)$.
Expanding to first order and neglecting higher-order terms gives
\begin{equation}
\label{eq:perturbation-firstorder}
    \sum_{i=1}^m \Bigl[ (\theta_i^k)^\ell \Delta t_i^k + \ell(\theta_i^k)^{\ell-1} t_i^k \Delta\theta_i^k \Bigr] = \Delta h_\ell,
    \qquad \ell = 0, 1, \ldots, 2m-1.
\end{equation}
Collecting the \(2m\) equations in \eqref{eq:perturbation-firstorder} yields the linear system
\begin{equation}
\label{eq:perturbation-system}
    M\,\xi = \Delta h,
    \qquad
    \xi = \bigl( t_1^k \Delta\theta_1^k, \ldots, t_m^k \Delta\theta_m^k,\, \Delta t_1^k, \ldots, \Delta t_m^k \bigr)^\top,
\end{equation}
where \(\Delta h = (\Delta h_0, \ldots, \Delta h_{2m-1})^\top\) and the coefficient matrix is
\begin{equation}
\label{eq:vandermonde-matrix}
    M =
    \begin{pmatrix}
    0 & \cdots & 0 & 1 & \cdots & 1 \\
    1 & \cdots & 1 & \theta_1^k & \cdots & \theta_m^k \\
    2\theta_1^k & \cdots & 2\theta_m^k & (\theta_1^k)^2 & \cdots & (\theta_m^k)^2 \\
    \vdots & & \vdots & \vdots & & \vdots \\
    (2m{-}1)(\theta_1^k)^{2m-2} & \cdots & (2m{-}1)(\theta_m^k)^{2m-2} & (\theta_1^k)^{2m-1} & \cdots & (\theta_m^k)^{2m-1}
    \end{pmatrix}.
\end{equation}
This is a confluent Vandermonde matrix; it is invertible whenever the Ritz values \(\theta_1^k, \ldots, \theta_m^k\) are distinct.

The moment perturbations \(\Delta h_\ell\) are approximated by comparing the full moment sequence with a projected one, \(\Delta h_\ell \approx \widetilde{h}_\ell - h_\ell\), where \(\widetilde{h}_\ell\) is computed from the projected gradients \(g^{j\prime} = Q_k^\top g^j\). The same lower-triangular recovery of Section~\ref{app:pseudo-memory} applies in the projected space, so no additional matrix-vector products with \(A\) are required.

Although theoretically well-motivated, direct inversion of \eqref{eq:perturbation-system} is numerically unreliable in practice. Confluent Vandermonde matrices are notoriously ill-conditioned: when the Ritz values \(\theta_1^k, \ldots, \theta_m^k\) cluster near convergence, the condition number of \(M\) grows exponentially in \(m\), making the solution of \eqref{eq:perturbation-system} numerically meaningless. The estimator is therefore retained only as a theoretical reference that explains how moment perturbations propagate to Ritz errors; the Kato--Temple estimator~\cite{parlett1998symmetric} and the long/short LMSD estimator of Section~\ref{subsec:delta-theta} are used in the experiments instead.

\subsection{Proof of Proposition~\ref{prop:long-short-residual}}
\label{app:long-short-proof}

\begin{proof}
For readability, write \(\theta_s=\theta_{i,\mathrm{short}}^k\), \(\theta_l=\theta_{i,\mathrm{long}}^k\), \(v_s=v_{i,\mathrm{short}}^k\), \(v_l=v_{i,\mathrm{long}}^k\). By \eqref{eq:long-short-ordering}, \(\theta_s\geq\theta_l\). After replacing \(v_l\) by \(-v_l\) if necessary, we may assume \(\delta_{v,i}=\|v_s-v_l\|_2\). Write \(v_s=v_l+e\), where \(\|e\|_2=\delta_{v,i}\).

For the short Ritz pair, the reciprocal Ritz value satisfies \(\theta_s = v_s^\top A^2v_s / v_s^\top Av_s\). Hence \(v_s^\top A^2v_s = \theta_s v_s^\top Av_s\). Therefore the short residual satisfies
\[
    \|r_{i,\mathrm{short}}^k\|_2^2 = \|Av_s-\theta_s v_s\|_2^2 = v_s^\top A^2v_s - 2\theta_s v_s^\top Av_s + \theta_s^2 = \theta_s(\theta_s-v_s^\top Av_s).
\]
Since \(0<\theta_s\leq\lambda_n\), we get \(\|r_{i,\mathrm{short}}^k\|_2^2 \leq \lambda_n(\theta_s-v_s^\top Av_s)\).

Next,
$v_s^\top Av_s = (v_l+e)^\top A(v_l+e) = v_l^\top Av_l + 2v_l^\top Ae + e^\top Ae$.
Since \(v_l^\top Av_l=\theta_l\), define \(\delta_A = 2v_l^\top Ae+e^\top Ae\). Then \(v_s^\top Av_s=\theta_l+\delta_A\). Moreover,
$|\delta_A| \leq \|A\|_2\|e\|_2(2+\|e\|_2) = \|A\|_2\delta_{v,i}(2+\delta_{v,i})$.
Thus
$\theta_s-v_s^\top Av_s = \theta_s-\theta_l-\delta_A \leq \theta_s-\theta_l+|\delta_A| \leq \theta_s-\theta_l + \|A\|_2\delta_{v,i}(2+\delta_{v,i})$.
Combining the preceding inequalities yields
$\|r_{i,\mathrm{short}}^k\|_2^2 \leq \lambda_n ( \theta_s-\theta_l + \|A\|_2\delta_{v,i}(2+\delta_{v,i}) )$,
which is \eqref{eq:long-short-bound}.
\end{proof}

\section{Proofs for Section~\ref{sec:convergence_analysis}}
\label{app:convergence-proofs}

\begin{proof}[Proof of Lemma~\ref{lem:single_step_bound}]
By construction of Algorithm~\ref{alg:adaptive_hanoi_main}, the reciprocal step-size \(\theta_k\) is a Rayleigh quotient of the symmetric positive definite matrix \(A\). It is either the initial Cauchy value \(\alpha_0^{-1}=(g^0)^\top Ag^0/\|g^0\|_2^2\), a Ritz value of the projected matrix \(T_k=Q_k^\top AQ_k\) from~\eqref{eq:projected-matrix}, a Ritz value reused within a continuation phase, or the anchor-based BB1 fallback value \(\theta_k^{\mathrm{BB1}}\) from~\eqref{eq:rank-deficient-fallback}. In every case, the min--max characterization gives \(\lambda_1 \leq \theta_k \leq \lambda_n\) for all \(k\geq0\).

Since \(\lambda_1\leq\lambda_i\leq\lambda_n\), it follows that
$1-\lambda_n/\lambda_1 \leq 1-\lambda_i/\theta_k \leq 1-\lambda_1/\lambda_n$.
By the definition of \(\delta\) in~\eqref{eq:sc_def}, \(|1-\lambda_i/\theta_k|\leq\delta\). Substituting this estimate into \eqref{eq:eigen_component_update} gives \((d_i^{k+1})^2 \leq \delta^2(d_i^k)^2\). Summing over \(i=1,\ldots,n\) and taking square roots yields \(\|g^{k+1}\|_2 \leq \delta\|g^k\|_2\).
\end{proof}

\begin{proof}[Proof of Lemma~\ref{lem:sweeping_l_plus_1}]
Assume, to the contrary, that
\begin{equation}
\label{eq:contrary_assump}
    (d_{l+1}^{k+j})^2 > s\,\varepsilon_l\,\|g^k\|_2^2,
    \qquad \forall\,j \in [m_l,\,m_l+\Delta_l+1].
\end{equation}
Repeated application of Lemma~\ref{lem:single_step_bound} from iteration \(k\) gives the preliminary estimate
\begin{equation}
\label{eq:crude_upper}
    (d_{l+1}^{k+m_l+\tilde{m}})^2 \leq \delta^{2(m_l+\tilde{m})}\|g^k\|_2^2.
\end{equation}
Fix \(j \in [m_l,\,m_l+\Delta_l-\tilde{m}]\) and set \(t=k+j+\tilde{m}\). Let \(K\leq t\) be the refresh index at which the reciprocal step-size \(\theta_t\) was computed. By the continuation rule, \(t-K \leq \kappa_{\mathrm{cont}}-1\). Let \(p:=K-m\) be the anchor index entering the admissibility threshold at refresh iteration \(K\). Since \(K \geq t-\kappa_{\mathrm{cont}}+1\), we obtain
$p \geq K-m \geq t-\kappa_{\mathrm{cont}}+1-m = t-\tilde{m} = k+j \geq k+m_l$.
Moreover, \(p \leq K \leq t = k+j+\tilde{m} \leq k+m_l+\Delta_l\). Consequently, \(p-k \in [m_l,\,m_l+\Delta_l]\).

The induction hypothesis gives \(D(p,l) \leq \varepsilon_l\|g^k\|_2^2\), whereas the contradictory assumption \eqref{eq:contrary_assump} gives \((d_{l+1}^{p})^2 > s\,\varepsilon_l\|g^k\|_2^2\).

The admissibility threshold \eqref{eq:admissibility-threshold} evaluated at the refresh index \(K\) and anchor \(p\) is
$\omega_K = \frac{1+\gamma}{2} \frac{ \sum_{i=1}^n\lambda_i(d_i^p)^2 }{ \sum_{i=1}^n(d_i^p)^2 }$.
Let \(H_p := \sum_{i=l+1}^n(d_i^p)^2\). Since \(H_p \geq (d_{l+1}^p)^2 > s\,\varepsilon_l\|g^k\|_2^2\) and \(D(p,l) \leq \varepsilon_l\|g^k\|_2^2\), we have
$\omega_K \geq \frac{1+\gamma}{2} \lambda_{l+1} \frac{H_p}{D(p,l)+H_p} \geq \frac{1+\gamma}{2} \lambda_{l+1} \frac{s}{s+1}$.
Because the step-size selector enforces \(\theta_t\geq\omega_K\), it follows that
$1-\frac{\lambda_{l+1}}{\theta_t} \geq 1- \frac{2(s+1)}{(1+\gamma)s} = -\frac{1}{1+\gamma}$,
where the last equality uses \(s=2/\gamma\).

On the other hand, the Ritz-value bound \(\theta_t\leq\lambda_n\) gives \(1-\lambda_{l+1}/\theta_t \leq 1-\lambda_1/\lambda_n\). The definition of \(c\) in~\eqref{eq:sc_def} therefore yields
$|1-\lambda_{l+1}/\theta_t| \leq c$.
By~\eqref{eq:eigen_component_update}, \((d_{l+1}^{t+1})^2 \leq c^2(d_{l+1}^{t})^2\). Applying this contraction for \(t = k+m_l+\tilde{m}, \ldots, k+m_l+\Delta_l\) and using~\eqref{eq:crude_upper} gives
$ (d_{l+1}^{k+m_l+\Delta_l+1})^2 \leq c^{2(\Delta_l-\tilde{m}+1)} (d_{l+1}^{k+m_l+\tilde{m}})^2 \leq c^{2(\Delta_l-\tilde{m}+1)} \delta^{2(m_l+\tilde{m})} \|g^k\|_2^2 $.
By the definition of \(\Delta_l\) in \eqref{eq:Delta_l_def}, \(c^{2(\Delta_l-\tilde{m}+1)} \delta^{2(m_l+\tilde{m})} \leq s\,\varepsilon_l\). Hence
$(d_{l+1}^{k+m_l+\Delta_l+1})^2 \leq s\,\varepsilon_l\|g^k\|_2^2$,
contradicting~\eqref{eq:contrary_assump} at \(j=m_l+\Delta_l+1\). Therefore, there exists \(j_0 \in [m_l,\,m_l+\Delta_l+1]\) such that \((d_{l+1}^{k+j_0})^2 \leq s\,\varepsilon_l\|g^k\|_2^2\).
\end{proof}

\begin{proof}[Proof of Lemma~\ref{lem:induction_l_plus_1}]
By Lemma~\ref{lem:sweeping_l_plus_1}, there exists \(r_0 \in [m_l,\,m_l+\Delta_l+1]\) such that \((d_{l+1}^{k+r_0})^2 \leq s\,\varepsilon_l\|g^k\|_2^2\). Fix any \(j \geq m_{l+1} = m_l+\Delta_l+1\), and let \(r\) be the largest index in \([m_l,j]\) satisfying
\begin{equation}
\label{eq:last-small-component-index}
    (d_{l+1}^{k+r})^2 \leq s\,\varepsilon_l\|g^k\|_2^2.
\end{equation}
Such an index exists because \(r_0\leq j\).

\medskip
\noindent\textit{Case 1: \(j-r\leq\tilde{m}+1\).}
Applying Lemma~\ref{lem:single_step_bound} \(j-r\) times gives
$(d_{l+1}^{k+j})^2 \leq \delta^{2(j-r)} (d_{l+1}^{k+r})^2 \leq s\,\delta^{2(\tilde{m}+1)} \varepsilon_l\|g^k\|_2^2$.

\medskip
\noindent\textit{Case 2: \(j-r>\tilde{m}+1\).}
Consider any \(t \in [k+r+\tilde{m}+1,\,k+j-1]\). Let \(K\leq t\) be the refresh index at which the reciprocal step-size \(\theta_t\) was computed. Then \(t-K \leq \kappa_{\mathrm{cont}}-1\). Set \(p:=K-m\). As in the proof of Lemma~\ref{lem:sweeping_l_plus_1},
$p \geq K-m \geq t-\kappa_{\mathrm{cont}}+1-m = t-\tilde{m} \geq k+r+1$.
Furthermore, \(p \leq K \leq t \leq k+j-1\). Thus \(p-k \in [r+1,\,j-1]\).

By the maximality of \(r\) in \eqref{eq:last-small-component-index}, \((d_{l+1}^{p})^2 > s\,\varepsilon_l\|g^k\|_2^2\). The induction hypothesis gives \(D(p,l) \leq \varepsilon_l\|g^k\|_2^2\). Repeating the admissibility argument from the proof of Lemma~\ref{lem:sweeping_l_plus_1} yields \(|1-\lambda_{l+1}/\theta_t|\leq c\).

The first \(\tilde{m}+1\) updates after iteration \(k+r\) are bounded by Lemma~\ref{lem:single_step_bound}; all subsequent updates through iteration \(k+j\) satisfy the contraction factor \(c\). Consequently,
$ (d_{l+1}^{k+j})^2 \leq c^{2(j-r-\tilde{m}-1)} \delta^{2(\tilde{m}+1)} (d_{l+1}^{k+r})^2 \leq s\, c^{2(j-r-\tilde{m}-1)} \delta^{2(\tilde{m}+1)} \varepsilon_l\|g^k\|_2^2 \leq s\,\delta^{2(\tilde{m}+1)} \varepsilon_l\|g^k\|_2^2 $.
Both cases therefore give \((d_{l+1}^{k+j})^2 \leq s\,\delta^{2(\tilde{m}+1)} \varepsilon_l\|g^k\|_2^2\). Combining this estimate with the induction hypothesis yields
$D(k+j,l+1) = D(k+j,l) + (d_{l+1}^{k+j})^2 \leq ( 1+s\,\delta^{2(\tilde{m}+1)} ) \varepsilon_l\|g^k\|_2^2 = \varepsilon_{l+1}\|g^k\|_2^2$.
This holds for every \(j\geq m_{l+1}\).
\end{proof}

\begin{proof}[Proof of Lemma~\ref{lem:global_halving}]
Define \(a := 1+s\,\delta^{2(\tilde{m}+1)}\). Choose
$\varepsilon_0 \in ( 0,\, \min\{ \gamma/(4a^n),\; 1/(2a^n) \} )$.
For \(l=0,\ldots,n\), set \(\varepsilon_l := a^l\varepsilon_0\). Then \(0<\varepsilon_l\leq\gamma/4\) for \(l=0,\ldots,n-1\), and \(\varepsilon_n \leq 1/2\).

Since \(D(k+j,0)=0\), the induction hypothesis at level \(l=0\) holds for every \(j\geq0\). Set \(m_0:=0\). Applying Lemma~\ref{lem:induction_l_plus_1} successively for \(l=0,1,\ldots,n-1\) produces integers \(m_1,\ldots,m_n\) such that
$D(k+j,l) \leq \varepsilon_l\|g^k\|_2^2$, \(\forall\,j\geq m_l\), \(l=1,\ldots,n\).
At level \(l=n\),
$\|g^{k+j}\|_2^2 = D(k+j,n) \leq \varepsilon_n\|g^k\|_2^2 \leq \frac12\|g^k\|_2^2$, \(\forall\,j\geq m_n\).
Taking \(M:=m_n\) proves the result. The integer \(M\) depends only on \(n\), \(\delta\), \(\gamma\), \(c\), and \(\tilde{m}\).
\end{proof}

\begin{proof}[Proof of Proposition~\ref{prop:r_linear_convergence}]
Lemma~\ref{lem:global_halving} gives \(\|g^{k+M}\|_2^2 \leq \frac12\|g^k\|_2^2\) for all \(k\geq m\). Applying this estimate recursively yields
$\|g^{m+qM}\|_2 \leq 2^{-q/2}\|g^m\|_2$, \(q=0,1,\ldots\).
For indices between consecutive block endpoints, Lemma~\ref{lem:single_step_bound} gives a uniform finite-step amplification bound. Hence the standard blocking argument of Dai et al.~\cite{dai2002r} yields constants \(C_g>0\) and \(\rho := 2^{-1/(2M)} \in (0,1)\) such that \(\|g^k\|_2 \leq C_g\rho^k\) for all \(k\geq0\).

Since \(g^k=A(x^k-x^*)\) and \(A\succ0\),
$\lambda_1\|x^k-x^*\|_2 \leq \|g^k\|_2 \leq \lambda_n\|x^k-x^*\|_2$.
Therefore \(\|x^k-x^*\|_2 \leq (C_g/\lambda_1)\rho^k\). Setting \(C:=C_g/\lambda_1\) completes the proof.
\end{proof}

\noindent\textbf{Acknowledgements}
This work was supported by the National Key R\&D Program of China grant No. 2022YFA1003800, the National Natural Science Foundation of China No. 12671364 and the Natural Science Foundation of Tianjin No. 25JCJQJC00300.

\medskip
\noindent\textbf{Declarations.}\quad
The authors declare that they have no conflict of interest.

\medskip
\noindent\textbf{Availability of data and materials.}\quad
All data and codes analyzed during this study are publicly available in the GitHub repository: \url{https://github.com/yijiazhoucollab/Hanoi}.

\bibliography{refoptimization}
\end{document}